\documentclass[leqno, dvipsnames]{amsart}
\usepackage[a4paper]{geometry}
\usepackage{mathtools,amscd,url}
\usepackage{amssymb}
\usepackage{latexsym}
\usepackage{amsthm}
\usepackage{amsmath}
\usepackage[utf8]{inputenc}
\usepackage{csquotes}
\usepackage{latexsym}
\usepackage{tikz}
\usetikzlibrary{decorations.pathreplacing,arrows,decorations.pathmorphing,positioning}
\usepackage[ruled, vlined]{algorithm2e}
\usepackage{graphicx}
\usepackage{ytableau}
\usepackage{multirow}
\usepackage[foot]{amsaddr}

\makeatletter
\def\namedlabel#1#2{\begingroup
   \def\@currentlabel{#2}%
   \label{#1}\endgroup
}
\makeatother

\newcommand{\SOodd}{\mathrm{SO}(2k+1)}

\newcommand{\SOdrei}{\mathrm{SO}(3)}
\newcommand{\LRcoef}{\mathrm{c}_{\lambda}^{\mu}(\mathfrak{d})}
\newcommand{\LRtabs}{\mathrm{LR}_\lambda^{\mu}(\mathfrak{d})}
\newcommand{\LRtabsa}{\mathrm{aLR}_\lambda^{\mu}}
\newcommand{\set}[2]{\{#1:#2\}}
\newcommand{\Vtensr}{V^{\otimes r}}
\newcommand{\GL}{\mathrm{GL}}
\newcommand{\SO}{\mathrm{SO}}
\newcommand{\Des}{\mathrm{Des}}
\newcommand{\Symr}{\mathfrak{S}_r}
\newcommand{\SYT}{\mathrm{SYT}}

\newtheorem{theorem}{Theorem}[section]
\newtheorem{lemma}[theorem]{Lemma}
\newtheorem{corollary}[theorem]{Corollary}
\newtheorem{proposition}[theorem]{Proposition}
\theoremstyle{definition}
\newtheorem{definition}[theorem]{Definition}
\theoremstyle{remark}
\newtheorem{remark}[theorem]{Remark} 
\newtheorem{example}[theorem]{Example}
\newtheorem{conjecture}[theorem]{Conjecture}

\title[A Sundaram type bijection for $\mathrm{SO}(2k+1)$]{A Sundaram type bijection for $\mathrm{SO}(2k+1)$:\\ vacillating tableaux and pairs consisting of a standard Young tableau and an orthogonal Littlewood-Richardson tableau}
\author{Judith Jagenteufel}

\address{Institut für Diskrete Mathematik und Geometrie, Fakult\"at f\"ur Mathematik und Geoinformation, TU Wien, Austria,\newline Supported by the Austrian science fund (FWF): P29275}

\begin{document}

\begin{abstract} We present a bijection between vacillating tableaux and pairs consisting of a standard Young tableau and an orthogonal Littlewood-Richardson tableau for the special orthogonal group $\mathrm{SO}(2k+1)$. This bijection is motivated by the direct-sum-decomposition of the $r$th tensor power of the defining representation of $\mathrm{SO}(2k+1)$.
To formulate it, we use Kwon's orthogonal Littlewood-Richardson tableaux and introduce new alternative tableaux they are in bijection with.

Moreover we use a suitably defined descent set for vacillating tableaux to determine the quasi-symmetric expansion of the Frobenius characters of the isotypic components.
\end{abstract}

\maketitle

\section{Introduction}

We present a bijection for $\SOodd$ between vacillating tableaux and pairs consisting of a standard Young tableau and an orthogonal Littlewood-Richardson tableau.
This bijection explains the direct-sum-decomposition of a tensor power $\Vtensr$ of the defining representation $V$ of $\SOodd$ combinatorially. In particular we consider
\begin{align*}
\Vtensr=\bigoplus_{\mu} V(\mu)\otimes U(r,\mu)=\bigoplus_{\mu} V(\mu)\otimes\bigoplus_{\lambda} c_{\lambda}^{\mu}(\mathfrak{d}) S(\lambda)
\end{align*}
as an $\SOodd \times \Symr$ representation. $V(\mu)$ is an irreducible representation of $\SOodd$ and $S(\lambda)$ is a Specht module. We concentrate on $U(r,\mu)$. A basis of $U(r,\mu)$ can be indexed by vacillating tableaux. The multiplicities $c^{\mu}_{\lambda}$ can be obtained by counting orthogonal Littlewood-Richardson tableaux. A basis of $S(\lambda)$ is indexed by standard Young tableaux.

To formulate our bijection, we use Kwon’s orthogonal Littlewood-Richardson tableaux~\cite{MR3814326}. Those are defined in a very general way in terms of crystal graphs. We introduce an alternative set of orthogonal Littlewood-Richardson tableaux, which is in bijection with Kwon’s set via Bijection $A$ described by Algorithm~\ref{alg:1}. Our alternative tableaux are described in terms of skew semistandard tableaux with a reading word that is Yamanouchi. Those are similar to Sundarams symplectic tableaux~\cite{MR2941115}.  However, the additional condition we obtain is far more complicated than the one she obtained. Our new set of tableaux reduces the problem to finding a bijection between vacillating tableaux and standard Young tableaux with $2k+1$ rows, all of them with lengths of the same parity. We solve this reduced problem with Bijection $B$ described by Algorithm~\ref{alg:2}.


The question of finding such a bijection was posed by Sundaram in her 1986 thesis~\cite{MR2941115} and has been attacked several times since Sundaram’s thesis; in particular by Sundaram~\cite{MR1041447} and Proctor~\cite{MR1043509}. A key ingredient for us to find it were Kwon's orthogonal Littlewood-Richardson tableaux, defined recently in~\cite{MR3814326}.
Okada~\cite{MR3604801} recently obtained the decomposition of $U(r,\mu)$ for multiplicity free cases implicitly using representation theoretic computations. We obtain parts of these results as a special case, which are on their part special cases of Okada’s work. In fact, Okada asks for bijective proofs of his results.

One might assume that Fomin’s machinery of growth diagrams could be employed to find such a bijection. For the symplectic group this was done by Roby~\cite{MR2716353} and Krattenthaler~\cite{MR3534070}. However, for the special orthogonal group the situation appears to be quite different. In particular, at least a naive application of Fomin’s ideas does not even yield the desired bijection between vacillating tableaux and the set of standard Young tableaux in question, not even for dimension $3$.

For $\SOdrei$ a bijection was provided in~\cite{3erAlgo}. In dimension $3$ vacillating tableaux are Riordan paths: lattice paths with north-east, east and south-east steps, no steps below the $x$-axis and no east steps on the $x$-axis. This special combinatorial structure had led to stronger results there. For dimension $3$ the results we get are essentially the same as in~\cite{3erAlgo}. The only new result for dimension $3$ is the description of our alternative orthogonal Littlewood-Richardson tableaux.

An advantage of our combinatorial, bijective approach is that we obtain additional properties and consequences such as the following.

We define a suitable notion of descents for vacillating tableaux and use the classical descent set for standard Young tableaux introduced by Schützenberger. We can show that our bijection is descent preserving. Thus we obtain the quasi-symmetric expansion of the Frobenius character of the isotypic space $U(r,\mu)$:
\begin{align*}
\mathrm{ch} \,U(r,\mu)=\sum F_{\Des(w)}.
\end{align*}
where $F_D$ denotes a fundamental quasi-symmetric function, the sum runs over all vacillating tableaux $w$ of length $r$ and shape $\mu$ and $\Des(w)$ denotes the descent set of $w$.

Among others, this property justifies our bijection to be called “Sundaram-like”, as she described a similar bijection for the defining representation of the symplectic group in her thesis~\cite{MR2941115}. There exists a similar (but less complicated) definition for descents in oscillating tableaux, which are used in the symplectic case instead of vacillating tableaux, and which Sundaram’s bijection preserves. Thus there also exists a similar quasi-symmetric expansion of the Frobenius character, obtained for the symplectic group by Rubey, Sagan and Westbury in~\cite{MR3226822}.

\section{Background}

\subsection{Schur-Weyl duality} 

Considering the general linear group we start with the \enquote{classical Schur-Weyl duality}
\begin{align*}
\Vtensr \cong \bigoplus_{\substack{\lambda \vdash r \\ \ell(\lambda) \leq n}} V^{\GL}(\lambda) \otimes S(\lambda).
\end{align*}
Here $V$ is a complex vector space of dimension $n$. The general linear group $\GL(V)$ acts diagonally (and on each position by matrix multiplication) and the symmetric group $\Symr$ permutes tensor positions. Thus we consider a $\GL(V)\times \Symr$ representation. $V^{\GL}(\lambda)$ is an irreducible representation of $\GL(V)$ and $S(\lambda)$ is a Specht module.

Now we consider a vector space $V$ of odd dimension $n=2k+1$. To obtain a similar decomposition, we use the restriction from $\GL(V)$ to $\SO(V)$
\begin{align}
\label{eq:BranchingRule}
V(\lambda)\downarrow^{\GL(V)}_{\SO(V)} \cong \bigoplus_{\substack{\mu \text{ a partition} \\ \ell(\mu) \leq k}} c_{\lambda}^{\mu}(\mathfrak{d}) V^{SO}(\mu ),
\end{align}
where $\LRcoef$ is the multiplicity of the irreducible representation $V^{\SO}(\mu)$ of $\SO(V)$ in $V^{\GL}(\lambda)$. For $\ell(\lambda)\leq k$ this simplifies to the classical branching rule due to Littlewood.

Combining Schur-Weyl duality and the branching rule stated above we obtain an isomorphism of $\SO(V)\times \Symr$ representations
\begin{align*}
\Vtensr\cong \bigoplus_{\substack{\lambda \vdash r \\ \ell(\lambda) \leq n}} \Big( \bigoplus_{\substack{ \mu \text{ a partition}\\ \ell(\mu) \leq k}} c_{\lambda}^{\mu}(\mathfrak{d}) V^{SO}(\mu ) \Big) \otimes S(\lambda)
 = \bigoplus_{\substack{\mu \text{ a partition} \\ \ell(\mu) \leq k}} V^{SO}(\mu )\otimes U(r,\mu)
\end{align*}
 with isotypic components of weight $\mu$
\begin{align*}
U(r,\mu)=\bigoplus_{\substack{\lambda \vdash r \\ \ell(\lambda) \leq n}}  \LRcoef S(\lambda).
\end{align*}

The isomorphism of $\SO(V)$ representations (e.g. Okada~\cite[Cor. 3.6]{MR3604801}),
\begin{align*}
V^{SO}(\mu )\otimes V\cong\bigoplus_{\substack{\ell(\lambda)\leq k\\ \lambda=\mu\pm \square\\ \text{or } \lambda=\mu \text{ and } \ell(\mu)=k}} V^{SO}(\lambda)
\end{align*}
implies that a basis of $U(r,\mu)$ can be indexed by so called vacillating tableaux of shape $\mu$, defined in Section~\ref{sec:VacTab}.
Kwon defined orthogonal Littlewood-Richardson tableaux, as set that is counted by $\LRcoef$. We present Kwon's definition, as well as a new combinatorial description in Section~\ref{sec:KwonsOLRT} and introduce or new alternative tableaux in Section~\ref{sec:BijA}.
A basis of $S(\lambda)$ can be indexed by standard Young tableaux.
Therefore we are interested in a bijection between vacillating tableaux and pairs that consist of a standard Young tableau and an orthogonal Littlewood-Richardson tableau.

Moreover we introduce descent sets for vacillating tableau (see Section~\ref{sec:VacTab}). We show that our bijection preserves these descents, and follow the approach taken by Rubey, Sagan and Westbury~\cite{MR3226822} for the symplectic group. This enables us to describe the quasi-symmetric expansion of the Frobenius character (see the textbook by Stanley~\cite{MR1676282}). Recall that the Frobeinus character can can be defined by the requirement that it be an isometry and 
\begin{align*}
\mathrm{ch}\, S(\lambda)=s_\lambda=\sum_{Q\in \SYT(\lambda)} F_{\Des(Q)}
\end{align*}
where $s_\lambda$ is a Schur function, $\Des(Q)$ denotes the descent set of a standard Young tableau (see Section~\ref{sec:Descents}) and $F_D$ is the \emph{fundamental quasi-symmetric function}
\begin{align*}
F_D=\sum_{\substack{i_1\leq i_2\leq \dots\leq i_r\\ j\in D \Rightarrow i_j<i_{j+1}}} x_{i_1}x_{i_2}\dots x_{i_r}.
\end{align*}

Therefore we obtain the following theorem.
\begin{theorem}
\begin{align*}
\mathrm{ch} \,U(r,\mu)=\sum F_{\Des(w)},
\end{align*}
where the sum runs over all vacillating tableaux $w$ of length $r$ and shape $\mu$ and $\Des(w)$ is the descent set of $w$.
\end{theorem}

\subsection{Standard Young Tableaux and Skew Semistandard Tableaux}

We now introduce some well known concepts in order to clarify notation. For a textbook treatment see~\cite{MR1676282}.

\begin{definition}
A \emph{partition $\lambda\vdash n$} of a nonnegative integer $n$ is a sequence of positive integers $(\lambda_1,\lambda_2,\dots,\lambda_k)$ such that $\lambda_1\geq\lambda_2\geq\dots\geq\lambda_k>0$ and $\lambda_1+\lambda_2+\dots+\lambda_k=n$. The length $\ell(\lambda)$ of a partition $\lambda$ is the number of integers in this sequence namely $k$.

A \emph{Young diagram} of a partition $\lambda$ is a collection of left-adjusted cells such that each row consists of $\lambda_i$ cells.

The conjugate partition $\lambda'$ is the partition belonging to the transposed Young diagram of the partition $\lambda$.

Let $\mu$ and $\lambda$ be partitions such that $\mu\subseteq \lambda$ (thus $\mu_i\leq \lambda_i$). The \emph{skew shape $\lambda\backslash \mu$} is the Young diagram of $\lambda$ with the cells of the Young diagram of $\mu$ missing. The partition $\mu$ is the inner shape while the partition $\lambda$ is the outer shape.

A \emph{horizontal strip} is a skew shape such that no two cells are in the same column.
\end{definition}

\begin{definition}
A \emph{semistandard Young tableau} of shape $\lambda$ is obtained by a filling of the cells (with natural numbers) of the Young diagram of shape $\lambda$ such that each row is weakly increasing and each column is strictly increasing.

We also consider \emph{skew semistandard tableaux} where we take the Young diagram of a skew shape instead. We sometimes regard the missing cells as empty cells.

A \emph{reversed (skew) semistandard tableau} is a filling such that each row is weakly decreasing and each column is strictly decreasing.

The \emph{type} of a (reversed) semistandard Young tableau is $\mu=(\mu_1,\mu_2\dots\mu_l)$ where $\mu_i$ is the number of $i$'s in the tableau. 

A \emph{standard Young tableau} of shape $\lambda$ is a semistandard Young tableau with entries $1,2,\dots,|\lambda|$. Thus rows are also strictly increasing. We write SYT($\lambda$) for the set of standard Young tableaux of shape $\lambda$.

A tableau is \emph{column} (respectively \emph{row}) strict if its columns (respectively rows) are strictly increasing. 
\end{definition}

By abuse of notation we call a horizontal strip in a tableau a collection of entries whose cells form a horizontal strip in the Young diagram. 

\begin{definition}
The \emph{Robinson-Schensted correspondence} maps a word $w_1,\dots,w_m$ with $w_i\in \mathbb{N}$ to a pair $(P,Q)$ consisting of a semistandard Young tableau $P$, the insertion tableau, and a standard Young tableau $Q$, the recording tableau. (If and only if $w$ is a permutation the insertion tableau $P$ is also a standard Young tableau.)

To construct it we start with empty tableaux $P$ and $Q$. We insert positions $w_i$ of $w$ from left to right into $P$. We insert $w_i$ into the first row using the following procedure:

Element $e$ gets inserted into row $j$ as follows:
\begin{itemize}
\item If all elements in row $j$ are smaller than or equal to $e$, (or row $j$ is empty) place $e$ to the end of row $j$.
\item Otherwise search for the leftmost element $f$, that is larger than $e$, in row $j$. Put $e$ to its spot and insert $f$ into row $j+1$ using the same procedure again. We say that $f$ got \enquote{bumped} into the next row.
\end{itemize}
Insert $i$ into $Q$, where a new cell in $P$ was added.
\end{definition}

\begin{definition}
The \emph{reading word} of a (skew) (semi)standard Young tableau is the word obtained by concatenating the rows from bottom to top.
\end{definition}

\begin{definition}
A word $w$ with entries in the natural numbers $w_1,w_2,\dots,w_l$ is called a Yamanouchi word (or lattice permutation) if for all $i$ and any initial sequence $s$ the number of $i$'s in $s$ is at least as great as the number of $(i+1)$'s in $s$. 

A word $w_1,w_2,\dots,w_m$ is a \emph{reverse Yamanouchi word} if $w_m,\dots,w_2,w_1$ is Yamanouchi.
\end{definition}

For reverse Yamanouchi words the following theorem holds (see~\cite{MR675953}):

\begin{theorem}
\label{theo:YamanouchiWordSSYT}
If and only if a word $w$ is a reverse Yamanouchi word, the insertion tableau $P$ obtained by Robinson-Schensted is of the form

\begin{tikzpicture}[scale=0.35]
  \draw (0,6) -- (6,6);
  \draw (0,5) -- (6,5);
  \draw (0,4) -- (4.5,4);
  \draw (0,3) -- (3,3);
  \draw (0,2) -- (1,2);
  \draw (0,6) -- (0,2);
  \draw (1,6) -- (1,2);
  \draw (2,4) -- (2,3);
  \draw (3,4) -- (3,3);
  \draw (3.5,5) -- (3.5,4);
  \draw (4.5,5) -- (4.5,4);
  \draw (5,6) -- (5,5);
  \draw (6,6) -- (6,5);
  \draw (0.5,5.5) node {1};
  \draw (3.1,5.5) node {$\dots$};
  \draw (5.5,5.5) node {1};
  \draw (0.5,4.5) node {2};    
  \draw (2.3,4.5) node {$\dots$};
  \draw (4,4.5) node {2};
  \draw (0.5,3.5) node {3};
  \draw (1.60,3.5) node {$\dots$};
  \draw (2.5,3.5) node {3};
  \draw (0.5,2.77) node {$\vdots$};
\end{tikzpicture}.
\end{theorem}

\subsubsection{Descents of Standard Young Tableaux}

\begin{definition}
Let $Q\in \SYT(\lambda)$ be a standard Young tableau. An entry $j$ is a \emph{descent} if $j+1$ is in a row below $j$. We define the \emph{descent set} of $Q$ as:
$\Des(Q)=\set{j}{j \text{ is a descent of } Q}$.
\end{definition}
\begin{example}
\label{ex:DescSYT}
The following standard Young tableau has descent set $\{2,3,5,7,12\}$. Descents $j$ are bold, $j+1$ are italic.

\hspace{5cm}
\begin{tikzpicture}[scale=0.35]
  \draw (0,6) -- (6,6);
  \draw (0,5) -- (6,5);
  \draw (0,4) -- (2,4);
  \draw (0,3) -- (2,3);
  \draw (0,2) -- (2,2);
  \draw (0,1) -- (2,1);
  \draw (0,6) -- (0,1);
  \draw (1,6) -- (1,1);
  \draw (2,6) -- (2,1);
  \draw (3,6) -- (3,5);
  \draw (4,6) -- (4,5);
  \draw (5,6) -- (5,5);
  \draw (6,6) -- (6,5);
  \draw (0.5,5.5) node {1};
  \draw (1.5,5.5) node {\textbf{2}};
  \draw (0.5,4.5) node {\textbf{\textit{3}}};
  \draw (0.5,3.5) node {\textit{4}};
  \draw (1.5,4.5) node {\textbf{5}};
  \draw (0.5,2.5) node {\textit{6}};
  \draw (1.5,3.5) node {\textbf{7}};
  \draw (0.5,1.5) node {\textit{8}};
  \draw (1.5,2.5) node {9};
  \draw (2.5,5.5) node {10};
  \draw (3.5,5.5) node {11};
  \draw (4.5,5.5) node {\textbf{12}};
  \draw (1.5,1.5) node {\textit{13}};
  \draw (5.5,5.5) node {14};
\end{tikzpicture}
\end{example}

\subsubsection{Concatenation of Standard Young tableaux}

\begin{definition}
The \emph{concatenation} $Q$ of two standard Young tableaux $Q_1$ and $Q_2$ is obtained as follows. First add the largest entry of $Q_1$ to each entry of $Q_2$ to obtain the tableau $\widetilde{Q_2}$. Then append row $i$ of $\widetilde{Q_2}$ to row $i$ of $Q_1$ to obtain $Q$.
\end{definition}

This procedure is associative, thus we can consider the concatenation of several standard Young tableaux. We say a standard Young tableau $Q$ is the concatenation of $m$ standard Young tableaux if we can find standard Young tableaux $Q_1,\dots,Q_m$ such that $Q$ is the concatenation of those. We will be interested only in those concatenations where all tableaux have either rows of even length, or row lengths of the same parity, each.

\begin{example}
\label{ex:ConCatSYTDef}
We concatenate two standard Young tableaux

\hspace{3cm}\begin{tikzpicture}[scale=0.35]
  \draw (0,6) -- (6,6);
  \draw (0,5) -- (6,5);
  \draw (0,4) -- (2,4);
  \draw (0,3) -- (2,3);
  \draw (0,6) -- (0,3);
  \draw (1,6) -- (1,3);
  \draw (2,6) -- (2,3);
  \draw (3,6) -- (3,5);
  \draw (4,6) -- (4,5);
  \draw (5,6) -- (5,5);
  \draw (6,6) -- (6,5);
  \draw (0.5,5.5) node {1};
  \draw (1.5,5.5) node {2};
  \draw (0.5,4.5) node {3};
  \draw (1.5,4.5) node {4};
  \draw (2.5,5.5) node {5};
  \draw (3.5,5.5) node {6};
  \draw (0.5,3.5) node {7};
  \draw (1.5,3.5) node {8};
  \draw (4.5,5.5) node {9};
  \draw (5.5,5.5) node {10};
  \draw (0,1);
\end{tikzpicture}
\raisebox{1cm}{and}
\begin{tikzpicture}[scale=0.35]
  \draw (0,6) -- (1,6);
  \draw (0,5) -- (1,5);
  \draw (0,4) -- (1,4);
  \draw (0,3) -- (1,3);
  \draw (0,2) -- (1,2);
  \draw (0,1) -- (1,1);
  \draw (0,6) -- (0,1);
  \draw (1,6) -- (1,1);
  \draw (0.5,5.5) node {1};
  \draw (0.5,4.5) node {2};
  \draw (0.5,3.5) node {3};
  \draw (0.5,2.5) node {4};
  \draw (0.5,1.5) node {5};
\end{tikzpicture}
\raisebox{1cm}{and obtain}
\begin{tikzpicture}[scale=0.35]
  \draw (0,6) -- (7,6);
  \draw (0,5) -- (7,5);
  \draw (0,4) -- (3,4);
  \draw (0,3) -- (3,3);
  \draw (0,2) -- (1,2);
  \draw (0,1) -- (1,1);
  \draw (0,6) -- (0,1);
  \draw (1,6) -- (1,1);
  \draw (2,6) -- (2,3);
  \draw (3,6) -- (3,3);
  \draw (4,6) -- (4,5);
  \draw (5,6) -- (5,5);
  \draw (6,6) -- (6,5);
  \draw (7,6) -- (7,5);
  \draw (0.5,5.5) node {1};
  \draw (1.5,5.5) node {2};
  \draw (0.5,4.5) node {3};
  \draw (1.5,4.5) node {4};
  \draw (2.5,5.5) node {5};
  \draw (3.5,5.5) node {6};
  \draw (0.5,3.5) node {7};
  \draw (1.5,3.5) node {8};
  \draw (4.5,5.5) node {9};
  \draw (5.5,5.5) node {10};
  \draw (6.5,5.5) node {11};
  \draw (2.5,4.5) node {12};
  \draw (2.5,3.5) node {13};
  \draw (0.5,2.5) node {14};
  \draw (0.5,1.5) node {15};
\end{tikzpicture}
\raisebox{1cm}{.}

The first tableau itself is a concatenation of standard Young tableaux. The parts are the tableau containing only numbers $1$ up to  $8$ and two single cells containing $1$. If one is interested in a concatenation of tableaux with row lengths of the same parity, we can also take the tableau containing numbers $1$ up to $8$ and as second tableau, the one rowed tableau containing $1$ and $2$. (Empty rows $j$ are counted as rows of even length for $j<n$.)
\end{example}

\subsection{Vacillating Tableaux}
\label{sec:VacTab}

We define vacillating tableaux (as defined by Sundaram in~\cite[Def. 4.1]{MR1041447}) in three different ways, once as sequence of Young diagrams, once in terms of highest weight words and once as certain $k$-tuples of lattice paths.
\begin{definition}
\begin{enumerate}
\item A ($(2k+1)$-\emph{orthogonal}) \emph{vacillating tableau} of length $r$ is a sequence of Young diagrams $\emptyset=\mu^0, \mu^1,\dots,\mu^r=\mu$ each of at most $k$ parts, such that:
\begin{itemize}
\item $\mu^i$ and $\mu^{i+1}$ differ in at most one cell,
\item $\mu^i=\mu^{i+1}$ only occurs if the $k$th row of cells is non-empty.
\end{itemize}
The partition belonging to the final Young diagram $\mu$ is the \emph{shape} of the tableau.
\item A ($(2k+1)$-\emph{orthogonal}) \emph{highest weight word} is a word $w$ with letters in $\{\pm 1, \pm 2, \dots, \pm k, 0\}$ of length $r$ such that for every initial segment $s$ of $w$ the following holds (we write $\#i$ for the number of $i$'s in $s$):
\begin{itemize}
\item $\#i - \#(-i) \geq 0$,
\item $\#i-\#(-i)\geq \#(i+1)-\#(-i-1)$,
\item if the last letter is $0$ then $\#k-\#(-k)>0$.
\end{itemize}
The partition $(\#1-\#(-1),\#2-\#(-2),\dots, \#k - \#(-k))$ is the \emph{weight} of a highest weight word. The vacillating tableau corresponding to a word $w$ is the sequence of weights of the initial segments of $w$.
\item Riordan paths are Motzkin paths without horizontal steps on the $x$-axis. They consist of up (north-east) steps, down (south-east) steps, and horizontal (east) steps, such that there is no step beneath the $x$-axis and no horizontal step on the $x$-axis.

A $k$-tuple of Riordan paths of length $r$ is a vacillating tableau of length $r$ if it meets the following conditions:
\begin{itemize}
\item The first path is a Riordan path of length $r$.
\item Path $i$ has steps where path $i-1$ has horizontal steps. Path $i$ is never higher than path $i-1$.
\end{itemize}
For a better readability we sometimes label the steps with $1,\dots,r$ in order to see which steps belong together and shift paths together.

The corresponding highest weight word is described as follows: A value $i$ is an up-step in path $i$ and a horizontal step in paths $1$ up to $i-1$. Similarly a value $-i$ is a down-step in path $i$ and a horizontal step in paths $1$ up to $i-1$ and a value $0$ is a horizontal step in every path, including $k$.
\end{enumerate}
By abuse of terminology we refer to all three objects as \emph{vacillating tableaux}.
\end{definition}

\begin{example}
\label{ex:VacTab}
The same object once written as a vacillating tableau, once as a highest weight word and once as a Riordan path. To the left we draw the Riodan path labeled and the second path shifted together in the way we described above.

\noindent 
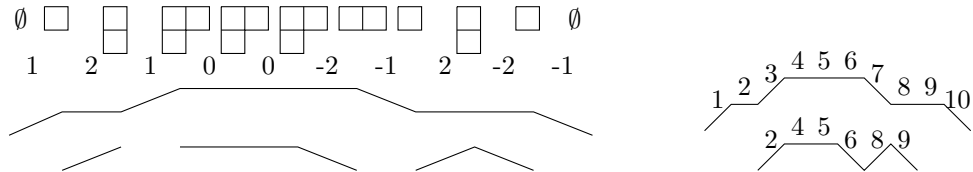
\begin{tikzpicture}[scale=0.31]
\draw (0,6.5) node{$\emptyset$};
\draw (1,7)--(2,7)--(2,6)--(1,6)--(1,7);
\draw (3.5,7)--(4.5,7)--(4.5,5)--(3.5,5)--(3.5,7);
\draw (3.5,6)--(4.5,6);
\draw (6,7)--(8,7)--(8,6)--(7,6)--(7,5)--(6,5)--(6,7);
\draw (6,6)--(7,6)--(7,7);
\draw (8.5,7)--(10.5,7)--(10.5,6)--(9.5,6)--(9.5,5)--(8.5,5)--(8.5,7);
\draw (8.5,6)--(9.5,6)--(9.5,7);
\draw (11,7)--(13,7)--(13,6)--(12,6)--(12,5)--(11,5)--(11,7);
\draw (11,6)--(12,6)--(12,7);
\draw (13.5,7)--(15.5,7)--(15.5,6)--(13.5,6)--(13.5,7);
\draw (14.5,6)--(14.5,7);
\draw (16,7)--(17,7)--(17,6)--(16,6)--(16,7);
\draw (18.5,7)--(19.5,7)--(19.5,5)--(18.5,5)--(18.5,7);
\draw (18.5,6)--(19.5,6);
\draw (21,7)--(22,7)--(22,6)--(21,6)--(21,7);
\draw (23.5,6.5) node{$\emptyset$};
\draw (0.5,4.5) node{1};
\draw (3,4.5) node{2};
\draw (5.5,4.5) node{1};
\draw (8,4.5) node{0};
\draw (10.5,4.5) node{0};
\draw (13,4.5) node{-2};
\draw (15.5,4.5) node{-1};
\draw (18,4.5) node{2};
\draw (20.5,4.5) node{-2};
\draw (23,4.5) node{-1};
\draw (-0.5,1.5)--(1.75,2.5)--(4.25,2.5)--(6.75,3.5)--(9.25,3.5)--(11.75,3.5)--(14.25,3.5)--(16.75,2.5)--(19.25,2.5)--(21.75,2.5)--(24.25,1.5);
\draw(1.75,0)--(4.25,1);
\draw(6.75,1)--(9.25,1)--(11.75,1)--(14.25,0);
\draw(16.75,0)--(19.25,1)--(21.75,0);
\end{tikzpicture}
\hspace{1.2cm}
\begin{tikzpicture}[scale=0.35]
\draw (0,0)  --
(1,1) node[midway, above] {1}--
(2,1) node[midway, above] {2}--
(3,2) node[midway, above] {3}--
(4,2) node[midway, above] {4}--
(5,2) node[midway, above] {5}--
(6,2) node[midway, above] {6}--
(7,1) node[midway, above] {7}--
(8,1) node[midway, above] {8}--
(9,1) node[midway, above] {9}--
(10,0) node[midway, above] {10};
\draw (2,-1.5)  --
(3,-0.5) node[midway, above] {2}--
(4,-0.5) node[midway, above] {4}--
(5,-0.5) node[midway, above] {5}--
(6,-1.5) node[midway, above] {6}--
(7,-0.5) node[midway, above] {8}--
(8,-1.5) node[midway, above] {9};
\end{tikzpicture}
\end{example}

\subsubsection{Descents of Vacillating Tableaux}
\label{sec:Descents}

\begin{definition}
We define descents for vacillating tableaux using highest weight words. A letter $w_i$ of $w$ is a \emph{descent} if there exists a directed path from $w_i$ to $w_{i+1}$ in the crystal graph for the defining representation of $\SOodd$
\[  1\to 2\to \dots\to k \to 0 \to -k \to \dots \to -1\]
and $w_iw_{i+1} \neq j(-j)$ if for the initial segment $w_1,\dots,w_{i-1}$ holds $\#j-\#(-j)=0$.

We define the \emph{descent set} of $w$ as $\Des(w)= \set{j}{ j \text{ is a descent of } w}$.
\end{definition}

In our tuple of paths a descent is a convex edge of consecutive steps, but not an up-step followed from a down-step on the bottom.

\begin{example}
The following vacillating tableau has descent set $\{2,3,5,7,12\}$. The corresponding positions are circled. Note that $10$ is no descent, as $10,11$ are on bottom level. (It is no coincidence that the standard Young tableau in Example~\ref{ex:DescSYT} has the same descents as they are assigned to each other by Bijection $B$.)

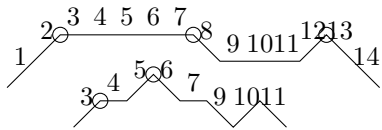
\hspace{4cm}\begin{tikzpicture}[scale=0.35]
\draw (0,0)  --
(1,1) node[midway, above] {1}--
(2,2) node[midway, above] {2} node[draw=black,circle,inner sep=0cm,minimum size=0.2cm] {}--
(3,2) node[midway, above] {3}--
(4,2) node[midway, above] {4}--
(5,2) node[midway, above] {5}--
(6,2) node[midway, above] {6}--
(7,2) node[midway, above] {7} node[draw=black,circle,inner sep=0cm,minimum size=0.2cm] {}--
(8,1) node[midway, above] {8}--
(9,1) node[midway, above] {9}--
(10,1) node[midway, above] {10}--
(11,1) node[midway, above] {11}--
(12,2) node[midway, above] {12} node[draw=black,circle,inner sep=0cm,minimum size=0.2cm] {}--
(13,1) node[midway, above] {13}--
(14,0) node[midway, above] {14};
\draw (2.5,-1.5)  --
(3.5,-0.5) node[midway, above] {3} node[draw=black,circle,inner sep=0cm,minimum size=0.2cm] {}--
(4.5,-0.5) node[midway, above] {4}--
(5.5,0.5) node[midway, above] {5} node[draw=black,circle,inner sep=0cm,minimum size=0.2cm] {}--
(6.5,-0.5) node[midway, above] {6}--
(7.5,-0.5) node[midway, above] {7}--
(8.5,-1.5) node[midway, above] {9}--
(9.5,-0.5) node[midway, above] {10}--
(10.5,-1.5) node[midway, above] {11};
\end{tikzpicture}
\end{example}

\subsubsection{Concatenation of Vacillating Tableaux}

\begin{definition}
The concatenation of vacillating tableaux of shape $\emptyset$ is obtained by writing them side by side. If we writing them labeled, we adjust the labels such that they are increasing from left to right.
\end{definition}

\begin{example}
The following vacillating tableau is the concatenation of three vacillating tableaux, first the steps $1$ to $8$, then $9,10$, and third the steps $11$ to $15$. (We will see that it corresponds to the standard Young tableau of Example~\ref{ex:ConCatSYTDef} under Bijection $B$.)

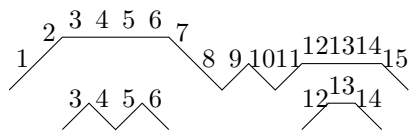
\hspace{4cm}\begin{tikzpicture}[scale=0.35]
\draw (0,0)--
(1,1) node[midway, above] {1}--
(2,2) node[midway, above] {2}--
(3,2) node[midway, above] {3}--
(4,2) node[midway, above] {4}--
(5,2) node[midway, above] {5}--
(6,2) node[midway, above] {6}--
(7,1) node[midway, above] {7}--
(8,0) node[midway, above] {8}--
(9,1) node[midway, above] {9}--
(10,0) node[midway, above] {10}--
(11,1) node[midway, above] {11}--
(12,1) node[midway, above] {12}--
(13,1) node[midway, above] {13}--
(14,1) node[midway, above] {14}--
(15,0) node[midway, above] {15};
\draw (2,-1.5)--
(3,-0.5) node[midway, above] {3}--
(4,-1.5) node[midway, above] {4}--
(5,-0.5) node[midway, above] {5}--
(6,-1.5) node[midway, above] {6};
\draw (11,-1.5)--
(12,-0.5) node[midway, above] {12}--
(13,-0.5) node[midway, above] {13}--
(14,-1.5) node[midway, above] {14};
\end{tikzpicture}
\end{example}

\subsection{Crystal Graphs}
\label{sec:CrystalGraphs}

In this section we summarize some properties of crystal graphs. In particular, we describe a certain crystal graph, that we need for defining orthogonal Littlewood-Richardson tableaux. For more information on crystals see the textbook by Hong and Kang~\cite{MR1881971}.

\medskip

\emph{Crystal graphs} are certain acyclic directed graphs where vertices have finite in- and out-degree and each edge is labeled by a natural number. We only use crystal graphs whose vertices are labeled with certain tableaux.

For each vertex $C$ there is at most one outgoing edge labeled with $i$. If such an edge exists we denote its target by $f_i(C)$. Otherwise $f_i(C)$ is defined to be the distinguished symbol $\varnothing$.
Analogously there is at most one incoming edge labeled with $i$ and we define $e_i(C)$ as the tableau obtained by following an incoming edge labeled with $i$. We denote by $\varphi_i(C)$ (respectively $\varepsilon_i(C)$) the number of times one can apply $f_i$ (respectively $e_i$) to $C$.

We consider infinite crystal graphs. However, for the crystal graphs we consider, it holds that if we fix a natural number $\ell$ and delete all edges labeled with $\ell$ or larger, as well as all vertices that have incoming edges labeled with $\ell$ or larger, we obtain a finite crystal graph. Thus a lot of properties proven for finite crystal graphs hold also for our infinite crystal graphs.

The crystal graphs we consider are all tensor products of the following crystal graph.

\begin{definition}
The crystal graph of one-column tableaux is defined as follows:
\begin{enumerate}
\item The vertices are column strict tableaux with a single column and positive integers as entries.
\item Suppose that $i\in \mathbb{N}, i>0$ is an entry in a tableau $C$ but $i+1$ is not. Then $f_i(C)$ is the tableau one obtains by replacing $i$ by $i+1$.
Otherwise $f_i(C)=\varnothing$.
\item Suppose that neither $1$ nor $2$ is an entry in a tableau $C$. Then $f_0(C)$ is the tableau one obtains by adding a domino
\begin{tikzpicture}[scale=0.30]
\draw (3,14) -- (3,16);
\draw (4,14) -- (4,16);
\draw (3,14) -- (4,14);
\draw (3,15) -- (4,15);
\draw (3,16) -- (4,16);
\draw (3.5,14.5) node{2};
\draw (3.5,15.5) node{1};
\end{tikzpicture}
on top of $C$. Otherwise $f_0(C)=\varnothing$.
\end{enumerate}
See Figure~\ref{fig:Crystal} for an example.
\end{definition}

\begin{figure}
\centering
\begin{tikzpicture}[scale=0.35]
\draw (-6.1,-3.4) node{$\vdots$};
\draw (-2.5,-3.4) node{$\vdots$};
\draw (1.5,-3.4) node{$\vdots$};
\draw (3.5,-4.4) node{$\vdots$};
\draw (5.5,-5.4) node{$\vdots$};
\draw [<-] (5.5,-4.9) -- (3.7,-2.1) node[midway, above]{\footnotesize 0};
\draw [<-] (1.7,-2.9) -- (3.3,-2.1) node[midway, above]{\footnotesize 5};
\draw [<-] (3.5,-3.9) -- (3.5,-2.1) node[midway, left]{\footnotesize 3};
\draw [<-] (1.3,-2.9) -- (-0.3,-2.1) node[midway, above]{\footnotesize 2};
\draw [<-] (-2.5,-2.9) -- (-0.7,-2.1) node[midway, above]{\footnotesize 6};
\draw [<-] (-2.7,-2.9) -- (-4.3,-2.1) node[midway, above]{\footnotesize 1};
\draw [<-] (-6.3,-2.9) -- (-4.7,-2.1) node[midway, above]{\footnotesize 7};
\draw [<-] (5.7,-4.9) -- (7.5,-4.1) node[midway, above]{\footnotesize 4};
\draw (7,-4) -- (7,0);
\draw (8,-4) -- (8,0);
\draw (7,-4) -- (8,-4);
\draw (7,-3) -- (8,-3);
\draw (7,-2) -- (8,-2);
\draw (7,-1) -- (8,-1);
\draw (7,0) -- (8,0);
\draw (7.5,-3.5) node{4};
\draw (7.5,-2.5) node{3};
\draw (7.5,-1.5) node{2};
\draw (7.5,-0.5) node{1};
\draw (3,-2) -- (3,0);
\draw (4,-2) -- (4,0);
\draw (3,-2) -- (4,-2);
\draw (3,-1) -- (4,-1);
\draw (3,0) -- (4,0);
\draw (3.5,-1.5) node{5};
\draw (3.5,-0.5) node{3};
\draw (-1,-2) -- (-1,0);
\draw (0,-2) -- (0,0);
\draw (-1,-2) -- (0,-2);
\draw (-1,-1) -- (0,-1);
\draw (-1,0) -- (0,0);
\draw (-0.5,-1.5) node{6};
\draw (-0.5,-0.5) node{2};
\draw (-5,-2) -- (-5,0);
\draw (-4,-2) -- (-4,0);
\draw (-5,-2) -- (-4,-2);
\draw (-5,-1) -- (-4,-1);
\draw (-5,0) -- (-4,0);
\draw (-4.5,-1.5) node{7};
\draw (-4.5,-0.5) node{1};
\draw [<-] (3.5,0.1) -- (1.7,0.9) node[midway, above]{\footnotesize 2};
\draw [<-] (-0.3,0.1) -- (1.3,0.9) node[midway, above]{\footnotesize 5};
\draw [<-] (-0.7,0.1) -- (-2.3,0.9) node[midway, above]{\footnotesize 1};
\draw [<-] (-4.5,0.1) -- (-2.7,0.9) node[midway, above]{\footnotesize 6};
\draw [<-] (3.7,0.1) -- (5.3,0.9) node[midway, above]{\footnotesize 4};
\draw [<-] (7.5,0.1) -- (5.7,0.9) node[midway, above]{\footnotesize 0};
\draw (5,1) -- (5,3);
\draw (6,1) -- (6,3);
\draw (5,1) -- (6,1);
\draw (5,2) -- (6,2);
\draw (5,3) -- (6,3);
\draw (5.5,1.5) node{4};
\draw (5.5,2.5) node{3};
\draw (1,1) -- (1,3);
\draw (2,1) -- (2,3);
\draw (1,1) -- (2,1);
\draw (1,2) -- (2,2);
\draw (1,3) -- (2,3);
\draw (1.5,1.5) node{5};
\draw (1.5,2.5) node{2};
\draw (-3,1) -- (-3,3);
\draw (-2,1) -- (-2,3);
\draw (-3,1) -- (-2,1);
\draw (-3,2) -- (-2,2);
\draw (-3,3) -- (-2,3);
\draw (-2.5,1.5) node{6};
\draw (-2.5,2.5) node{1};
\draw [<-] (5.5,3.1) -- (3.7,3.9) node[midway, above]{\footnotesize 2};
\draw [<-] (1.7,3.1) -- (3.3,3.9) node[midway, above]{\footnotesize 4};
\draw [<-] (1.3,3.1) -- (-0.3,3.9) node[midway, above]{\footnotesize 1};
\draw [<-] (-2.5,3.1) -- (-0.7,3.9) node[midway, above]{\footnotesize 5};
\draw (3,4) -- (3,6);
\draw (4,4) -- (4,6);
\draw (3,4) -- (4,4);
\draw (3,5) -- (4,5);
\draw (3,6) -- (4,6);
\draw (3.5,4.5) node{4};
\draw (3.5,5.5) node{2};
\draw (-1,4) -- (-1,6);
\draw (0,4) -- (0,6);
\draw (-1,4) -- (0,4);
\draw (-1,5) -- (0,5);
\draw (-1,6) -- (0,6);
\draw (-0.5,4.5) node{5};
\draw (-0.5,5.5) node{1};
\draw [<-] (3.3,6.1) -- (1.7,6.9) node[midway, below]{\footnotesize 1};
\draw [<-] (3.7,6.1) -- (5.5,6.9) node[midway, below]{\footnotesize 3};
\draw [<-] (-0.5,6.1) -- (1.3,6.9) node[midway, below]{\footnotesize 4};
\draw (1,7) -- (1,9);
\draw (2,7) -- (2,9);
\draw (1,7) -- (2,7);
\draw (1,8) -- (2,8);
\draw (1,9) -- (2,9);
\draw (1.5,7.5) node{4};
\draw (1.5,8.5) node{1};
\draw (5,7) -- (5,9);
\draw (6,7) -- (6,9);
\draw (5,7) -- (6,7);
\draw (5,8) -- (6,8);
\draw (5,9) -- (6,9);
\draw (5.5,7.5) node{3};
\draw (5.5,8.5) node{2};
\draw [->] (3.3,9.9) -- (1.5,9.1) node[midway, above]{\footnotesize 3};
\draw [->] (3.8,9.9) -- (5.5,9.1) node[midway, above]{\footnotesize 1};
\draw (3,10) -- (3,12);
\draw (4,10) -- (4,12);
\draw (3,10) -- (4,10);
\draw (3,11) -- (4,11);
\draw (3,12) -- (4,12);
\draw (3.5,10.5) node{3};
\draw (3.5,11.5) node{1};
\draw [<-] (3.5,12.1) -- (3.5,13.9) node[midway, right]{\footnotesize 2};
\draw (3,14) -- (3,16);
\draw (4,14) -- (4,16);
\draw (3,14) -- (4,14);
\draw (3,15) -- (4,15);
\draw (3,16) -- (4,16);
\draw (3.5,14.5) node{2};
\draw (3.5,15.5) node{1};
\end{tikzpicture}
\caption{The component with even-length tableaux of a crystal consisting of one-column tableaux.}
\label{fig:Crystal}
\end{figure}
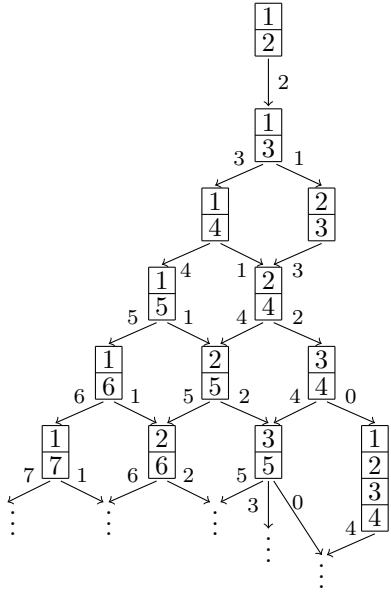

We define now the tensor product of crystal graphs. We will use the tensor products of the crystal graph defined above for defining orthogonal Littlewood-Richardson tableaux.

\begin{definition}
The tensor product  $B_1\otimes B_2$  of two crystal graphs $B_1$ and $B_2$ is a crystal graph with vertex set $B_1 \times B_2$ and edges satisfying:
\begin{align*}
f_i(b\otimes b')&=\left\{ \begin{array}{ll} b\otimes f_i(b') & \text{if } \varepsilon_i(b)<\varphi_i(b')\\ f_i(b)\otimes b' & \text{otherwise} \end{array} \right.\\
e_i(b\otimes b')&=\left\{ \begin{array}{ll} e_i(b)\otimes b'  & \text{if } \varepsilon_i(b)>\varphi_i(b')\\ b\otimes e_i(b') & \text{otherwise} \end{array} \right.
\end{align*}
and
\begin{align*}
\varphi_i(b\otimes b')&= \varphi_i(b) + \mathrm{max}(0,\varphi_i(b')-\varepsilon_i(b))\\
\varepsilon_i(b\otimes b')&= \varepsilon_i(b') + \mathrm{max}(0,\varepsilon_i(b)-\varphi_i(b')).\\
\end{align*}
\end{definition}






\section{Kwon's Orthogonal Littlewood-Richardson Tableaux}
\label{sec:KwonsOLRT}

In this section we will first present Kwon's orthogonal Littlewood-Richardson tableaux. This description is very general, so we give a new, explicit formulation of his orthogonal Littlewood-Richardson tableaux afterwards.

Although Kwon considers $\mathrm{O}(n)$, for odd $n$ ($n=2k+1$) we get $\SO(n)$ as a special case. In this case $V(\lambda)\downarrow^{\mathrm{O}(2k+1)}_{\SO(2k+1)}$ is an irreducible $\SO(2k+1)$ representation and every such representation is isomorphic to such a restriction (see for example Okada~\cite[Sect. 2.4]{MR3604801}).

We start with introducing some notation we will use.

\begin{definition}
Let $T$ be a two column skew semistandard tableau of shape $(2^b,1^m)/(1^a)$, with $b\geq a \geq 0$ and $m>0$.

The \emph{tail} of $T$ is the part where only the first column exists, that is, the lower $m$ entries of the first column. The topmost tail position is the \emph{tail root} and the rest of the tail is the \emph{lower tail}.

The \emph{fin} of $T$ is the largest entry in the second column.

The \emph{residuum} of $T$ is the number of positions the second column can be shifted down while maintaining semistandardness. In particular, the residuum of $T$ is at most $\min(a,m)$.
\end{definition}

\begin{definition}
For a partition $\mu$ with at most $k$ parts we define the crystal graph $B^{\mathfrak{d}}(\mu)$ as follows. It is the subgraph of the the tensor product of $n=2k+1$ one column crystal graphs, whose vertices are tuples $(T_1,T_2,\dots,T_{\ell(\mu)},S)$ of skew semistandard tableaux such that:
\begin{itemize}
\item Each $T_j$ has shape  $(2^{b_j},1^{\mu_j})/(1^{a_j})$, with $b_j\geq a_j \geq 0$, $b_j,a_j$ even and residuum at most one.
\item $S$ is of rectangular outer shape and has $n-2\ell(\mu)$ (possibly empty) columns, all whose lengths have the same parity. We say $S$ is even if its columns have even length, and $S$ is odd otherwise.
\end{itemize}
\end{definition}

\begin{lemma}[Defining Lemma, Kwon] The crystal $T^\mathfrak{d}(\mu)$ is the subgraph of $B^{\mathfrak{d}}(\mu)$ whose vertices are in the same component as one of the following highest weight elements:
\begin{itemize}
\item $T_j$ has its left column filled with $1,2,\dots,\mu_j$ and its right column empty.
\item Either $S$ is empty or $S$ is a single row of $n-2\ell(\mu)$ entries equal to $1$.
\end{itemize}
\end{lemma}

\begin{definition}
The set of \emph{orthogonal Littlewood-Richardson tableaux} is \[\LRtabs =\set{L\in T^{\mathfrak{d}}(\mu)}{\text{ $i$ occurs in $L$ exactly $\lambda'_i$ times and } \varepsilon_i(L)=0 \text{ for } i \neq 0}\]
with $\ell(\lambda)\leq n=2k+1$ and $\ell(\mu)\leq k$. 
\end{definition}

As announced before the set of orthogonal Littlewood-Richardson tableaux is counted by $\LRcoef$. See~\cite[Theorem 5.3]{MR3814326}. This is one of the main results of \cite{MR3814326}.

\begin{theorem}[Kwon]
 $\LRcoef=|\LRtabs|$
\end{theorem}

For two column skew shape semistandard tableaux we define admissibility which tells us if an element of $B^{\mathfrak{d}}(\mu)$ is in $T^\mathfrak{d}(\mu)$. To do so we need for a skew semistandard tableau consisting of a left and a right column $T=(T^L,T^R)$ the pairs $(^LT,^RT)$ and $(T^{L^*},T^{R^*})$:

\begin{definition}
Let $T=(T^L,T^R)$ be a two column skew semistandard tableau.

We define the pair $(^LT,^RT)$ of two one-column, column strict tableaux as follows. Beginning at the bottom, we slide each cell of $T^R$ down as far as possible, not beyond the bottom cell of $T^L$ and so that the entry of its left neighbor is not larger. Then ${^RT}$ consists of all entries $T^R$, together with those in $T^L$ that have no right neighbor. ${^LT}$ consists of the remaining entries in $T^L$.

If $T$ has residuum $1$, we define additionally the pair $(T^{L^*},T^{R^*})$ of two one-column, column strict tableaux as follows. Beginning on the top, we slide each cell of $T^L$ up as far as possible, not beyond the top cell of $T^R$ and so that the entry of its right neighbor is not smaller. Then $T^{L^{*}}$ consists of all entries in $T^L$, together with the largest entry in $T^R$ that has no left neighbor. Note that such an entry must exist because $T$ has residuum $1$ and $a$ is even, thus $a\geq 2$. $T^{R^{*}}$ consists of the remaining entries in $T^R$.
\end{definition}

See Figure~\ref{fig:Admissible} for examples.

\begin{definition}[Kwon]
For a single column $C$, let $C(i)$ be the $i$th entry from the bottom and $\mathrm{ht}(C)$ its length.

Let $T$ and $U$ be two two-column skew semistandard tableaux with tails of length $\mu_T$ and $\mu_U$ such that $\mu_T\geq\mu_U>0$ and residuum $r_T\leq 1$ and $r_U\leq 1$, respectively. The pair  $(T,U)$ is \emph{admissible}, if the following conditions are met: 
\begin{align}
\mathrm{ht}(T^R) \leq \mathrm{ht}(U^L)-\mu_U+2r_Tr_U \tag{H}
\label{eq:H}
\end{align}
\begin{align}
 T^R(i)\leq  {^LU(i)} & \quad \text{ if }  r_T\cdot r_U=0\tag{A1}
\label{eq:A1}\\
T^{R^*}(i)\leq  {^LU(i)} & \quad \text{ if }  r_T\cdot r_U=1\nonumber
\end{align}
\begin{align}
{^RT}(i+\mu_T -\mu_U)\leq U^L(i) & \quad \text{ if } r_T\cdot r_U=0\tag{A2}
\label{eq:A2}\\
{^RT}(i+\mu_T -\mu_U)\leq U^{L^*}(i) & \quad \text{ if } r_T\cdot r_U=1\nonumber
\end{align}

Let $T$ be a two-column skew semistandard tableaux with tail of length $\mu_T>0$ and residuum $r_T\leq 1$. Let $S$ be a skew semistandard tableau of rectangular outer shape with first column $S^L$ and columns with lengths of the same parity. The pair $(T,S)$ is \emph{admissible}, if the following conditions are met:
\begin{align}
\mathrm{ht}(T^R) \leq \mathrm{ht}(S^L) & \quad \text{ if } S \text{ is even}\tag{H$'$}
\label{eq:H_}\\
\mathrm{ht}(T^R) \leq \mathrm{ht}(S^L)-1+2r_T & \quad \text{ otherwise }\nonumber
\end{align}
\begin{align}
 T^R(i)\leq  {S^L(i)} & \quad \text{ if $S$ is even or }  r_T=0 \tag{A1$'$}
\label{eq:A1_}\\
T^{R^*}(i)\leq  {S^L(i)} & \quad \text{ otherwise}\nonumber
\end{align}
\begin{align}
{^RT}(i+\mu_T-1)\leq S^L(i) & \quad \text{ if $S$ is odd and }  r_T=0\tag{A2$'$}
\label{eq:A2_}\\
{^RT}(i+\mu_T)\leq S^L(i) & \quad \text{ otherwise}\nonumber
\end{align}

\end{definition}

\begin{theorem}[Kwon]
\label{theo:KwonLRTabs}
Let $L=(T_1,T_2,\dots,T_{\ell(\mu)},S)$ be a vertex in $B^{\mathfrak{d}}(\mu)$. Then $L$ is a vertex of $T^\mathfrak{d}(\mu)$ if and only if any pair of successive tableaux in $L$ is admissible.
\end{theorem}

See Figure~\ref{fig:Admissible} for an example.

\begin{figure}
\hspace{1.7cm}\begin{tikzpicture}[scale=0.35]
  \draw (1,6)--(2,6);
  \draw (1,5)--(2,5);
  \draw (0,4)--(2,4);
  \draw (0,3)--(2,3);
  \draw (0,2)--(2,2);
  \draw (0,1)--(1,1);
  \draw (0,0)--(1,0);
  \draw (0,-1)--(1,-1);
  \draw (0,4) -- (0,-1);
  \draw (1,6) -- (1,-1);
  \draw (2,6) -- (2,2);
  \draw (0.5,3.5) node {1};
  \draw (0.5,2.5) node {2};
  \draw (0.5,1.5) node {3};
  \draw (0.5,0.5) node {7};
  \draw (0.5,-0.5) node {8};
  \draw (1.5,5.5) node {1};
  \draw (1.5,4.5) node {2};
  \draw (1.5,3.5) node {3};
  \draw (1.5,2.5) node {6};
  \draw (3.5,8)--(4.5,8);
  \draw (3.5,7)--(4.5,7);   
  \draw (2.5,6)--(4.5,6);
  \draw (2.5,5)--(4.5,5);   
  \draw (2.5,4)--(4.5,4);
  \draw (2.5,3)--(4.5,3);
  \draw (2.5,2)--(4.5,2);
  \draw (2.5,2)--(3.5,2);
  \draw (2.5,1)--(3.5,1);
  \draw (2.5,6) -- (2.5,1);
  \draw (3.5,8) -- (3.5,1);
  \draw (4.5,8) -- (4.5,2);
  \draw (3,5.5) node {1};
  \draw (3,4.5) node {2};
  \draw (3,3.5) node {3};
  \draw (3,2.5) node {4};
  \draw (3,1.5) node {5};
  \draw (4,7.5) node {1};
  \draw (4,6.5) node {2};
  \draw (4,5.5) node {3};
  \draw (4,4.5) node {4};
  \draw (4,3.5) node {5};
  \draw (4,2.5) node {6};
  \draw (9,14)--(10,14);
  \draw (9,13)--(10,13);
  \draw (9,12)--(10,12);
  \draw (9,11)--(10,11);
  \draw (8,10)--(10,10);
  \draw (8,9)--(10,9);
  \draw (5,8)--(10,8);
  \draw (5,7)--(10,7);
  \draw (5,6)--(10,6);
  \draw (5,5)--(10,5);
  \draw (5,4)--(10,4);
  \draw (5,3)--(10,3);
  \draw (5,2)--(10,2);
  \draw (5,8) -- (5,2);
  \draw (6,8) -- (6,2);
  \draw (7,8) -- (7,2);
  \draw (8,10) -- (8,2);
  \draw (9,14) -- (9,2);
  \draw (10,14) -- (10,2);
  \draw (5.5,7.5) node {1};
  \draw (5.5,6.5) node {2};
  \draw (5.5,5.5) node {3};
  \draw (5.5,4.5) node {4};
  \draw (5.5,3.5) node {5};
  \draw (5.5,2.5) node {6};
  \draw (6.5,7.5) node {1};
  \draw (6.5,6.5) node {2};
  \draw (6.5,5.5) node {3};
  \draw (6.5,4.5) node {4};
  \draw (6.5,3.5) node {5};
  \draw (6.5,2.5) node {6};
  \draw (7.5,7.5) node {1};
  \draw (7.5,6.5) node {2};
  \draw (7.5,5.5) node {3};
  \draw (7.5,4.5) node {4};
  \draw (7.5,3.5) node {5};
  \draw (7.5,2.5) node {6};
  \draw (8.5,9.5) node {1};
  \draw (8.5,8.5) node {2};
  \draw (8.5,7.5) node {3};
  \draw (8.5,6.5) node {2};
  \draw (8.5,5.5) node {5};
  \draw (8.5,4.5) node {6};
  \draw (8.5,3.5) node {7};
  \draw (8.5,2.5) node {8};
  \draw (9.5,13.5) node {1};
  \draw (9.5,12.5) node {2};
  \draw (9.5,11.5) node {3};
  \draw (9.5,10.5) node {4};
  \draw (9.5,9.5) node {5};
  \draw (9.5,8.5) node {6};
  \draw (9.5,7.5) node {7};
  \draw (9.5,6.5) node {8};
  \draw (9.5,5.5) node {9};
  \draw (9.5,4.5) node {10};
  \draw (9.5,3.5) node {11};
  \draw (9.5,2.5) node {12};
\end{tikzpicture}\vspace{-0.4cm}

\hspace{0.455cm}\begin{tikzpicture}[scale=0.35]
  \draw (-2,3) node{$T^LT^R=$};
  \draw (1,6)--(2,6);
  \draw (1,5)--(2,5);
  \draw (0,4)--(2,4);
  \draw (0,3)--(2,3);
  \draw (0,2)--(2,2);
  \draw (0,1)--(1,1);
  \draw (0,0)--(1,0);
  \draw (0,-1)--(1,-1);
  \draw (0,4) -- (0,-1);
  \draw (1,6) -- (1,-1);
  \draw (2,6) -- (2,2);
  \draw (0.5,3.5) node {1};
  \draw (0.5,2.5) node {2};
  \draw (0.5,1.5) node {3};
  \draw (0.5,0.5) node {7};
  \draw (0.5,-0.5) node {8};
  \draw (1.5,5.5) node {1};
  \draw (1.5,4.5) node {2};
  \draw (1.5,3.5) node {3};
  \draw (1.5,2.5) node {6};
    \draw[->] (1.5,2) -- (1.5,1);
\end{tikzpicture}
\hspace{-0.2cm}
\begin{tikzpicture}[scale=0.35]
  \draw (-0.6,3) node{$\mapsto$};
  \draw (1,5)--(2,5);
  \draw (0,4)--(2,4);
  \draw (0,3)--(2,3);
  \draw (0,2)--(2,2);
  \draw (0,1)--(2,1);
  \draw (0,0)--(1,0);
  \draw (0,0)--(1,0);
  \draw (0,-1)--(1,-1);
  \draw (0,4) -- (0,-1);
  \draw (1,5) -- (1,-1);
  \draw (2,5) -- (2,1);
  \draw (0.5,3.5) node {1};
  \draw (0.5,2.5) node {2};
  \draw (0.5,1.5) node {3};
  \draw (0.5,0.5) node {7}; 
    \draw[->] (1,0.5) -- (1.5,0.5);
    \draw (0.5,0.5) node[draw, circle, scale=1.2] {};
  \draw (0.5,-0.5) node {8}; 
    \draw[->] (1,-0.5) -- (1.5,-0.5);
    \draw (0.5,-0.5) node[draw, circle, scale=1.2] {};
  \draw (1.5,4.5) node {1};
  \draw (1.5,3.5) node {2};
  \draw (1.5,2.5) node {3};
  \draw (1.5,1.5) node {6};
\end{tikzpicture}
\hspace{-0.2cm}
\begin{tikzpicture}[scale=0.35]
  \draw (-0.6,3) node{$\mapsto$};
  \draw (0,5)--(2,5);
  \draw (0,4)--(2,4);
  \draw (0,3)--(2,3);
  \draw (0,2)--(2,2);
  \draw (1,1)--(2,1);
  \draw (1,0)--(2,0);
  \draw (1,-1)--(2,-1);
  \draw (0,5) -- (0,2);
  \draw (1,5) -- (1,-1);
  \draw (2,5) -- (2,-1);
  \draw (0.5,4.5) node {1};
  \draw (0.5,3.5) node {2};
  \draw (0.5,2.5) node {3};
  \draw (1.5,0.5) node {7};
  \draw (1.5,-0.5) node {8}; 
  \draw (1.5,4.5) node {1};
  \draw (1.5,3.5) node {2};
  \draw (1.5,2.5) node {3};
  \draw (1.5,1.5) node {6};
  \draw (4,3) node{$={^LT}{^RT}$};
\end{tikzpicture}
\begin{tikzpicture}[scale=0.35]
  \draw (-2,3) node{$T^LT^R=$};
  \draw (1,6)--(2,6);
  \draw (1,5)--(2,5);
  \draw (0,4)--(2,4);
  \draw (0,3)--(2,3);
  \draw (0,2)--(2,2);
  \draw (0,1)--(1,1);
  \draw (0,0)--(1,0);
  \draw (0,-1)--(1,-1);
  \draw (0,4) -- (0,-1);
  \draw (1,6) -- (1,-1);
  \draw (2,6) -- (2,2);
  \draw (0.5,3.5) node {1};
  \draw (0.5,2.5) node {2};
  \draw (0.5,1.5) node {3};
  \draw (0.5,0.5) node {7};
  \draw (0.5,-0.5) node {8};
  \draw (1.5,5.5) node {1};
  \draw (1.5,4.5) node {2};
  \draw (1.5,3.5) node {3};
  \draw (1.5,2.5) node {6};
    \draw[->] (0.5,4) -- (0.5,5);
\end{tikzpicture}
\hspace{-0.2cm}
\begin{tikzpicture}[scale=0.35]
  \draw (-0.6,3) node{$\mapsto$};
  \draw (0,6)--(2,6);
  \draw (0,5)--(2,5);
  \draw (0,4)--(2,4);
  \draw (0,3)--(2,3);
  \draw (0,2)--(2,2);
  \draw (0,1)--(1,1);
  \draw (0,0)--(1,0);
  \draw (0,6) -- (0,0);
  \draw (1,6) -- (1,0);
  \draw (2,6) -- (2,2);
  \draw (0.5,5.5) node {1};
  \draw (0.5,4.5) node {2};
  \draw (0.5,3.5) node {3};
  \draw (0.5,1.5) node {7};
  \draw (0.5,0.5) node {8}; 
  \draw (0.5,0);
    \draw[->] (1,2.5) -- (0.5,2.5);
    \draw (1.5,2.5) node[draw, circle, scale=1.2] {};
  \draw (1.5,5.5) node {1};
  \draw (1.5,4.5) node {2};
  \draw (1.5,3.5) node {3};
  \draw (1.5,2.5) node {6};
\end{tikzpicture}
\hspace{-0.2cm}
\begin{tikzpicture}[scale=0.35]
  \draw (-0.6,3) node{$\mapsto$};
  \draw (0,6)--(2,6);
  \draw (0,5)--(2,5);
  \draw (0,4)--(2,4);
  \draw (0,3)--(2,3);
  \draw (0,2)--(1,2);
  \draw (0,1)--(1,1);
  \draw (0,0)--(1,0);
  \draw (0,6) -- (0,0);
  \draw (1,6) -- (1,0);
  \draw (2,6) -- (2,3);
  \draw (0.5,5.5) node {1};
  \draw (0.5,4.5) node {2};
  \draw (0.5,3.5) node {3};
  \draw (0.5,2.5) node {6};
  \draw (0.5,1.5) node {7};
  \draw (0.5,0.5) node {8}; 
  \draw (0.5,0);
  \draw (1.5,5.5) node {1};
  \draw (1.5,4.5) node {2};
  \draw (1.5,3.5) node {3};
  \draw (4.5,3) node{$=T^{L^*}T^{R^*}$};
\end{tikzpicture}

\hspace{0.455cm}
\begin{tikzpicture}[scale=0.35]
  \draw (0.5,3) node{$U^LU^R=$};
  \draw (3.5,8)--(4.5,8);
  \draw (3.5,7)--(4.5,7);   
  \draw (2.5,6)--(4.5,6);
  \draw (2.5,5)--(4.5,5);   
  \draw (2.5,4)--(4.5,4);
  \draw (2.5,3)--(4.5,3);
  \draw (2.5,2)--(4.5,2);
  \draw (2.5,1)--(3.5,1);
  \draw (2.5,6) -- (2.5,1);
  \draw (3.5,8) -- (3.5,1);
  \draw (4.5,8) -- (4.5,2);
  \draw (3,5.5) node {1};
  \draw (3,4.5) node {2};
  \draw (3,3.5) node {3};
  \draw (3,2.5) node {4};
  \draw (3,1.5) node {5};
  \draw (4,7.5) node {1};
  \draw (4,6.5) node {2};
  \draw (4,5.5) node {3};
  \draw (4,4.5) node {4};
  \draw (4,3.5) node {5};
  \draw (4,2.5) node {6};
    \draw[->] (4,2) -- (4,1);
\end{tikzpicture}
\hspace{-0.2cm}
\begin{tikzpicture}[scale=0.35]
  \draw (1.9,3) node{$\mapsto$};
  \draw (3.5,7)--(4.5,7);   
  \draw (2.5,6)--(4.5,6);
  \draw (2.5,5)--(4.5,5);   
  \draw (2.5,4)--(4.5,4);
  \draw (2.5,3)--(4.5,3);
  \draw (2.5,2)--(4.5,2);
  \draw (2.5,1)--(4.5,1);
  \draw (2.5,6) -- (2.5,1);
  \draw (3.5,7) -- (3.5,1);
  \draw (4.5,7) -- (4.5,1);
  \draw (3,5.5) node {1};
  \draw (3,4.5) node {2};
  \draw (3,3.5) node {3};
  \draw (3,2.5) node {4};
  \draw (3,1.5) node {5};
  \draw (4,6.5) node {1};
  \draw (4,5.5) node {2};
  \draw (4,4.5) node {3};
  \draw (4,3.5) node {4};
  \draw (4,2.5) node {5};
  \draw (4,1.5) node {6};
    \draw[->] (3,6) -- (3,7);
\end{tikzpicture}
\hspace{-0.2cm}
\begin{tikzpicture}[scale=0.35]
  \draw (1.9,3) node{$\mapsto$};
  \draw (2.5,7)--(4.5,7);   
  \draw (2.5,6)--(4.5,6);
  \draw (2.5,5)--(4.5,5);   
  \draw (2.5,4)--(4.5,4);
  \draw (2.5,3)--(4.5,3);
  \draw (2.5,2)--(4.5,2);
  \draw (3.5,1)--(4.5,1);
  \draw (2.5,7) -- (2.5,2);
  \draw (3.5,7) -- (3.5,1);
  \draw (4.5,7) -- (4.5,1);
  \draw (3,6.5) node {1};
  \draw (3,5.5) node {2};
  \draw (3,4.5) node {3};
  \draw (3,3.5) node {4};
  \draw (3,2.5) node {5};
  \draw (4,6.5) node {1};
  \draw (4,5.5) node {2};
  \draw (4,4.5) node {3};
  \draw (4,3.5) node {4};
  \draw (4,2.5) node {5};
  \draw (4,1.5) node {6};
  \draw (6.5,3) node{$={^LU}{^RU}$};
\end{tikzpicture}
\hspace{-0.07cm}
\begin{tikzpicture}[scale=0.35]
  \draw (0.5,3) node{$U^LU^R=$};
  \draw (3.5,8)--(4.5,8);
  \draw (3.5,7)--(4.5,7);   
  \draw (2.5,6)--(4.5,6);
  \draw (2.5,5)--(4.5,5);   
  \draw (2.5,4)--(4.5,4);
  \draw (2.5,3)--(4.5,3);
  \draw (2.5,2)--(4.5,2);
  \draw (2.5,1)--(3.5,1);
  \draw (2.5,6) -- (2.5,1);
  \draw (3.5,8) -- (3.5,1);
  \draw (4.5,8) -- (4.5,2);
  \draw (3,5.5) node {1};
  \draw (3,4.5) node {2};
  \draw (3,3.5) node {3};
  \draw (3,2.5) node {4};
  \draw (3,1.5) node {5};
  \draw (4,7.5) node {1};
  \draw (4,6.5) node {2};
  \draw (4,5.5) node {3};
  \draw (4,4.5) node {4};
  \draw (4,3.5) node {5};
  \draw (4,2.5) node {6};
    \draw[->] (3,6) -- (3,7);
\end{tikzpicture}
\hspace{-0.2cm}
\begin{tikzpicture}[scale=0.35]
  \draw (1.9,3) node{$\mapsto$};
  \draw (2.5,8)--(4.5,8);
  \draw (2.5,7)--(4.5,7);   
  \draw (2.5,6)--(4.5,6);
  \draw (2.5,5)--(4.5,5);   
  \draw (2.5,4)--(4.5,4);
  \draw (2.5,3)--(4.5,3);
  \draw (2.5,2)--(4.5,2);
  \draw (2.5,1);
  \draw (2.5,8) -- (2.5,2);
  \draw (3.5,8) -- (3.5,2);
  \draw (4.5,8) -- (4.5,2);
  \draw (3,7.5) node {1};
  \draw (3,6.5) node {2};
  \draw (3,5.5) node {3};
  \draw (3,4.5) node {4};
  \draw (3,3.5) node {5};
  \draw (4,7.5) node {1};
  \draw (4,6.5) node {2};
  \draw (4,5.5) node {3};
  \draw (4,4.5) node {4};
  \draw (4,3.5) node {5};
  \draw (4,2.5) node {6};
    \draw[->] (3.5,2.5) -- (3,2.5);
    \draw (4,2.5) node[draw, circle, scale=1.2] {};
\end{tikzpicture}
\hspace{-0.2cm}
\begin{tikzpicture}[scale=0.35]
  \draw (1.9,3) node{$\mapsto$};
  \draw (2.5,8)--(4.5,8);
  \draw (2.5,7)--(4.5,7);   
  \draw (2.5,6)--(4.5,6);
  \draw (2.5,5)--(4.5,5);   
  \draw (2.5,4)--(4.5,4);
  \draw (2.5,3)--(4.5,3);
  \draw (2.5,2)--(3.5,2);
  \draw (2.5,1);
  \draw (2.5,8) -- (2.5,2);
  \draw (3.5,8) -- (3.5,2);
  \draw (4.5,8) -- (4.5,3);
  \draw (3,7.5) node {1};
  \draw (3,6.5) node {2};
  \draw (3,5.5) node {3};
  \draw (3,4.5) node {4};
  \draw (3,3.5) node {5};
  \draw (3,2.5) node {6};
  \draw (4,7.5) node {1};
  \draw (4,6.5) node {2};
  \draw (4,5.5) node {3};
  \draw (4,4.5) node {4};
  \draw (4,3.5) node {5};
  \draw (7,3) node{$=U^{L^*}U^{R^*}$};
\end{tikzpicture}
\caption{An orthogonal Littlewood-Richardson tableau in $\LRtabs$ with $n=5$, $k=3$, $\lambda=(12,8,8,6,6,6,6,3,3)$ and $\mu=(3,1)$. We calculated the columns we need to prove its admissibility.}
\label{fig:Admissible}
\end{figure}
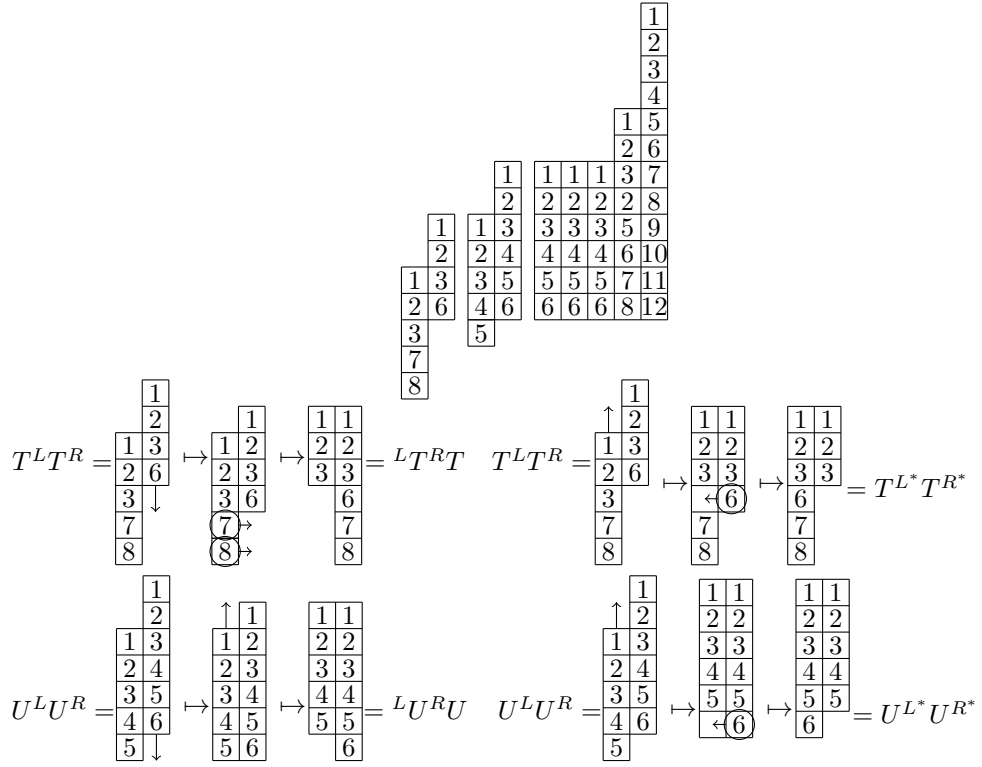

\begin{remark}
\label{rem:LRTabsSubSet}
Let $L\in\LRtabs$ be an orthogonal Littlewood-Richardson tableau. Moreover let  $\tilde{L}=(T_j,T_{j+1},\dots,T_{\ell(\mu)},S)$ be the tableau, which is obtained from $L$ by deleting the first $j-1$ semistandard tableaux. Due to Theorem~\ref{theo:KwonLRTabs} $\tilde{L}$ is an orthogonal Littlewood-Richardson tableau in $\mathrm{LR}_{\lambda}^{\tilde{\mu}}(\mathfrak{d})$, where $\tilde{\mu}=(\mu_j,\mu_{j+1},\dots\mu_{\ell(\mu)})$.
\end{remark}

We give now an explicit description of Kwon's orthogonal Littlewood-Richardson tableaux. For it we need the concept of gaps and slots.

\begin{definition}
Let $T$ be a semistandard tableau. A position $j>1$ of $T$ is a \emph{gap} if $j-1$ is not in the same column as $j$. A position $j>0$ of $T$ is a \emph{slot} if $j+1$ is not in the same column as $j$.
\end{definition}

Note that above a gap there is either a slot or nothing and below a slot there is either a gap or nothing. In the first tableau of Figure~\ref{fig:Admissible} the $3$ and the $8$ in the  first column and the $3$ in the second column are slots, while the $7$ in the first column is a gap and the $6$ in the second column is both, a gap and a slot.

\begin{theorem}
\label{theo:LRTabs}
Let $\lambda\vdash r$, $\ell(\lambda)\leq n(=2k+1)$, $\ell(\mu)\leq k$. Let $L=(T_1,T_2,\dots,T_{\ell(\mu)},S)$ be a vertex in $B^{\mathfrak{d}}(\mu)$. Then $L$ is an orthogonal Littlewood-Richardson tableau in $\LRtabs$ for $\SO(n)$ if and only if for all $i$ there are $\lambda_i'$ $i$'s in $L$ and the following conditions are met:
\begin{enumerate}
\item[{(H)}\namedlabel{eq:HeightProperty}{H}]$b_i\leq b_{i+1}-a_{i+1}+2r_ir_{i+1}$ for $1\leq i \leq \ell(\mu)-1$.
\item[{(H$'$)}\namedlabel{eq:Height_Property}{H$'$}] $b_{\ell(\mu)}\leq ht(S^L)$ if $S$ is even and $
b_{\ell(\mu)}\leq ht(S^L)-1+2r_{\ell(\mu)}$  if $S$ is odd.
\item[{(S)}\namedlabel{eq:SProperty}{S}] $S$ contains no gap.
\item[{(T1)}\namedlabel{eq:TypeProperty}{T1}] Tableaux $T_1,T_2,\dots,T_{\ell(\mu)}$ are of one of the following three types.
\begin{enumerate}
\item Type 1 tableaux have residuum 0. Gaps can be only in the tail.
\item Type 2 tableaux have residuum 1. Gaps can be only in the lower tail.
\item Type 3 tableaux have residuum 1. The fin is a gap. Other gaps can be only in the lower tail.
\end{enumerate}
If $T_i$ is of type 3, $i<\ell(\mu)$, $T_{i+1}$ has residuum 1 and the fin of $T_i$ is not larger than the fin of $T_{i+1}$. If $T_{\ell(\mu)}$ is of type 3, $S$ is odd.

If $T_{\ell(\mu)}$ is of type 1 and $S$ is odd, the tail root is smaller than or equal to $S^L(1)$, the bottommost position in the first column of $S$.
\item[{(T2)}\namedlabel{eq:TailProperty}{T2}] The tails shifted together such that they share the top line form a semistandard Young tableau.
\item[{(G)}\namedlabel{eq:GapProperty}{G}] For each gap $j$ there is a slot $j-1$ in a column to the right. This can be in the same Tableau $T_i$ or in another one that is right of $T_i$ in $L$ including $S$. More precisely, if there are $m$ gaps $j$ there are $m$ slots $j-1$ such that we can build pairs of a gap and a slot such that each slot is to the right of its gap.
\end{enumerate}
\end{theorem}

\begin{remark}
Properties \eqref{eq:H} and \eqref{eq:HeightProperty}, as well as  Properties  \eqref{eq:H_} and \eqref{eq:Height_Property} are just reformulations of each other. That is why we named them identically.
\end{remark}

\begin{lemma}
\label{lem:TensorEpsilonZero}
Let $L$ be in $B^{\mathfrak{d}}(\mu)$. Then $\varepsilon_i(L)=0 \text{ for } i \neq 0$ if and only if \eqref{eq:SProperty} and \eqref{eq:GapProperty}.
\end{lemma}

\begin{proof}
If and only if $\varepsilon_j(C)>0$ a column $C$ contains a gap $j$. In this case $\varepsilon_j(C)=1$. On the other hand if and only if $\varphi_j(C)>0$ a column $C$ contains a slot $j$. In this case $\varphi_j(C)=1$.

The tensor product tells us $\varepsilon_j(b\otimes b')=\varepsilon_j(b')+\max(0,(\varepsilon_j(b)-\varphi_j(b')))$ and therefore $\varepsilon_j(b\otimes b')\geq \varepsilon_j(b')$.

For a a tensor product consisting of several columns to have $\varepsilon_j=0$ this means that the first column needs to contain no gap and \eqref{eq:GapProperty}. Because $S$ is a skew semistandard tableau and the rightmost column has no gaps, it cannot have gaps, because slots to the right are to big.
\end{proof}

\begin{remark}
This also shows that the filling of such a tableau is a partition.
\end{remark}

\begin{lemma}
\label{lem:TypeProperty}
Let $L=(T_1, T_2, \dots, T_{\ell(\mu)}, S)$ be a tableau in $B^{\mathfrak{d}}(\mu)$ such that \eqref{eq:H}, \eqref{eq:H_}, \eqref{eq:SProperty} and \eqref{eq:GapProperty} hold. Then if and only if \eqref{eq:A1} and \eqref{eq:A1_} hold also \eqref{eq:TypeProperty} without the tail root condition for residuum zero tableaux holds.
\end{lemma}

\begin{proof}
We first show inductively that the following two statements hold if and only if \eqref{eq:A1} and \eqref{eq:A1_} hold.
\begin{itemize}
\item Suppose $T_i$ has residuum 1.
\begin{itemize}
\item Then $T_i^{R^*}$ is $T_i^R$ without the fin and ${^LT_i}$ is $T_i^L$ without the lower tail. The tail root is not a gap.
\item There is no slot smaller than the bottommost position of $T_i^{R^*}$ in $T_i^R$ or to the right. There is no slot smaller than the tail root in $T_i^L$ or to the right.
\item If the fin is a gap, then $T_{i+1}$ has also residuum 1, or if $i=\ell(\mu)$, $S$ is odd.
\end{itemize}
\item Suppose $T_i$ has residuum 0.
\begin{itemize}
\item Then ${^LT_i}$ is $T_i^L$ without the tail.
\item There is no slot smaller than the fin in $T_i^R$ or to the right. There is no slot smaller than the bottommost position of ${^LT_i}$ in $T_i^L$ or to the right.
\end{itemize}
\end{itemize}
This implies that there are no gaps at and above the positions in question, because slots to the right are to big.

In the base case $L=S$ we can argue that this is equivalent to $S$ being a skew semistandard tableau.

In the induction step we consider $T_1$. (Compare with Remark~\ref{rem:LRTabsSubSet}.) If $T_1$ has residuum 1, it holds that:
\begin{itemize}
\item $T^{R^{*}}_1$ contains one position less than $T^R_1$. Let us call this position $l_1$. Suppose that $l_1$ is not the fin. In this case there exists a position $l_3$ directly below $l_1$. As $l_1$ is not in $T^{R^{*}}_1$, there exists a position $l_2$ in $T^{L}_1$, that is shifted next to $l_3$ when determining $T^{R^{*}}_1$. Therefore $l_1<l_2\leq l_3$.
If $l_2-1$ is in $T_1^L$, it is at most one position above $l_2$, thus directly besides $l_1$, which is a contradiction. Therefore $l_2$ is a gap. If $l_2-1=l_1$, either $l_3$ is a gap or $l_1$ is no slot. Thus either $l_2$ or $l_3$ is a gap with no slot in $T_1$. However $l_3$ is in $T^{R^{*}}_1$ and therefore smaller than or equal to the bottommost position of ${^LT_2}$ (or $S^L$ if $\ell(\mu)=1$, respectively). We have seen by induction that there are no smaller slots to the right. This is a contradiction.
\item The bottommost position of $T^{R^{*}}_1$ (the position above the fin) is smaller than or equal to the bottommost position of ${^LT_2}$ (or $S^L$ if $\ell(\mu)=1$, respectively). Thus there is no slot that is small enough for this position or one above to be a gap.
\item Because $T^{R^{*}}_1$ is $T^R_1$ without the fin, the tail root is shifted above the fin when calculating $T_1^{R^*}$. Therefore it is smaller than or equal to the bottommost position of $T_1^{R^*}$. By the same argumentation as above, neither it nor a position above is a gap and no slot is smaller than it.
\item If we consider the procedure to obtain ${^LT_1}$ we see that the fin is placed besides the tail root due to residuum 1, and therefore only the lower tail is shifted right.
\end{itemize}
If $T_1$ has residuum 0, it holds that:.
\begin{itemize}
\item The fin is smaller than or equal to the smallest slot to the right. Therefore it is no gap and there are no gaps above. The same holds for the position above the tail root.
\item Due to residuum 0, nothing is shifted besides the tail root when calculating ${^LT_1}$, thus the whole tail changes column.
\end{itemize}

On the other hand, if those statements hold, the inequalities that hold for the bottommost positions of the considered columns, the column strictness and the lack of gaps imply \eqref{eq:A1} and \eqref{eq:A1_}.

We prove now that those statements hold if and only if \eqref{eq:TypeProperty} without the tail root condition for residuum zero tableaux holds. The statements about the slots imply where gaps are. On the other hand, if the gaps are where they are described in \eqref{eq:TypeProperty} and \eqref{eq:HeightProperty} and \eqref{eq:Height_Property} hold, then we also get the inequalities between the slots in question. Finally the statements about ${^LT_i}$ and $T^{R*}_i$ follow from the residuum and the places where a gap can be.
\end{proof}

\begin{lemma}
\label{lem:TailProperty}
Let $L=(T_1, T_2, \dots, T_{\ell(\mu)}, S)$ be a tableau in $B^{\mathfrak{d}}(\mu)$ such that \eqref{eq:H}, \eqref{eq:H_}, \eqref{eq:SProperty}, \eqref{eq:GapProperty} and \eqref{eq:TypeProperty} without the tail root condition for residuum zero tableaux hold. Then if and only if \eqref{eq:A2} and \eqref{eq:A2_} hold also \eqref{eq:TailProperty} and the tail root condition for residuum zero tableaux of \eqref{eq:TypeProperty} hold.
\end{lemma}

\begin{proof}
Due to what we have seen before about ${^LT_i}$ and $T^{R^*}_i$ this holds once we argue, that for residuum 1 tableaux the tail root is smaller than or equal to the fin.

The tail root condition for residuum zero tableaux and $S$ odd is equivalent to the second condition of \eqref{eq:A2_}.
\end{proof}

Now Theorem~\ref{theo:LRTabs} follows directly from Lemmas~\ref{lem:TensorEpsilonZero}, \ref{lem:TypeProperty} and~\ref{lem:TailProperty}.

We finish this section by proving further properties about orthogonal Littlewood-Richardson tableaux we will use later on.

\begin{proposition}
\label{prop:LRtabs1}
If $T_i$ is of type 2 or 3 the tail root is a slot. 
\end{proposition}

\begin{proof}
We have seen in the proof of Lemma~\ref{lem:TypeProperty}, that the tail root is strictly smaller than the fin. Since the residuum is exactly $1$, the entry below the tail root, if it exists, is larger than the fin.
\end{proof}

\begin{proposition}
\label{prop:LRtabs2}
If the fin of a tableau $T_i$ exists, it is even and not larger than the fin of $T_{i+1}$, which then also exists.
\end{proposition}

\begin{proof}
The fin of $T_i$ is even for type $1$ or $2$, as $T_i^R$ has no gap and even length. We show for these cases that the fin is smaller than or equal to the fin of $T_{i+1}$. If $T_{i}$ or $T_{i+1}$ is of type $1$, $T_{i+1}^L$ without tail is at least as long as $T_i^R$ by \eqref{eq:H}. Therefore, as $a_{i+1}\geq 0$, also $T_{i+1}^R$ is at least as long as $T_i^R$. If both tableaux have residuum $1$, $T_{i+1}^L$ without tail plus $2$ is at least as long as $T_i^R$ by \eqref{eq:H} and $T_{i+1}^R$ is longer than $T_{i+1}^L$ by at least $2$. The claim follows as the fin of $T_i$ is equal to the length of $T_i^R$ and the fin of $T_{i+1}$ is larger than or equal to the length of $T_{i+1}^R$.

If $T_i$ is of type $3$, we know that the fin of $T_i$ is not larger than the fin of $T_{i+1}$ by assumption. We show for this case that the fin is even. We do so by showing that any possible slot is odd.

Let $T_{j}$ be the next tableau of residuum $0$ to the right of $T_i$, if this exists, or $T_{\ell(\mu)}$, otherwise. Tableaux between $T_i$ and $T_j$ are therefore of type $3$ or $2$. Tableaux of type 3 have at least two odd slots, namely the position above the fin and the tail root. Tableaux of type 2 have at least one odd slot, namely the tail root. Other slots need to be at least as large as the fin. Therefore slots between $T_j$ and $T_i$, that are small enough for the fin of one of those tableaux or $T_i$ to be their slot, are also odd.

It remains to show, that there is no even slot right of $T_j$ (and in $T_j$ if it is of type 1), that is small enough for any fin of $T_i$ or a tableau between $T_i$ and $T_j$ to be its gap.

If $T_j$ is of type $2$ or $3$, it is directly left of $S$. As $S$ contains no gap by \eqref{eq:SProperty}, slots in $S$ are in the bottom line. If $S$ is odd, the slots of $S$ are also odd. If $S$ is even, $T_j$ is of type $2$, due to \eqref{eq:TypeProperty} and any slot of $S$ is larger than the fin of $T_2$ due to \eqref{eq:Height_Property}.

If $T_j$ is of type $1$, slots of $T_j$ are at least as large as the fin of $T_{j-1}$ due to \eqref{eq:HeightProperty} and because the fin of $T_{j-1}$ is not a gap, as it is of type 2 \eqref{eq:TypeProperty}. Due to \eqref{eq:HeightProperty} and \eqref{eq:Height_Property} (and because gaps are the fin or in the tail by \eqref{eq:TypeProperty}) this also holds for slots further to the right.
\end{proof}

\section{Alternative Orthogonal Littlewood-Richardson Tableaux}
\label{sec:BijA}

In this section we define an alternative set of Littlewood-Richardson tableaux in terms of skew tableaux.

Moreover define a bijection (Bijection $A$) between Kwon's orthogonal Littlewood-Richardson ta\-bleaux and our new tableaux. We will use our new set of tableaux in the main bijection (Bijection $B$) to map pairs consisting of a standard Young tableau and a Littlewood-Richardson tableau to a vacillating tableau.

\subsection{Definition and Examples}

\begin{definition}
\label{def:aoLRT}
We define the set of alternative orthogonal Littlewood-Richardson tableaux  $\LRtabsa$ as follows. A tableau $L\in\LRtabsa$ is a reverse skew semistandard tableau of inner shape $\lambda$ and type $\mu$ (thus the filling consists of $\mu_j$ $j$'s, for all $j$). The outer shape has $2k+1$ possibly empty rows, whose lengths have all the same parity.
The following two properties are satisfied.

\begin{enumerate}
\item The reading word is a Yamanouchi word. This is satisfied if and only if the $j$th cell from left, labeled $i$ is above the $j$th cell from left labeled $i-1$ for all $i>1$.

\item We go through the reading word of $L$ from right to left. Let $p$ be the current position. We define a sequence $v_p$ of positions of the reading word. The first entry of $v_p$ is $p$. If $m-1$ entries of $v_p$ are defined, let $e$ be entry number $m-1$. We search now for entry number $m$. For that we consider entries whose letter is larger than the letter of $e$ and which are in exactly $m-1$ sequences of positions right of $p$ (thus sequences already defined). If this set is nonempty we search for the smallest letter in it and take the rightmost position with this letter as entry $m$. If it is empty, $v_p$ has no more entries.

Let $r_p$ be the row $p$ is in.
Now we define the value $o_p$ to be the number of entries in $v_p$  with the following properties. It is the rightmost occurrence of its letter and if number $m$ in $v_p$ all $v_{\tilde{p}}$ with $\tilde{p}\neq p$ in the same row as $p$, have at most $m-1$ entries.

We require $r_p\geq 2 |v_p| - o_p$.
\end{enumerate}
\end{definition}

\begin{example}
\label{ex:aoLRT}
Two alternative orthogonal Littlewood-Richardson tableaux. 

\hspace{2cm}
\begin{tikzpicture}[scale=0.35]
  \draw (4.5,5)--(8.5,5);
  \draw (4.5,4)--(8.5,4);
  \draw (4.5,3)--(8.5,3);
  \draw (4.5,2)--(8.5,2);
  \draw (4.5,1)--(6.5,1);
  \draw (4.5,0)--(6.5,0);
  \draw (4.5,5) -- (4.5,0);
  \draw (5.5,5) -- (5.5,0);
  \draw (6.5,5) -- (6.5,0);
  \draw (7.5,5) -- (7.5,2);
  \draw (8.5,5) -- (8.5,2);
  \draw[blue] (5,0.5) node {1};
  \draw[blue] (6,0.5) node {1};
  \draw[blue] (8,2.5) node {1};
  \draw[violet] (6,1.5) node {2};
  \draw[violet] (7,2.5) node {2};
  \draw[red] (6,2.5) node {3};
  \draw (5,-1);
\end{tikzpicture}
\hspace{2cm}
\begin{tikzpicture}[scale=0.35]
  \draw (0,7) -- (8,7);
  \draw (0,6) -- (8,6);
  \draw (0,5) -- (8,5);
  \draw (0,4) -- (8,4);
  \draw (0,3) -- (8,3);
  \draw (0,2) -- (6,2);
  \draw (0,1) -- (4,1);
  \draw (0,7) -- (0,1);
  \draw (1,7) -- (1,1);
  \draw (2,7) -- (2,1);
  \draw (3,7) -- (3,1);
  \draw (4,7) -- (4,1);
  \draw (5,7) -- (5,2);
  \draw (6,7) -- (6,2);
  \draw (7,7) -- (7,3);
  \draw (8,7) -- (8,3);
  \draw[red] (4.5,3.5) node {3};
  \draw[red] (6.5,4.5) node {3};
  \draw[red] (7.5,4.5) node {3};
  \draw[violet] (4.5,2.5) node {2};
  \draw[violet] (5.5,3.5) node {2};
  \draw[violet] (6.5,3.5) node {2};
  \draw[violet] (7.5,3.5) node {2};
  \draw[blue] (0.5,1.5) node {1};
  \draw[blue] (1.5,1.5) node {1};
  \draw[blue] (2.5,1.5) node {1};
  \draw[blue] (3.5,1.5) node {1};
  \draw[blue] (5.5,2.5) node {1};
      \draw (0,1);
\end{tikzpicture}

We write the reading word as a sequence of entries $l_p$ where $l$ is the letter and $p$ counts the position. The reading words are: $(1_1,1_2,2_3,3_4,2_5,1_6)$ and $(1_1,1_2,1_3,1_4,2_5,1_6,3_7,2_8,2_9,2_{10},3_{11},3_{12})$

Then we have the following $v$'s, where rows are separated by semicolons: $(1_6)$, $(2_5)$, $(3_4)$; $(2_3,3_4)$; $(1_2,2_5,3_4)$, $(1_1,2_3)$ and $(3_{12})$, $(3_{11})$; $(2_{10},3_{12})$, $(2_9,3_{11})$, $(2_8)$, $(3_7)$; $(1_6,2_{10},3_{12})$, $(2_5,3_7)$; $(1_4,2_9,3_{11})$, $(1_3,2_8,3_7)$, $(1_2,2_5)$, $(1_1)$.

$1_6$ in the second tableau is in row $5$, which is fine as $3_{12}$ is counted by $o$.
\end{example}

\begin{proposition}
\label{prop:2ndPropYamanouchi}
We can obtain the sequences $v_e$ by using Robinson-Schensted on the reversed reading word of $L$. In particular $v_e$ can be defined as the set of elements that got bumped during the insertion process of $e$.

Therefore, by Theorem~\ref{theo:YamanouchiWordSSYT}, the first property is satisfied if and only if the tableau one obtains by Robinson-Schensted on the reversed reading word is of the form as described in~\ref{theo:YamanouchiWordSSYT}. This is satisfied if and only if every element $j$ is bumped exactly $j$ times. In terms of our $v_e$'s this means that every $j$ is in exactly $j$ $v_e$'s.
\end{proposition}

\begin{proof}
We show inductively that a position gets bumped if and only if it is in the current $v_e$. Therefore elements in the $j$-row were in $j$ $v_e$'s before.

For the base case we consider the first element of an $v_e$. This is always the one we are inserting. Thus it ends up in the first row. On the other hand an element that ends up in the first row, does so only during the insertion process of itself, thus when it is the first element of an $v_e$.

Now if an element is in $j$ different $v_e$'s, by induction hypothesis it got bumped $j$ times thus it is now in row $j$. Now if it is element number $j+1$ in a $v_e$, it is the rightmost one of the smallest letter that is larger than the letter of element number $j$. As elements of the same value get inserted into a row from left to right in Robinson-Schensted, this is the rightmost element in the reading word. The same observation leads to the other direction.
\end{proof}

\subsection{Formulation of Bijection $A$}

Bijection $A$ is formulated by Algorithm~\ref{alg:1}. Its inverse is formulated by Algorithm~\ref{alg:u1}. It maps an orthogonal Littlewood-Richardson tableau of Kwon in $\LRtabs$ to an alternative orthogonal Littlewood-Richardson tableau in $\LRtabsa$.

\begin{algorithm}[h]
\label{alg:1}
\SetKwInOut{Input}{input}\SetKwInOut{Output}{output}
\Input{orthogonal Littlewood-Richardson tableau $L=(T_1,T_2,\dots,T_{\ell(\mu)},S)\in \LRtabs$}
\Output{alternative orthogonal Littlewood-Richardson tableau $\tilde{L} \in \LRtabsa$}
let $\tilde{L}$ be the Young diagram of $S$, reflected on $y=x$\;
\For{$i={\ell(\mu)},{\ell(\mu)}-1,\dots,1$}
  {
  for each $l$ in $T_i$ add an empty cell into column $l$ of $\tilde{L}$\;
  \lIf{$T_i$ has Type 1}{add below of cells coming from the tail of $T_i$ a cell with entry $i$}
  \lElse{add below of cells coming from the lower tail and the fin of $T_i$ a cell with entry $i$}
  sort each column such that empty cells are on top and entries are weakly decreasing\;
  \For{rows $r$ from top to bottom}
    {
    \While(\tcc*[f]{merge}){there is a $j$ left of an $l$ such that $j<l$ in $r$}
      {
      put the rightmost such $l$ and the $i$ from the same column one column to the left\;
      shift cells that were below $l$ upwards and sort the column to the left such that entries are weakly decreasing\;}
      
    \While(\tcc*[f]{shift}){in the row below of $r$ are elements with no cell to the left}
    {
    shift those elements and their lower neighbors one column to the left\;
      }
    }
  \While(\tcc*[f]{correct parity}){not all rows have the parity of $S$}
  {
  shift the rightmost $i$ of the bottommost row with different parity as $S$ to the next such row above\;
  }
}
\Return $\tilde{L}$\;
\caption{Orthogonal Littlewood-Richardson Tableaux: obtaining the alternative}
\end{algorithm}

\begin{algorithm}[h]
\label{alg:u1}
\SetKwInOut{Input}{input}\SetKwInOut{Output}{output}
\Input{alternative orthogonal Littlewood-Richardson tableau $\tilde{L}\in\LRtabsa$}
\Output{orthogonal Littlewood-Richardson tableau $L=(T_1,T_2,\dots,T_{\ell(\mu)},S)\in\LRtabs$}
\For{$i=1,2,\dots,\ell(\mu)$}
  {
  \If{there is an odd number of $i$'s in a row that is not row $2i+1$}
    {\For(\tcc*[f]{correct parity}){all such rows}{put the rightmost such $i$ one row below\;}}
  \For{rows $r$ from bottom to top}
    {
    \While{there are $i$'s or vertically neighbored pairs $i,j$ that can be shifted to the right such that there is still a cell directly above them and $i$ or $j$ is in $r$}
    {shift them one column to the right\tcc*{shift}}
    \While(\tcc*[f]{merge}){there are $i<j_1<j_2$ in a column such that: the column to the right is shorter by at least two such that there exists a cell in the row above $j_2$, contains no $i$ and no $j_2$; $j_3$, the position right of $j_2$, satisfies $j_2>j_1>j_3$ if it exists; $j_2$ is the topmost position in its column satisfying this and $j_2$ is in $r$}
      {put $i$ and $j_2$ one to the column to the right\;
      shift cells below $j_2$ upwards and sort the column to the right such that entries are weakly increasing\;}
    }
  \lFor{$i$ in $\tilde{L}$}
    {mark an unmarked empty cell in the same column, delete $i$ and its cell}
  \For{$l$; marked cells in column $l$}
  {insert a cell labeled $l$ into the tail of the two rowed tableau $T_i$; delete it in $\tilde{L}$\; shift the remaining cells upwards\;}
  \If{row $(2i+1)$ or $(2i)$ are non-empty}
    {for each cell in column $l$ in row $(2i+1)$ (respectively $(2i)$) insert a cell labeled $l$ to the first (respectively second) column of $T_i$ such that they are sorted increasingly\;}
  \lIf{both new columns are of odd length (without tail)}
    {put the topmost tail position to the right column, shift the left column one position down}
  }
 
reflect $\tilde{L}$ by $x=y$ and fill each column with $1,2,\dots$ to obtain $S$\;
let $L$ be $(T_1,T_2,\dots,T_{\ell(\mu)},S)$ and \Return{($L$, $Q$)}
\caption{Orthogonal Littlewood-Richardson Tableaux: obtaining the original}
\end{algorithm}

\subsection{Examples explaining Bijection $A$}
\begin{example}
We consider an orthogonal Littlewood-Richardson tableau and apply Algorithm~\ref{alg:1}.

\noindent \begin{tikzpicture}[scale=0.35]
  \draw (1,6)--(2,6);
  \draw (1,5)--(2,5);
  \draw (0,4)--(2,4);
  \draw (0,3)--(2,3);
  \draw (0,2)--(2,2);
  \draw (0,1)--(1,1);
  \draw (0,0)--(1,0);
  \draw (0,4) -- (0,0);
  \draw (1,6) -- (1,0);
  \draw (2,6) -- (2,2);
  \draw (0.5,3.5) node {1};
  \draw (0.5,2.5) node {2};
  \draw (0.5,1.5) node {3};
  \draw[blue] (0.5,0.5) node {7};
  \draw (1.5,5.5) node {1};
  \draw (1.5,4.5) node {2};
  \draw (1.5,3.5) node {3};
  \draw[blue] (1.5,2.5) node {6};
  \draw (3.5,8)--(4.5,8);
  \draw (3.5,7)--(4.5,7);   
  \draw (2.5,6)--(4.5,6);
  \draw (2.5,5)--(4.5,5);   
  \draw (2.5,4)--(4.5,4);
  \draw (2.5,3)--(4.5,3);
  \draw (2.5,2)--(4.5,2);
  \draw (2.5,2)--(3.5,2);
  \draw (2.5,1)--(3.5,1);
  \draw (2.5,6) -- (2.5,1);
  \draw (3.5,8) -- (3.5,1);
  \draw (4.5,8) -- (4.5,2);
  \draw (3,5.5) node {1};
  \draw (3,4.5) node {2};
  \draw (3,3.5) node {3};
  \draw (3,2.5) node {4};
  \draw (3,1.5) node {5};
  \draw (4,7.5) node {1};
  \draw (4,6.5) node {2};
  \draw (4,5.5) node {3};
  \draw (4,4.5) node {4};
  \draw (4,3.5) node {5};
  \draw[violet] (4,2.5) node {6};
  \draw (5,8)--(6,8);
  \draw (5,7)--(6,7);
  \draw (5,6)--(6,6);
  \draw (5,5)--(6,5);
  \draw (5,4)--(6,4);
  \draw (5,3)--(6,3);
  \draw (5,2)--(6,2);
  \draw (5,8) -- (5,2);
  \draw (6,8) -- (6,2);
  \draw (5.5,7.5) node {1};
  \draw (5.5,6.5) node {2};
  \draw (5.5,5.5) node {3};
  \draw (5.5,4.5) node {4};
  \draw (5.5,3.5) node {5};
  \draw (5.5,2.5) node {6};
\end{tikzpicture}
\hspace{-0.1cm}
\begin{tikzpicture}[scale=0.35]
 \draw (0,0) node{};
 \draw (0,8) node{};
 \draw [->,decorate,
decoration={snake,amplitude=.4mm,segment length=2mm,post length=1mm, pre length=1mm}] (0,5) -- (2,5);
\end{tikzpicture}
\begin{tikzpicture}[scale=0.35]
  \draw (0,8) -- (0,7);
  \draw (1,8) -- (1,7);
  \draw (2,8) -- (2,7);
  \draw (3,8) -- (3,7);
  \draw (4,8) -- (4,7);
  \draw (5,8) -- (5,7);
  \draw (6,8) -- (6,7);
  \draw (0,8)--(6,8);
  \draw (0,7)--(6,7);
  \draw (5.5,0);
\end{tikzpicture}
\hspace{-0.1cm}
\begin{tikzpicture}[scale=0.35]
  \draw (-0.7,5) node{$\mapsto$};
  \draw (0,8) -- (0,5);
  \draw (1,8) -- (1,5);
  \draw (2,8) -- (2,5);
  \draw (3,8) -- (3,5);
  \draw (4,8) -- (4,5);
  \draw (5,8) -- (5,5);
  \draw (6,8) -- (6,5);
  \draw (0,8)--(6,8);
  \draw (0,7)--(6,7);
  \draw (0,6)--(6,6);
  \draw (0,5)--(6,5);
  \draw[violet] (5.5,5.5) node{2};
  \draw (5.5,0);
\end{tikzpicture}
  \hspace{-0.1cm}
  \begin{tikzpicture}[scale=0.35]
  \draw (-0.7,5) node{$\mapsto$};
  \draw (0,8) -- (0,3);
  \draw (1,8) -- (1,3);
  \draw (2,8) -- (2,3);
  \draw (3,8) -- (3,3);
  \draw (4,8) -- (4,5);
  \draw (5,8) -- (5,3);
  \draw (6,8) -- (6,3);
  \draw (7,8) -- (7,6);
  \draw (0,8)--(7,8);
  \draw (0,7)--(7,7);
  \draw (0,6)--(7,6);
  \draw (0,5)--(6,5);
  \draw (0,4)--(3,4);
  \draw (5,4)--(6,4);
  \draw (0,3)--(3,3);
  \draw (5,3)--(6,3);
  \draw[violet] (5.5,4.5) node{2};
  \draw[blue] (5.5,3.5) node{1};
  \draw[blue] (6.5,6.5) node{1};
  \draw (5.5,0);
\end{tikzpicture}
\hspace{-0.1cm}
\begin{tikzpicture}[scale=0.35]
  \draw (-0.7,5) node{$\mapsto$};
  \draw (0,8) -- (0,3);
  \draw (1,8) -- (1,3);
  \draw (2,8) -- (2,3);
  \draw (3,8) -- (3,3);
  \draw (4,8) -- (4,3);
  \draw (5,8) -- (5,5);
  \draw (6,8) -- (6,5);
  \draw (7,8) -- (7,7);
  \draw (8,8) -- (8,7);
  \draw (0,8)--(8,8);
  \draw (0,7)--(8,7);
  \draw (0,6)--(6,6);
  \draw (0,5)--(6,5);
  \draw (0,4)--(4,4);
  \draw (0,3)--(4,3);
  \draw[violet] (3.5,4.5) node {2};
  \draw[blue] (3.5,3.5) node {1};
  \draw[blue] (7.5,7.5) node {1};
  \draw (5.5,0);
\end{tikzpicture}
Doing so we insert first $T_2$ and then $T_1$. When inserting $T_2$, which is of type 2, we add a cell containing $2$ below the cell coming from the fin and use neither \emph{merge} nor \emph{shift} nor \emph{correct parity}. When inserting $T_1$, which is of type 3, we \emph{shift} the pair $2,1$ to the left and put the other $1$ to a row above in \emph{correct parity}.
\end{example}

\begin{example}
We consider another orthogonal Littlewood-Richardson tableau and apply again Algorithm~\ref{alg:1}.

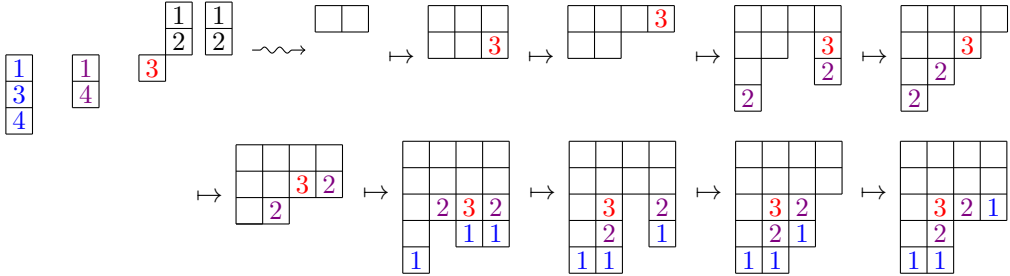
\begin{tikzpicture}[scale=0.35]
  \draw (0,3)--(1,3);
  \draw (0,2)--(1,2);
  \draw (0,1)--(1,1);
  \draw (0,0)--(1,0);
  \draw (0,3) -- (0,0);
  \draw (1,3) -- (1,0);
  \draw[blue] (0.5,2.5) node {1};
  \draw[blue] (0.5,1.5) node {3};
  \draw[blue] (0.5,0.5) node {4};
  \draw (2.5,3)--(3.5,3);
  \draw (2.5,2)--(3.5,2);
  \draw (2.5,1)--(3.5,1);
  \draw (2.5,3) -- (2.5,1);
  \draw (3.5,3) -- (3.5,1);
  \draw[violet] (3,2.5) node {1};
  \draw[violet] (3,1.5) node {4};
  \draw (6,5)--(7,5);
  \draw (6,4)--(7,4);
  \draw (5,3)--(7,3);
  \draw (5,2)--(6,2);
  \draw (5,3) -- (5,2);
  \draw (6,5) -- (6,2);
  \draw (7,5) -- (7,3);
  \draw (6.5,4.5) node {1};
  \draw (6.5,3.5) node {2};
  \draw[red] (5.5,2.5) node {3};
  \draw (7.5,5)--(8.5,5);
  \draw (7.5,4)--(8.5,4);
  \draw (7.5,3)--(8.5,3);
  \draw (7.5,5) -- (7.5,3);
  \draw (8.5,5) -- (8.5,3);
  \draw (8,4.5) node {1};
  \draw (8,3.5) node {2};
\end{tikzpicture}
\begin{tikzpicture}[scale=0.35]
 \draw (0,0) node{};
 \draw (0,5) node{};
 \draw [->,decorate,
decoration={snake,amplitude=.4mm,segment length=2mm,post length=1mm, pre length=1mm}] (0,3) -- (2,3);
\end{tikzpicture}
\begin{tikzpicture}[scale=0.35]
  \draw (4.5,5)--(6.5,5);
  \draw (4.5,4)--(6.5,4);
  \draw (4.5,5) -- (4.5,4);
  \draw (5.5,5) -- (5.5,4);
  \draw (6.5,5) -- (6.5,4);
  \draw (5,0);
\end{tikzpicture}
\hspace{-0.1cm}
\begin{tikzpicture}[scale=0.35]
  \draw (3.5,3) node{$\mapsto$};
  \draw (4.5,5)--(7.5,5);
  \draw (4.5,4)--(7.5,4);
  \draw (4.5,3)--(7.5,3);
  \draw (4.5,5) -- (4.5,3);
  \draw (5.5,5) -- (5.5,3);
  \draw (6.5,5) -- (6.5,3);
  \draw (7.5,5) -- (7.5,3);
  \draw[red] (7,3.5) node{3};
  \draw (5,0);
\end{tikzpicture}
\begin{tikzpicture}[scale=0.35]
  \draw (3.5,3) node{$\mapsto$};
  \draw (4.5,5)--(8.5,5);
  \draw (4.5,4)--(8.5,4);
  \draw (4.5,3)--(6.5,3);
  \draw (4.5,5) -- (4.5,3);
  \draw (5.5,5) -- (5.5,3);
  \draw (6.5,5) -- (6.5,3);
  \draw (7.5,5) -- (7.5,4);
  \draw (8.5,5) -- (8.5,4);
  \draw[red] (8,4.5) node{3};
  \draw (5,0);
\end{tikzpicture}
\begin{tikzpicture}[scale=0.35]
  \draw (3.5,3) node{$\mapsto$};
  \draw (4.5,5)--(8.5,5);
  \draw (4.5,4)--(8.5,4);
  \draw (4.5,3)--(6.5,3);
  \draw (7.5,3)--(8.5,3);
  \draw (4.5,2)--(5.5,2);
  \draw (7.5,2)--(8.5,2);
  \draw (4.5,1)--(5.5,1);
  \draw (4.5,5) -- (4.5,1);
  \draw (5.5,5) -- (5.5,1);
  \draw (6.5,5) -- (6.5,3);
  \draw (7.5,5) -- (7.5,2);
  \draw (8.5,5) -- (8.5,2);
  \draw[red] (8,3.5) node{3};
  \draw[violet] (8,2.5) node{2};
  \draw[violet] (5,1.5) node{2};
  \draw (5,0);
\end{tikzpicture}
\begin{tikzpicture}[scale=0.35]
  \draw (3.5,3) node{$\mapsto$};
  \draw (4.5,5)--(8.5,5);
  \draw (4.5,4)--(8.5,4);
  \draw (4.5,3)--(7.5,3);
  \draw (4.5,2)--(6.5,2);
  \draw (4.5,1)--(5.5,1);
  \draw (4.5,5) -- (4.5,1);
  \draw (5.5,5) -- (5.5,1);
  \draw (6.5,5) -- (6.5,2);
  \draw (7.5,5) -- (7.5,3);
  \draw (8.5,5) -- (8.5,4);
  \draw[red] (7,3.5) node{3};
  \draw[violet] (6,2.5) node{2};
  \draw[violet] (5,1.5) node{2};
  \draw (5,0);
\end{tikzpicture}

\hspace{2.3cm}
\begin{tikzpicture}[scale=0.35]
  \draw (3.5,3) node{$\mapsto$};
  \draw (4.5,5)--(8.5,5);
  \draw (4.5,4)--(8.5,4);
  \draw (4.5,3)--(8.5,3);
  \draw (4.5,2)--(6.5,2);
  \draw (4.5,2)--(5.5,2);
  \draw (4.5,5) -- (4.5,2);
  \draw (5.5,5) -- (5.5,2);
  \draw (6.5,5) -- (6.5,2);
  \draw (7.5,5) -- (7.5,3);
  \draw (8.5,5) -- (8.5,3);
  \draw[red] (7,3.5) node{3};
  \draw[violet] (6,2.5) node{2};
  \draw[violet] (8,3.5) node{2};
  \draw (5,0);
\end{tikzpicture}
\begin{tikzpicture}[scale=0.35]
  \draw (3.5,3) node{$\mapsto$};
  \draw (4.5,5)--(8.5,5);
  \draw (4.5,4)--(8.5,4);
  \draw (4.5,3)--(8.5,3);
  \draw (4.5,2)--(8.5,2);
  \draw (4.5,1)--(5.5,1);
  \draw (6.5,1)--(8.5,1);
  \draw (4.5,0)--(5.5,0);
  \draw (4.5,5) -- (4.5,0);
  \draw (5.5,5) -- (5.5,0);
  \draw (6.5,5) -- (6.5,1);
  \draw (7.5,5) -- (7.5,1);
  \draw (8.5,5) -- (8.5,1);
  \draw[red] (7,2.5) node{3};
  \draw[violet] (6,2.5) node{2};
  \draw[violet] (8,2.5) node{2};
  \draw[blue] (5,0.5) node{1};
  \draw[blue] (7,1.5) node{1};
  \draw[blue] (8,1.5) node{1};
  \draw (5,0);
\end{tikzpicture}
\begin{tikzpicture}[scale=0.35]
  \draw (3.5,3) node{$\mapsto$};
  \draw (4.5,5)--(8.5,5);
  \draw (4.5,4)--(8.5,4);
  \draw (4.5,3)--(8.5,3);
  \draw (4.5,2)--(6.5,2);
  \draw (7.5,2)--(8.5,2);
  \draw (4.5,1)--(6.5,1);
  \draw (7.5,1)--(8.5,1);
  \draw (4.5,0)--(6.5,0);
  \draw (4.5,5) -- (4.5,0);
  \draw (5.5,5) -- (5.5,0);
  \draw (6.5,5) -- (6.5,0);
  \draw (7.5,5) -- (7.5,1);
  \draw (8.5,5) -- (8.5,1);
  \draw[red] (6,2.5) node{3};
  \draw[violet] (6,1.5) node{2};
  \draw[violet] (8,2.5) node{2};
  \draw[blue] (5,0.5) node{1};
  \draw[blue] (6,0.5) node{1};
  \draw[blue] (8,1.5) node{1};
  \draw (5,0);
\end{tikzpicture}
\begin{tikzpicture}[scale=0.35]
  \draw (3.5,3) node{$\mapsto$};
  \draw (4.5,5)--(8.5,5);
  \draw (4.5,4)--(8.5,4);
  \draw (4.5,3)--(8.5,3);
  \draw (4.5,2)--(7.5,2);
  \draw (4.5,1)--(7.5,1);
  \draw (4.5,0)--(6.5,0);
  \draw (4.5,5) -- (4.5,0);
  \draw (5.5,5) -- (5.5,0);
  \draw (6.5,5) -- (6.5,0);
  \draw (7.5,5) -- (7.5,1);
  \draw (8.5,5) -- (8.5,3);
  \draw[red] (6,2.5) node{3};
  \draw[violet] (6,1.5) node{2};
  \draw[violet] (7,2.5) node{2};
  \draw[blue] (5,0.5) node{1};
  \draw[blue] (6,0.5) node{1};
  \draw[blue] (7,1.5) node{1};
  \draw (5,0);
\end{tikzpicture}
\begin{tikzpicture}[scale=0.35]
  \draw (3.5,3) node{$\mapsto$};
  \draw (4.5,5)--(8.5,5);
  \draw (4.5,4)--(8.5,4);
  \draw (4.5,3)--(8.5,3);
  \draw (4.5,2)--(8.5,2);
  \draw (4.5,1)--(6.5,1);
  \draw (4.5,0)--(6.5,0);
  \draw (4.5,5) -- (4.5,0);
  \draw (5.5,5) -- (5.5,0);
  \draw (6.5,5) -- (6.5,0);
  \draw (7.5,5) -- (7.5,2);
  \draw (8.5,5) -- (8.5,2);
  \draw[blue] (5,0.5) node {1};
  \draw[blue] (6,0.5) node {1};
  \draw[blue] (8,2.5) node {1};
  \draw[violet] (6,1.5) node {2};
  \draw[violet] (7,2.5) node {2};
  \draw[red] (6,2.5) node {3};
  \draw (5,0);
\end{tikzpicture}

Doing so we insert first $T_3$ then $T_2$ and in the end $T_1$. All three of them are of type 1. When inserting $T_3$ we use only \emph{correct parity} to put the $3$ upwards. When inserting $T_2$ we first \emph{shift} the pair $3,2$ to the left and then put the other $2$ upwards in \emph{correct parity}. When inserting $T_1$ we first \emph{merge} the pair $3,1$ with the $2$ to the left. Then we \emph{shift} the pair $2,1$ to the left and in the end we put this $1$ upwards in \emph{correct parity}.
\end{example}

\begin{example}
The empty cells of our tableaux can be determined by the filling of Kwon's tableaux. However this shape does not define our tableaux by far. As the following tableaux show it neither defines where to add the filled cells:

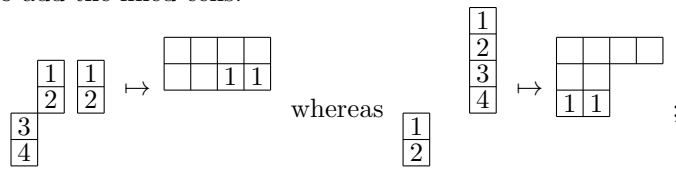
\hspace{2cm}\begin{tikzpicture}[scale=0.35]
  \draw (6,5)--(7,5);
  \draw (6,4)--(7,4);
  \draw (5,3)--(7,3);
  \draw (5,2)--(6,2);
  \draw (5,1)--(6,1);
  \draw (5,3) -- (5,1);
  \draw (6,5) -- (6,1);
  \draw (7,5) -- (7,3);
  \draw (6.5,4.5) node {1};
  \draw (6.5,3.5) node {2};
  \draw (5.5,2.5) node {3};
  \draw (5.5,1.5) node {4};
  \draw (7.5,5)--(8.5,5);
  \draw (7.5,4)--(8.5,4);
  \draw (7.5,3)--(8.5,3);
  \draw (7.5,5) -- (7.5,3);
  \draw (8.5,5) -- (8.5,3);
  \draw (8,4.5) node {1};
  \draw (8,3.5) node {2};
\end{tikzpicture}
\begin{tikzpicture}[scale=0.35]
  \draw (3.5,3) node{$\mapsto$};
  \draw (4.5,5)--(8.5,5);
  \draw (4.5,4)--(8.5,4);
  \draw (4.5,3)--(8.5,3);
  \draw (4.5,5) -- (4.5,3);
  \draw (5.5,5) -- (5.5,3);
  \draw (6.5,5) -- (6.5,3);
  \draw (7.5,5) -- (7.5,3);
  \draw (8.5,5) -- (8.5,3);
  \draw (7,3.5) node{1};
  \draw (8,3.5) node{1};
  \draw (5,0);
\end{tikzpicture}
\raisebox{0.7cm}{ whereas }
\begin{tikzpicture}[scale=0.35]
  \draw (5,3)--(6,3);
  \draw (5,2)--(6,2);
  \draw (5,1)--(6,1);
  \draw (5,3) -- (5,1);
  \draw (6,3) -- (6,1);
  \draw (5.5,2.5) node {1};
  \draw (5.5,1.5) node {2};
  \draw (7.5,7)--(8.5,7);
  \draw (7.5,6)--(8.5,6);
  \draw (7.5,5)--(8.5,5);
  \draw (7.5,4)--(8.5,4);
  \draw (7.5,3)--(8.5,3);
  \draw (7.5,7) -- (7.5,3);
  \draw (8.5,7) -- (8.5,3);
  \draw (8,6.5) node {1};
  \draw (8,5.5) node {2};
  \draw (8,4.5) node {3};
  \draw (8,3.5) node {4};
\end{tikzpicture}
\begin{tikzpicture}[scale=0.35]
  \draw (3.5,3) node{$\mapsto$};
  \draw (4.5,5)--(8.5,5);
  \draw (4.5,4)--(8.5,4);
  \draw (4.5,3)--(6.5,3);
  \draw (4.5,2)--(6.5,2);
  \draw (4.5,5) -- (4.5,2);
  \draw (5.5,5) -- (5.5,2);
  \draw (6.5,5) -- (6.5,2);
  \draw (7.5,5) -- (7.5,4);
  \draw (8.5,5) -- (8.5,4);
  \draw (5,2.5) node{1};
  \draw (6,2.5) node{1};
  \draw (5,0);
\end{tikzpicture}
\raisebox{0.7cm}{;}

\noindent nor does it define how to fill those cells:

\hspace{0.7cm}\begin{tikzpicture}[scale=0.35]
  \draw (2.5,3)--(3.5,3);
  \draw (2.5,2)--(3.5,2);
  \draw (2.5,1)--(3.5,1);
  \draw (2.5,3) -- (2.5,1);
  \draw (3.5,3) -- (3.5,1);
  \draw (3,2.5) node {1};
  \draw (3,1.5) node {3};
  \draw (6,5)--(7,5);
  \draw (6,4)--(7,4);
  \draw (5,3)--(7,3);
  \draw (5,2)--(6,2);
  \draw (5,3) -- (5,2);
  \draw (6,5) -- (6,2);
  \draw (7,5) -- (7,3);
  \draw (6.5,4.5) node {1};
  \draw (6.5,3.5) node {2};
  \draw (5.5,2.5) node {5};
  \draw (7.5,7)--(8.5,7);
  \draw (7.5,6)--(8.5,6);
  \draw (7.5,5)--(8.5,5);
  \draw (7.5,4)--(8.5,4);
  \draw (7.5,3)--(8.5,3);
  \draw (7.5,7) -- (7.5,3);
  \draw (8.5,7) -- (8.5,3);
  \draw (8,6.5) node {1};
  \draw (8,5.5) node {2};
  \draw (8,4.5) node {3};
  \draw (8,3.5) node {4};
\end{tikzpicture}
\begin{tikzpicture}[scale=0.35]
  \draw (3.5,3) node{$\mapsto$};
  \draw (4.5,5)--(10.5,5);
  \draw (4.5,4)--(10.5,4);
  \draw (4.5,3)--(8.5,3);
  \draw (4.5,2)--(6.5,2);
  \draw (4.5,5) -- (4.5,2);
  \draw (5.5,5) -- (5.5,2);
  \draw (6.5,5) -- (6.5,2);
  \draw (7.5,5) -- (7.5,3);
  \draw (8.5,5) -- (8.5,3);
  \draw (9.5,5) -- (9.5,4);
  \draw (10.5,5) -- (10.5,4);
  \draw (6,2.5) node{1};
  \draw (8,3.5) node{1};
  \draw (10,4.5) node{2};
  \draw (5,0);
\end{tikzpicture}
\raisebox{0.7cm}{ whereas }
\begin{tikzpicture}[scale=0.35]
  \draw (2.5,3)--(3.5,3);
  \draw (2.5,2)--(3.5,2);
  \draw (2.5,1)--(3.5,1);
  \draw (2.5,3) -- (2.5,1);
  \draw (3.5,3) -- (3.5,1);
  \draw (3,2.5) node {1};
  \draw (3,1.5) node {5};
  \draw (6,5)--(7,5);
  \draw (6,4)--(7,4);
  \draw (5,3)--(7,3);
  \draw (5,2)--(6,2);
  \draw (5,3) -- (5,2);
  \draw (6,5) -- (6,2);
  \draw (7,5) -- (7,3);
  \draw (6.5,4.5) node {1};
  \draw (6.5,3.5) node {2};
  \draw (5.5,2.5) node {3};
  \draw (7.5,7)--(8.5,7);
  \draw (7.5,6)--(8.5,6);
  \draw (7.5,5)--(8.5,5);
  \draw (7.5,4)--(8.5,4);
  \draw (7.5,3)--(8.5,3);
  \draw (7.5,7) -- (7.5,3);
  \draw (8.5,7) -- (8.5,3);
  \draw (8,6.5) node {1};
  \draw (8,5.5) node {2};
  \draw (8,4.5) node {3};
  \draw (8,3.5) node {4};
\end{tikzpicture}
\begin{tikzpicture}[scale=0.35]
  \draw (3.5,3) node{$\mapsto$};
  \draw (4.5,5)--(10.5,5);
  \draw (4.5,4)--(10.5,4);
  \draw (4.5,3)--(8.5,3);
  \draw (4.5,2)--(6.5,2);
  \draw (4.5,5) -- (4.5,2);
  \draw (5.5,5) -- (5.5,2);
  \draw (6.5,5) -- (6.5,2);
  \draw (7.5,5) -- (7.5,3);
  \draw (8.5,5) -- (8.5,3);
  \draw (9.5,5) -- (9.5,4);
  \draw (10.5,5) -- (10.5,4);
  \draw (6,2.5) node{1};
  \draw (8,3.5) node{2};
  \draw (10,4.5) node{1};
  \draw (5,0);
\end{tikzpicture}
\raisebox{0.7cm}{.}
\end{example}

\subsection{Properties and Proofs for Bijection $A$}

\begin{theorem}
\label{theo:algo1WellDef}
Algorithm~\ref{alg:1} is well-defined and returns an alternative orthogonal Littlewood-Richardson tableau.
\end{theorem}

\begin{proof}[Outline of the proof]
We will prove this theorem by induction. In particular we show that after every iteration $i$ of the outer for-loop the rules for alternative orthogonal Littlewodd-Richardson tableaux are satisfied, if we would subtract $(i-1)$ from every entry. For the base case we get the shape of $S$ reflected, which satisfies our conditions for $\mu=\emptyset$ and $k=0$. The induction step is  shown by the following lemmas.
\end{proof}

We state some properties following from the formulation of the algorithm first. We refer to parts or operations in the algorithms by the comments placed next to them.

\begin{corollary}
\label{cor:LongerColum}
\begin{enumerate}
\item In Algorithm~\ref{alg:1} there are two types of rows that get longer during the inner for-loop. One type consists of those rows in which the new $i$'s are inserted and the rows directly above. Those get longer by one for each such $i$. The other type consists of the bottommost two rows which get longer by values of the same parity.

\item \emph{Correct parity} can be reformulated to the following and still leads to the same result.

Go through $\tilde{L}$ from bottom to top. If the current row has a length of a different parity than $S$, put the rightmost $i$ to the next row such that it is the leftmost $i$ in this new row. (Shift other $i$'s one column to the right.)

\item Unfilled positions form a Young diagram of a partition (and not a skew-shape).
\end{enumerate}
\end{corollary}

\begin{proof}
\begin{enumerate}
\item For cells that do not come from the lower tail and the tail root / fin we consider \eqref{eq:HeightProperty} and \eqref{eq:Height_Property}. Due to those there is for each newly inserted empty cell an already inserted one coming from $T_{i+1}$. If $r_{T_i}=r_{T_{i+1}}=1$ or $S$ odd and $r_{T_{\ell(\mu)}}=1$ this holds for the tail root and the non-tail-parts except for the fin. Otherwise this holds for the non-tail-parts.

Adding cells for the lower tail and the tail root / fin and adding cells with $i$ corresponding to them extends columns by two. \emph{Merge} or \emph{shift} preserves this until the point where only an $i$ is \emph{shifted}. In this case this is the last movement of this $i$ and it is still in the row below, the one that gets longer too.

\item Wrong parity is caused by columns getting longer by one. Therefore, if the row above the bottommost row with wrong parity has the right parity, there is also an $i$ (an odd number of $i$'s in fact). Iterating this argument completes the proof.

\item This follows directly from \eqref{eq:GapProperty} and \eqref{eq:SProperty}.\qedhere
\end{enumerate}
\end{proof}

\begin{lemma}
Each step of the outer for-loop is well-defined if it adds a new $T_{i}$ to an $\tilde{L}$ as demanded.
\end{lemma}

\begin{proof}
There are three steps, in which this is not obvious:
\begin{itemize}
\item There is always an $i$ in the current row when we \emph{merge}.

Columns that get longer by inserted cells of the non-tail parts cannot cause a \emph{merge}-situation because then also the columns to the left get longer.
Besides this, a column can only get longer by inserting an $i$, \emph {merge} or \emph{shift}.  Only at insert $i$ and \emph{merge} a former inserted $l$ can move to a different row. Thus only in those cases a \emph{merge}-situation can arise. Therefore we can show inductively that there is always an $i$ in such a column.

\item  There are always $i$'s to find and places to put them at correct parity.

For rows that are not the two bottommost ones this is follows from Corollary~\ref{cor:LongerColum}. We make rows longer in pairs. If we make them longer by an odd number, there needs to be enough space to put an element of the upper row to the lower one due to parity reasons. (If the latter got longer together with the one above, then this also got longer. We can iterate this argument.)

Now we consider the bottommost rows.
We start with the case that $S$ is even and consider a newly inserted $T_i$. Suppose that the bottommost row is of odd length and contains no $i$. This means that an empty cell has been added but no $i$ has been put below it. As non-tail parts of our $T_i$ have even length, $T_i$ needs to be of type $2$ or $3$.

If $T_{i+1}$ has residuum $0$ (thus $T_i$ is of type 2), the bottommost row so far had even length and was longer than or of equal length as $T_i^R$ by \eqref{eq:HeightProperty}. Therefore the fin of $T_i$ is placed into the second row from bottom, with an $i$ below. This $i$ is put into the bottommost row, which is a contradiction. At this point there are no other $j$ left or below to this.

If $T_{i+1}$ has residuum $1$ it is also of type $2$ or $3$. In both cases the bottommost rows (those which consist of odd many empty cells, thus those above the fin of $T_{i+1}$) contain each one $j$ increasingly from bottom to top. Now if we insert the new fin, it is inserted to one of these spots, as it is even and smaller than or equal to the fin of $T_i$. A sequence of \textit{merge} puts the $i$ into the bottommost row and the $i+1$ to the row above, as in each step a $j$ is put into the row of an $j-1$.

Now we consider the case that $S$ is odd. We have to show that there are $i$'s in a row, if this is of even length. Even columns without $i$ cannot come from type 2 or 3 tableaux thus we have to consider type 1. If this is the first tableau from right thus directly left of $S$, the additional condition prevents this. Otherwise, we know that the tail root is smaller than or equal to the tail root to the right. Now we can argue exactly as above, but with the tail root instead of the fin.

\item While-loops stop.

The first while-loop stops after several steps because \emph{merge} always works, as we have seen in the previous point, and moves $i$'s to the left. As there are finitely many columns, this has to stop at some point.

The second while-loop stops as after some steps everything is \emph{shifted} to the left.

The third while-loop stops as after some steps all rows have the same parity, because the parity of $S$ equals the parity of the number of elements in $\tilde{L}$. This holds as the number of elements in $S$ has the same parity as $S$ and when inserting $T_i$ we insert $2b_i-a_i+\mu_{T_i}$ empty cells and $\mu_{T_i}$ $i$'s, which gives an even number of added cells to $\tilde{L}$. \qedhere
\end{itemize}
\end{proof}

\begin{lemma}
After each insertion of a $T_i$ we obtain a reversed skew semistandard tableau.
\end{lemma}

\begin{proof}
Each column is sorted such that fillings decrease and empty cells are on top as the columns get sorted after any operation that might change this. Each row is sorted the same way as due to Corollary~\ref{cor:LongerColum} the empty cells build a Young diagram, thus are left in each row, and unsorted filled cells get eliminated by \emph{merge}.
Therefore it suffices to show that there is at most one $j$ in each column.

As we insert at most one $i$ to each column, there is at most one $i$ in each column at the begin of an iteration. The same holds by induction for $j$ with $j>i$.

We show that this holds also at the end of this iteration.

We note that the only situations where an $l$ or $i$ moves to another column are \emph{merge}, \emph{shift} and \emph{correct parity}.

First we consider \emph{merge}: Whenever there is a $j$ in left of an $l$ with $j<l$, this happens exactly if below the $l$ is a newly inserted $i$ or, if there was such a situation in the column to the right too, but not in the column of $j$, or if there was a shift of $l$ and $i$. Therefore there cannot be an $i$ in the column to the left because otherwise entries in this column would have moved one down too and $j<l$ would have caused a disorder before. Thus after \emph{merge} the number of $i$ in each column is still at most one. What remains to show is that there cannot be another $l$ in the new column of $l$. Therefore we consider the position directly to the right of this other $l$. This needs to be smaller, as rows above are sorted and it needs to be larger as columns are sorted. This is a contradiction.

After \emph{merge} the current row is sorted. Therefore \emph{shifting} cells in the same row to the left does not increase the number of $j$'s to two for any $j$ and any column.

Finally we consider \emph{correct parity}. Suppose an $i$ is put into a column where already an $i$ is. 
This would mean that there is not enough space for this $i$ to be put into this row, if the other $i$ would not be here. However we make rows longer in pairs, and if we make them longer by an odd number, there needs to be enough space to put an element of the upper row to the lower one due to parity reasons.
\end{proof}

\begin{lemma}
After each insertion of a $T_i$ the first property of alternative orthogonal Littlewood-Richardson tableaux holds.
\end{lemma}

\begin{proof}
Due to \eqref{eq:TailProperty} each element $i$ is inserted left below of the $(i+1)$ directly to the right in the tails. We show that operations in the outer for-loop do not change that.

If $(i+1)$ is still in the column where it was inserted or, due to \emph{correct parity}, to the right, $i$ gets inserted left of $(i+1)$. In this case neither \emph{merge} nor \emph{shift} can change this.

Now we consider the case that $(i+1)$ has changed column in \emph{merge} or \emph{shift}. If this happened and $i$ is inserted to the right of it, we show that there needs to be an $l$ above of $i$, thus \emph{merge} also takes place for $i$. (If there is such an $l$, there was a position on top of the upper \emph{merged} or \emph{shifted} position, that now ends up to the right.) Where the empty cell belonging to $i$ gets inserted, there needs to be an $l$ or no cell and an empty cell to the left for the empty cells to form a tableau. No cell is not possible as $(i+1)$ was inserted to the same column (or to the right) and changed column in \emph{merge} or \emph{shift} and there is an empty cell to the left. A sequence of \emph{merge} and \emph{shift} will be followed by a sequence of \emph{merge} by the same argument. (We always put two elements leftwards and the upper one will be the next candidate for \emph{merge} as it is smaller than the element it is \emph{shifted} to, because it was in the same column on top on it.)

What is left to consider is \emph{correct parity}. In this case the row of $i$ has odd (respectively even, depending on the parity of $S$) length and $i$ is the rightmost position in it. Thus $i$ is in an odd (respectively even) column. If $(i+1)$ is in a column to the right of  $i$, $i$ still ends up in the same column or to the left. If $i$ is in the same column as $(i+1)$ we show that there cannot happen correct parity. The column where $(i+1)$ is now, got longer when it was inserted (or \emph{merged/shifted to}). As it is still there and it is an odd (respectively even) column, the column to the left also got longer such that they had the same length. Thus it too contains an $(i+1)$, but no $i$, which is a contradiction.

$l$ also changes column. If $(l-1)$ is in the same column, this was next to $j$ before this column got longer through $i$. This means by induction $l-1\leq j$ but since $j<l$ we obtain $l-1=j$. Thus the $(l-1)$ in the original column of $l$ is not its according $(l-1)$. (The $(l-1)$ in the column $l$ is moved to is its according $(l-1)$.)
\end{proof}

\begin{lemma}
After each insertion of a $T_i$ also the second property of alternative orthogonal Littlewood-Richardson tableau holds.
\end{lemma}

\begin{proof}
We start this proof by reformulating the second property:

Instead of putting elements into a $v_e$ we can also mark them using the same rules. We remember how often an element was marked and count which element was marked at with position during considering $e$.
Now we observe the following:
\begin{itemize}
\item It only matters how often an element is marked. It is not important by which elements it was marked.
\item Whenever we consider an $i$ and mark elements, if we mark an element the $(j+1)$-th time, we mark another, smaller element, the $j$-th time. Thus the number elements marked $j$ times does never decrease.
\end{itemize}
Now we prove the statement.

First we have a closer look at what happens locally, when there is an $i$ inserted (or \emph{merge} happens) in one column but not in its neighbor.

To do so we first consider a column together with its left neighbor. We examine four elements in a pattern as below with $j_1<j_2$, $j_3<j_4$ and $j_2\geq j_3$. \emph{Merge} or insert $i$ happens in the right column. Thus this is placed downwards by one. If $j_1\geq j_3$ and $j_2\geq j_4$ nothing changes, while if $j_1< j_3$ and $j_2\geq j_4$ or $j_2<j_4$ we merge.

\begin{tikzpicture}[scale=0.35]
  \draw (1,4)--(2,4);
  \draw (0,3)--(2,3);
  \draw (0,2)--(2,2);
  \draw (0,1)--(1,1);
  \draw (0,3) -- (0,1);
  \draw (1,4) -- (1,1);
  \draw (2,4) -- (2,2);
  \draw (1.5,3.5) node{$j_4$};
  \draw (1.5,2.5) node{$j_3$};
  \draw (0.5,2.5) node{$j_2$};
  \draw (0.5,1.5) node{$j_1$};
  \draw (0,0);
\end{tikzpicture}
\hspace{-1cm} is changed into:
 if $j_1\geq j_3$ and $j_2\geq j_4$:
\begin{tikzpicture}[scale=0.35]
  \draw (0,3)--(2,3);
  \draw (0,2)--(2,2);
  \draw (0,1)--(2,1);
  \draw (0,3) -- (0,1);
  \draw (1,3) -- (1,1);
  \draw (2,3) -- (2,1);
  \draw (1.5,2.5) node{$j_4$};
  \draw (1.5,1.5) node{$j_3$};
  \draw (0.5,2.5) node{$j_2$};
  \draw (0.5,1.5) node{$j_1$};
  \draw (0,0);
\end{tikzpicture}
 if $j_1< j_3$ and $j_2\geq j_4$: 
\begin{tikzpicture}[scale=0.35]
  \draw (0,3)--(2,3);
  \draw (0,2)--(2,2);
  \draw (0,1)--(1,1);
  \draw (0,0)--(1,0);
  \draw (0,3) -- (0,0);
  \draw (1,3) -- (1,0);
  \draw (2,3) -- (2,2);
  \draw (1.5,2.5) node{$j_4$};
  \draw (0.5,1.5) node{$j_3$};
  \draw (0.5,2.5) node{$j_2$};
  \draw (0.5,0.5) node{$j_1$};
  \draw (0,0);
\end{tikzpicture}
if $j_2<j_4$:
\begin{tikzpicture}[scale=0.35]
  \draw (0,3)--(2,3);
  \draw (0,2)--(2,2);
  \draw (0,1)--(1,1);
  \draw (0,0)--(1,0);
  \draw (0,3) -- (0,0);
  \draw (1,3) -- (1,0);
  \draw (2,3) -- (2,2);
  \draw (0.5,2.5) node{$j_4$};
  \draw (1.5,2.5) node{$j_3$};
  \draw (0.5,1.5) node{$j_2$};
  \draw (0.5,0.5) node{$j_1$};
  \draw (0,0);
\end{tikzpicture}

\noindent Now we determine which elements get marked if no other elements interfere:
\begin{itemize}
  \item $j_1\geq j_3$ and $j_2\geq j_4$: $\{j_4\}$, $\{j_3,j_4\}$, $\{j_2\}$, $\{j_1,j_2\}$. After the insertion process this changes to $\{j_4\}$,  $\{j_2\}$, $\{j_3,j_4\}$, $\{j_1,j_2\}$, which only changed the order.
\item $j_1< j_3$ and $j_2\geq j_4$: $\{j_4\}$, $\{j_3,j_4\}$, $\{j_2\}$, $\{j_1,j_3, j_4\}$. After the insertion process this changes to $\{j_4\}$, $\{j_2\}$, $\{j_3,j_4\}$, $\{j_1,j_3, j_4\}$, which also only changed the order.
\item For $j_2<j_4$ we distinguish the cases $j_1<j_3$ and $j_1\geq j_3$: $\{j_4\}$, $\{j_3,j_4\}$, $\{j_2\}$, $\{j_1,j_3, j_4\}$ or $\{j_1,j_2, j_4\}$. After the insertion process this changes to $\{j_3\}$, $\{j_4\}$, $\{j_2,j_4\}$, $\{j_1,j_3, j_4\}$ or $\{j_1,j_2, j_4\}$, which is more than just a change of order but does not change anything about what is marked afterwards. $j_3$ and $j_2$ swap number of marked elements, which is allowed as $j_2$, which number increases, is one row below of where $j_3$ was. (The same row would have been sufficient.)
\end{itemize}

Now we consider a column together with its right neighbor. Again we examine four elements in a pattern as below with $j_1<j_2$, $j_3<j_4$, $j_1\geq j_3$ and $j_2\geq j_4$ (and therefore $j_2>j_3$). \emph{Merge} or insert $i$ happens in the right column. Thus this is placed downwards by one. Everything ends up sorted, so no \emph{merge} happens:

\begin{tikzpicture}[scale=0.35]
  \draw (0,3)--(2,3);
  \draw (0,2)--(2,2);
  \draw (0,1)--(2,1);
  \draw (0,3) -- (0,1);
  \draw (1,3) -- (1,1);
  \draw (2,3) -- (2,1);
  \draw (1.5,2.5) node{$j_4$};
  \draw (1.5,1.5) node{$j_3$};
  \draw (0.5,2.5) node{$j_2$};
  \draw (0.5,1.5) node{$j_1$};
  \draw (0,0);
\end{tikzpicture}
is changed into
\begin{tikzpicture}[scale=0.35]
  \draw (1,3)--(2,3);
  \draw (0,2)--(2,2);
  \draw (0,1)--(2,1);
  \draw (0,0)--(1,0);
  \draw (0,2) -- (0,0);
  \draw (1,3) -- (1,0);
  \draw (2,3) -- (2,1);
  \draw (1.5,2.5) node{$j_4$};
  \draw (1.5,1.5) node{$j_3$};
  \draw (0.5,1.5) node{$j_2$};
  \draw (0.5,0.5) node{$j_1$};
  \draw (0,0);
\end{tikzpicture}.

\noindent Now we consider the marked elements in the insertion process: 
$\{j_4\}$,  $\{j_2\}$, $\{j_3,j_4\}$, $\{j_1,j_2\}$ which changes to $\{j_4\}$,  $\{j_3,j_4\}$, $\{j_2\}$, $\{j_1,j_2\}$ which is again only a change of the order.

As a second step we show that there are no relevant changes in those columns to the left and to the right. Once we have shown this, we can conclude, that $j>i$ still satisfy the third property, if we ignore elements counted by $o$.

To see this we can argue that anything even more to the left of a column that got changed is larger. Thus it marks even larger elements. In the case that it marks elements that would have been marked by $j_1$ to $j_4$ due to their change of rows, they simply swap which elements they mark. As they used to be in the same row, the third property still holds.

Everything more to the right of a column that got changed is smaller,  and takes smaller elements, still it could be that the same kind of swapping occurs.

In the third step we have a closer look at what happens to $i$. If an $i$ is inserted and not changed by \emph{correct parity}, it is two rows below from where the elements one row above are. It can mark at most one element more, which still satisfies the third property.

\emph{Correct parity} puts $i$ one row up. If there are other $i$'s in this row we can argue as above. Otherwise we see that it can mark at most one element more than those elements one row above. Lets call the rightmost one of them $j$. Suppose this $j$ has only one row for every rightmost element in $v_j$ and our $i$ is not the rightmost $i$.

(If our $i$ is the rightmost one, it can take only rightmost elements, and can take more elements as $o$. If $j$ has more rightmost element $i$ can take also more elements as $o$. This is because the element left of $i$ can take at most as many elements as $j$ but is the row of $i$ now. Thus the \enquote{only} condition holds.)

For $j$ to be taken in $v_i$, elements that are smaller or equal but more to the right have to be taken by other $i$. Let's call the leftmost such element $m$. Below $m$ there is no row without a number because if there would be one, an element of $v_j$ would be in the same row as one of $v_m$, that cannot satisfy the \enquote{only} condition. Thus when we insert another $i$ below $m$ that is not moved in correct parity above or besides $m$, either this $i$ or another $i$ left of it moves a position of $v_j$. Thus $v_i$ has less elements too and satisfies the third condition.

Finally we consider elements that are counted by an $o$. Normally they just stay the same as every other element. When they are put one row down, it can happen that they count one time less as an $o$ element, which is fine, as they also went down one row (and with them those which mark it). This happens if only this columns gains length and not the one directly to the left.
\end{proof}

\begin{lemma}
$\tilde{L}$ has at most $2(k-i+1)+1$ rows, if $S$ is odd, this number is met.
\end{lemma}

\begin{proof}
Each column grows by adding empty cells and cells labeled $i$. \emph{Merge} and \emph{shift} only lead to grows of columns untouched so far. There are at most two numbers with the same value in $T_i$, thus at most two new empty cells in each column. We show that if there are two, no $i$ is inserted to the same column. Recall that $i$'s are inserted below lower tail elements and the tail root or the fin, depending on the parity.

If a tableau has residuum zero, there cannot be an element that is in both, in the tail and in the right column. If a tableau has residuum 1, such a tail element could only be the tail root which produces no $i$ either. The number of the fin cannot be in both columns, as the left column without the tail is shorter by at least two. Therefore no row can grow by more than two.

It remains to show the second claim. Thus we consider odd $S$ and $T_i$ that do not contain two $1$'s. Tableaux with residuum one contain a $1$ in each column, as neither the position above the fin (which exists) nor the tail root nor any position above one of them is a gap. Thus we only consider tableaux of type $1$.

If only the left column contains a $1$, thus the right one is empty, we add $i$ an empty cell in column 1.

Tableaux with no $1$ in the left column consist only of the tail and a right column, if the tableau to the right had a left column larger than its tail (due to \eqref{eq:HeightProperty}). All those tail elements get inserted together with an $i$. Now the tail root needs to be smaller than or equal to all other tail roots by \eqref{eq:TailProperty}. Residuum $1$ tableaux produce two $1$'s, two $2$'s up to two such tail roots. Type $1$ tableaux right of residuum 1 tableaux (respectively $S$) also do so due to \eqref{eq:HeightProperty} (respectively \eqref{eq:Height_Property}). Thus the first tableau of type $1$, $T_j$, whose left column consists only of the tail inserts a $j$ into row $2j+1$. Now as the second property of alternative tableau holds, $i$'s that come afterwards will end up below, and their column will grow by two. (It is not possible that they end up there by a \emph{shift} where only one element is \emph{shifted}, as this would need two $(i+1)$'s belonging to one $i$ which is a contradiction.)
\end{proof}

We have now proven every Lemma that proves Theorem~\ref{theo:algo1WellDef}. Next we show the same for the reversed algorithm.

\begin{lemma}
\label{theo:ualgo2WellDef}
Algorithm~\ref{alg:u1} is well-defined. 

More precisely, after each iteration $i$ of the outer for-loop we obtain an alternative orthogonal Littlewood-Richardson tableau $\tilde{L}$ if we would decrease numbers by $i-1$.
\end{lemma}

\begin{proof}
Due to construction the algorithm is well-defined once we ensure that we have always enough cells to mark. We will see this during the proof. First we show that $\tilde{L}$ is a reverse skew semistandard tableau and the two properties of an alternative orthogonal Littlewood-Richardson tableau hold if they held before.
\begin{itemize}
\item Again, to show that $\tilde{L}$ is a reverse skew semistandard tableau, it suffices to show that everything is sorted and that there is at most one $j$ in each column. Again no operation puts more than one $i$ into the same column. Columns get sorted after deleting something, a violation of the row order is be prevented by \emph{merge}, because a violation occurs exactly when the condition of \emph{merge} arises.

If an empty cell is erased left of another empty cell and thus shifts a cell labeled $j_2$ to an empty cell, or deleting shifts a smaller entry $j_2$ next to a bigger one $j_3$, this column is at least three cells larger, otherwise there would have been a \emph{shift}. In this case we obtain \emph{merge} and define $j_1$ to be the largest entry between $j_2$ and $i$.

\item The only operation which can destroy the first property (that there is always a $j$ below a $(j+1)$) is \emph{merge} at $j_2$. This makes a problem if there is a $(j_2+1)$ in the same column, and this $j_2$ is the one belonging to it. For $(j_2+1)$ not to  be taken instead of $j_2$ one of the following conditions must be met: either $(j_2+1)$ is in the column to the right or $j_2<j_4$ where $j_4$ is the position right of $(j_2+1)$. In the former case $j_2$ belongs to that, which is a contradiction. In the latter case we obtain $(j_2+1)>j_4>j_2$ which is also a contradiction as all three numbers are natural numbers.

\item For the second property we can do a similar case study of local changes as we did before for Algorithm~\ref{alg:1}. However there are some steps we have to consider more precisely.

When an $i$ gets extracted, elements move one row up. This row however, was not necessary then, because the $i$ causing this move needed two more rows for its formula, so the moved elements needed at least two less.

For $o$ we also argue similar as for Algorithm~\ref{alg:1}. If those which count for $o$ are put upwards, but not the ones to the left, they are counted as $o$ once more as before. This needs to be the case as we might not have this \enquote{not necessary} row in the $o$ case.
\end{itemize}

Now we show that the row parity is constant. We shift one $i$ for $i$'s that are in odd sequences to the left. Thus the shifted $i$'s shorten the row where they were by two (this $i$ and their empty cell), while the other $i$'s (an even number) shorten this row and the one above by an even number each.

The tableau has at most $2i+1$ rows after an iteration as lower ones are taken as left and right column of the new tableau.

To see that there are enough cells to mark we consider the second condition of the alternative orthogonal Littlewood-Richardson tableau. This ensures one position to mark for positions, except for some $o$ cases. In this cases \emph{correct parity} puts it one row above.
\end{proof}

\begin{lemma}
Each iteration of the outer for-loop produces a tableau of one of the three types.
\end{lemma}

\begin{proof}
The shape follows from well-definedness and the last if-query once we can show that there are never one even and one odd row to be taken for the left and the right column of $T_i$.
Thus we consider rows $2i+1$ and $2i$ and distinguish two cases.

In the first case no $i$ is put to row $2i+1$ in correct parity.
Therefore there are even many $i$'s in row $2i$. Thus the parity of row $2i$ after extracting is the same as the parity of row $2i+1$, as $i$'s in row $2i+1$  shorten both, row $2i$ and $2i+1$. \emph{Shift} or \emph{merge} do not change this.
In the second case an $i$ is put from row $2i$ to row $2i+1$ in correct parity. In the end this shortens row $2i$ by two, so parity is still preserved, by the same argument as in the first case.

Because elements that were put originally to the tail are larger than other elements, the residuum of $T_i$ is $0$ before the last if-query.

This also shows that a gap in the right column is the fin. If it is one the residuum is 1.

Moreover the tail root is not a gap for residuum one tableaux, as it comes from the left column without tail.

$T_i$ is semistandard as row $2i+1$ cannot be longer than row $2i$.
\end{proof}

\begin{lemma}
The fin of such a tableau is even.
\end{lemma}

\begin{proof}
The fin is even if it is no gap. If its a gap, rows $2i$ and $2i+1$ are of odd length before extracting them. The leftmost $i$ after \emph{correct parity} is in an even column. (If $S$ would be odd and this $i$ would be in an odd column, it changes in correct parity. If $S$ is even, it needs to be in row $2i+1$ and was there before.) Neither \emph{shift} nor \emph{merge} can change this, as $i$'s that are in the same row can be \emph{shifted} to the right together and if \emph{merge} occurs, there can either be another \emph{merge} for the $i$ to the left or they stay in the row where they are. As the parity of row lengths is constant and there are even many $i$'s in other rows, this is sufficient.
\end{proof}

\begin{lemma}
Let $T_i$ be of type 3. If $i<\ell(\mu)$, $T_{i+1}$ cannot be of type 1. If $i=\ell(\mu)$, $S$ is odd.
\end{lemma}

\begin{proof}
A tableau $T_i$ of type 3 is formed in the last if-query where the tail root becomes the fin. The fin has to come from a row strictly above of row $2i$ as it is a gap. We show now that the bottommost row after extracting $T_i$, thus row $2i-1$ consist of odd many empty cells. As this will get into the left column of $T_{i+1}$ we ensure residuum $1$ or an odd $S$ by what we have seen before.

As the fin of $T_i$ is even, if something is extracted from row $2i-1$ during extracting $T_i$ this row has odd many empty cells afterwards. Suppose that row $2i-1$ has even many empty cells and nothing is extracted from it. Row $2i+1$ has as many $i$'s as there are $(i+1)$'s in row $2i$ and $(i+2)$'s in row $2i-1$. Otherwise $i$ would not be \emph{shifted} and there are to less $j$ involved for \emph{merge}. (If an $i$ would not end up below an $(i+1)$, that is below an $(i+2)$, something would have been extracted from row $2i-1$.) Moreover we can conclude that no $i$ was put to row $2i+1$ in \emph{correct parity}. By the same argument $i$'s that are in row $2i$ have exactly as many $(i+1)$'s in row $2i-1$. This number is even because nothing is changed in \emph{correct parity}. We can conclude that row $2i$ contains odd many empty cells and several $i$'s, while row $2i-1$ contains even many empty cells, the same number of $(i+1)$'s and an even number of $i$'s. Thus the parity of the rows is different. This is a contradiction.
\end{proof}

\begin{lemma}
If $T_i$ is of type 1 and the tail root is a gap then either $i\neq \ell(\mu)$ or the tail root is odd.
\end{lemma}

\begin{remark}
Once we have shown \eqref{eq:HeightProperty}, \eqref{eq:Height_Property} and  \eqref{eq:GapProperty}, it follows that the tail root needs an even slot to the right if it is a gap. In the case that $S$ is odd, this would be the fin. Due to \eqref{eq:Height_Property} this is smaller by at least one than $S^L(1)$, which makes the tail root smaller than or equal to it.
\end{remark}

\begin{proof}
We consider a tableau $T_i$ of type 1 such that the tail root is a gap and even and $i=\ell(\mu)$. Thus $i$'s are the last numbers in $\tilde{L}$. Therefore, at least one row ends with an odd position once $T_i$ is extracted. Thus $\tilde{L}$ had rows with odd length. Now we consider row $2i+1$ and $2i$ before extracting $T_i$. Those get $T^L$ without tail and $T^R$. Thus they are even. The leftmost $i$ is in row $2i$ or above in order to take something from a different row, which is necessary for the gap. Thus row $2i+1$ has even length which is a contradiction.
\end{proof}

\begin{lemma}
\eqref{eq:TailProperty} holds between two consecutive tableaux.
\end{lemma}

\begin{proof}
All $i$'s have corresponding $(i+1)$'s to the right. Thus, if this is not changed, they take smaller or equal numbers. We show that if \emph{merge} occurs and $i$ ends up in a column to the right, then either another \emph{merge} occurs for $(i+1)$, or a \emph{shift}, or that this was not the corresponding $(i+1)$.

Note that \emph{correct parity} puts $i$'s left so we do not need to consider it.
Moreover $i$'s below $(i+1)$'s are \emph{shifted} together.

A \emph{merge} situation in question occurs if $(i+1)$ is in the same column as $i$. It makes the column of $(i+1)$ shorter by two. $(i+1)$ could not be put one row above by \emph{correct parity} as the row above has the same length. \emph{Merge} puts $(i+1)$ one position upwards and puts $i$ together with another $j$ into the next column. If those columns have the same length after extracting $i$, there would be another $(i+1)$ belonging to another $i$ due to the requirements of the positions right of $j$. Inductively this gives a contradiction. As those columns have different length such that all empty cells in the column of $(i+1)$ have right neighbors (because they had them before merging with $i$). $(i+1)$ either changes column by \emph{shift} or by another \emph{merge} when $T_{i+1}$ is extracted.

Moreover we need to consider further applications of \emph{merge} or \emph{shift}. If $i$ is \emph{shifted} after \emph{merge}, $(i+1)$ can follow this path. When it comes to another \emph{merge} situation, the length of this column is not changed, thus $(i+1)$ can also introduce a \emph{merge} situation if no $(i+1)$ is in this column (compare with above).
\end{proof}

\begin{lemma}
\eqref{eq:HeightProperty} and \eqref{eq:Height_Property} hold between two consecutive tableaux.
\end{lemma}

\begin{proof}
This follows as empty cells form a Young diagram of a partition. If row $2i+2$ was taken for $T^L_i$, row $2i+1$ is taken for $T^R_{i+1}$. The tail cannot consist of more than one position from this row, because to take something from this row an $i$ must change row in \emph{correct parity} or change column in the last if-query. The former cannot happen. The latter can happen only once. The only situation where $T_i^R$ is shorter than $T_{i+1}^L$ without tail is if in both tableaux there was something taken from their left column and put into the right column in the last if-query. In this case both have residuum $1$ and the column is longer by at most two, which is allowed.
\end{proof}

Due to construction also the following holds:
\begin{corollary}
For each gap there is a slot to the right. Thus \eqref{eq:GapProperty} holds. Moreover there are no gaps in $S$. Thus \eqref{eq:SProperty} holds. 
\end{corollary}

Therefore it follows that:

\begin{theorem}
\label{theo:ualgo1LRtabs}
Algorithm~\ref{alg:u1} returns an orthogonal Littlewood-Richardson tableaux of shape $\mu$.
\end{theorem}

\begin{theorem}
\label{theo:algo1ualgo1Inverse}
Algorithm~\ref{alg:1} and Algorithm~\ref{alg:u1} are inverse.
\end{theorem}

\begin{proof}
It suffices to show that one iteration of the outer for-loop (one insertion/extraction of a $T_i$) is inverse. 

We insert an empty cell and one filled with $i$ below and extract the same. Therefore what we have to show is that \emph{merge}, \emph{shift}, \emph{correct parity} and dealing with the non tail parts are inverse. We can consider those separately as they do not interfere.

It is clear that the \emph{shift}-procedures are inverse.

Now we consider the \emph{merge} procedures. It is clear that they act inverse and that after \emph{merge} in one algorithm we also \emph{merge} in the other one. It remains to show that we do not \emph{merge} in any other situation.

Merge in Algorithm~\ref{alg:u1} deals with an $i$. If it was not \emph{merged} to get there, it was inserted there, or there was a \emph{shift}. The former is not possible, as this would mean, that rows were not sorted before or the empty cells did not form a tableaux. The latter is prevented by $j_2>j_1>j_3$.
\emph{Merge} in Algorithm~\ref{alg:1} happens if a row is not sorted. The only procedure that leaves a row unsorted in Algorithm~\ref{alg:u1} is \emph{merge}.

To show see that \emph{correct parity} is inverse we have to show that there cannot be an odd number of $i$'s when there was no \emph{correct parity} during inserting. (When we do correct parity in Algorithm~\ref{alg:u1}, this changes the parity.) In Algorithm~\ref{alg:1} rows $2i-1$ or above get longer if and only if they contain an $i$ or the row below contains an $i$.
Therefore only odd many $i$'s produce a different parity in rows $1,\dots,2i-1$. Row $2i$ without $i$'s has the same parity as row $2i+1$. If that is the wrong one, an $i$ changes column. If row $2i$ has now an even number of $i$'s in it, this is the wrong parity and another $i$ changes column.
Therefore only odd many $i$'s produce a different parity and change place in \emph{correct parity}.

For type 2 and type 3 tableaux we argue that inserting fin and lower tail is inverse to shift the fin to its place in the last if-query.

Columns (non tail parts) are placed to row $2i$ and $2i+1$ due to \eqref{eq:HeightProperty} and \eqref{eq:Height_Property}.
\end{proof}

\subsection{Results}

With Theorems~\ref{theo:algo1WellDef}, \ref{theo:ualgo2WellDef}, \ref{theo:ualgo1LRtabs} and~\ref{theo:algo1ualgo1Inverse} we have proven the following Theorem. This is one of our main results.

\begin{theorem}
Our alternative orthogonal Littlewood-Richardson tableaux $\LRtabsa$ are in bijection with Kwon's orthogonal Littlewood-Richardson tableaux $\LRtabs$. Therefore they also count the multiplicities $c_{\lambda}^{\mu}$ in \eqref{eq:BranchingRule}.
\end{theorem}

\section{The Bijection for $\SOodd$}
\subsection{Formulation of the Bijection}

\begin{definition}[The Bijection for $\SOodd$]
We start with a pair $(Q,L)$ consisting of a standard Young tableau $Q$ in $\SYT(\lambda)$ and an orthogonal Littlewood-Richardson tableau $L$ in $\LRtabs$.

First we use Bijection $A$ (see Section~\ref{sec:BijA}) to change $L$ into an alternative orthogonal Littlewood-Richardson tableau $\tilde{L}$. Now we use $\tilde{L}$ to obtain a larger standard Young tableau $\tilde{Q}$ with row lengths of the same parity as follows. If $e$ is the largest entry in $Q$ we add a cells labeled $e+(\mu_{j+1}+\dots+\mu_{\ell(\mu)})+1$, $e+(\mu_{j+1}+\dots+\mu_{\ell(\mu)})+2$,  $\dots$, $e+(\mu_{j+1}+\dots+\mu_{\ell(\mu)})+\mu_j$ to the spots where cells labeled $j$ are in $\tilde{L}$, such that the numbers in the horizontal strip belonging to $j$ are increasing from left to right. We obtain a new standard Young tableau $\tilde{Q}$ with the same shape as $L$. Moreover the $\mu$ largest entries form a $\mu$-horizontal strip (see Section~\ref{sec:BijB}).

Now we distinguish two cases: If our resulting tableau $\tilde{Q}$ consists of even length rows this is the tableau we will use in Bijection $B$ (see Section~\ref{sec:BijB}). Otherwise, thus when $\tilde{Q}$ consists of $n$ odd length rows, we concatenate the one column tableau filled with $1,2,\dots,n$ from left to $\tilde{Q}$. We obtain an all even rowed standard Young tableau, which we will use in in Bijection $B$.

We continue applying Bijection $B$ to $\tilde{Q}$ and obtain a vacillating tableau $\tilde{V}$ with shape $\emptyset$ and cut-away-shape $\mu$ (shape $\emptyset$ ending with $\mu_{\ell(\mu)}$ $(-\ell(\mu))$'s, $\dots$, $\mu_2$ $(-2)$'s and $\mu_1$ $(-1)$'s, see Section~\ref{sec:BijB}).

Once again we distinguish the two cases from before. If we did not concatenate with a column, we do not change $\tilde{V}$. If we concatenated a column to $\tilde{Q}$, we delete the first $n$ entries of $\tilde{V}$. In this case those always are $1,2,\dots,k,0,-k,\dots,-2,-1$. 
Therefore we obtain once again a vacillating tableau $\tilde{V}$ with shape $\emptyset$ and cut-away-shape $\mu$.

We finish our algorithm by deleting the last $|\mu|=\mu_1+\mu_2+\dots+\mu_k$ entries to obtain a vacillating tableau $V$ of shape $\mu$ and length $r=|\lambda|$.
\end{definition}

In Figures~\ref{fig:StrategyEven} and~\ref{fig:Strategyodd} we illustrate this using an even example for $r=15$, $k=2$, $n=2k+1=5$ and an odd example for $r=17$, $k=3$, $n=2k+1=7$. In Table~\ref{tab:ListOfExamples} in the appendix we provide a list of all tableaux with $r=3$ and $n=5$.

\begin{figure}
\begin{tikzpicture}[scale=0.35]
  \draw (0,5) -- (5,5);
  \draw (0,4) -- (5,4);
  \draw (0,3) -- (4,3);
  \draw (0,2) -- (4,2);
  \draw (0,1) -- (1,1);
  \draw (0,0) -- (1,0);
  \draw (0,5) -- (0,0);
  \draw (1,5) -- (1,0);
  \draw (2,5) -- (2,2);
  \draw (3,5) -- (3,2);
  \draw (4,5) -- (4,2);
  \draw (5,5) -- (5,4);
  \draw (0.5,4.5) node {1};
  \draw (1.5,4.5) node {2};
  \draw (2.5,4.5) node {3};
  \draw (3.5,4.5) node {4};
  \draw (0.5,3.5) node {5};
  \draw (1.5,3.5) node {6};
  \draw (2.5,3.5) node {7};
  \draw (3.5,3.5) node {8};
  \draw (0.5,2.5) node {9};
  \draw (1.5,2.5) node {10};
  \draw (2.5,2.5) node {11};
  \draw (3.5,2.5) node {12};
  \draw (0.5,1.5) node {13};
  \draw (0.5,0.5) node {14};
  \draw (4.5,4.5) node {15};
  \draw (7,3)--(8,3);
  \draw (7,2)--(8,2);
  \draw (6,1)--(8,1);
  \draw (6,0)--(7,0);
  \draw (6,-1)--(7,-1);
  \draw (6,1) -- (6,-1);
  \draw (7,3) -- (7,-1);
  \draw (8,3) -- (8,1);
  \draw (6.5,0.5) node {1};
  \draw (6.5,-0.5) node[blue] {5};
  \draw (7.5,2.5) node {1};
  \draw (7.5,1.5) node[blue] {4};
  \draw (9.5,5)--(10.5,5);
  \draw (9.5,4)--(10.5,4);   
  \draw (8.5,3)--(10.5,3);
  \draw (8.5,2)--(10.5,2);
  \draw (8.5,1)--(10.5,1);
  \draw (8.5,1)--(9.5,1);
  \draw (8.5,0)--(9.5,0);
  \draw (8.5,3) -- (8.5,0);
  \draw (9.5,5) -- (9.5,0);
  \draw (10.5,5) -- (10.5,1);
  \draw (9,2.5) node {1};
  \draw (9,1.5) node {2};
  \draw (9,0.5) node {3};
  \draw (10,4.5) node {1};
  \draw (10,3.5) node {2};
  \draw (10,2.5) node {3};
  \draw (10,1.5)[red] node {4};
  \draw (11,5)--(12,5);
  \draw (11,4)--(12,4);
  \draw (11,3)--(12,3);
  \draw (11,2)--(12,2);
  \draw (11,1)--(12,1);
  \draw (11,5) -- (11,1);
  \draw (12,5) -- (12,1);
  \draw (11.5,4.5) node {1};
  \draw (11.5,3.5) node {2};
  \draw (11.5,2.5) node {3};
  \draw (11.5,1.5) node {4};  
  \draw (0,-1.5) node {};
  \draw (6,6) node {(SYT, oLRT)};
\end{tikzpicture}
\hspace{0.3cm}
\begin{tikzpicture}[scale=0.35]
 \draw (0,0) node{};
 \draw [<->,decorate,
decoration={snake,amplitude=.4mm,segment length=2mm,post length=1mm, pre length=1mm}] (0,5) -- (2,5) node[midway, above] {Bij. $A$};
\end{tikzpicture}
\begin{tikzpicture}[scale=0.35]
  \draw (0,5) -- (5,5);
  \draw (0,4) -- (5,4);
  \draw (0,3) -- (4,3);
  \draw (0,2) -- (4,2);
  \draw (0,1) -- (1,1);
  \draw (0,0) -- (1,0);
  \draw (0,5) -- (0,0);
  \draw (1,5) -- (1,0);
  \draw (2,5) -- (2,2);
  \draw (3,5) -- (3,2);
  \draw (4,5) -- (4,2);
  \draw (5,5) -- (5,4);
  \draw (0.5,4.5) node {1};
  \draw (1.5,4.5) node {2};
  \draw (2.5,4.5) node {3};
  \draw (3.5,4.5) node {4};
  \draw (0.5,3.5) node {5};
  \draw (1.5,3.5) node {6};
  \draw (2.5,3.5) node {7};
  \draw (3.5,3.5) node {8};
  \draw (0.5,2.5) node {9};
  \draw (1.5,2.5) node {10};
  \draw (2.5,2.5) node {11};
  \draw (3.5,2.5) node {12};
  \draw (0.5,1.5) node {13};
  \draw (0.5,0.5) node {14};
  \draw (4.5,4.5) node {15};
  \draw (6,5) -- (12,5);
  \draw (6,4) -- (12,4);
  \draw (6,3) -- (10,3);
  \draw (6,2) -- (10,2);
  \draw (6,1) -- (8,1);
  \draw (6,0) -- (8,0);
  \draw (6,5) -- (6,0);
  \draw (7,5) -- (7,0);
  \draw (8,5) -- (8,0);
  \draw (9,5) -- (9,2);
  \draw (10,5) -- (10,2);
  \draw (11,5) -- (11,4);
  \draw (12,5) -- (12,4);
  \draw (7.5,1.5) node [color=red] {2};
  \draw (7.5,0.5) node [color=blue] {1};
  \draw (11.5,4.5) node [color=blue] {1}; 
  \draw (0,-1.5) node {};
  \draw (6,6) node{(SYT, aoLRT)};
\end{tikzpicture}
\begin{tikzpicture}[scale=0.35]
 \draw (0,0) node{};
 \draw (0,5) node{};
 \draw [<->,decorate,
decoration={snake,amplitude=.4mm,segment length=2mm,post length=1mm, pre length=1mm}] (0,5) -- (2,5);
\end{tikzpicture}
\begin{tikzpicture}[scale=0.35]
  \draw (0,5) -- (6,5);
  \draw (0,4) -- (6,4);
  \draw (0,3) -- (4,3);
  \draw (0,2) -- (4,2);
  \draw (0,1) -- (2,1);
  \draw (0,0) -- (2,0);
  \draw (0,5) -- (0,0);
  \draw (1,5) -- (1,0);
  \draw (2,5) -- (2,0);
  \draw (3,5) -- (3,2);
  \draw (4,5) -- (4,2);
  \draw (5,5) -- (5,4);
  \draw (6,5) -- (6,4);
  \draw (0.5,4.5) node {1};
  \draw (1.5,4.5) node {2};
  \draw (2.5,4.5) node {3};
  \draw (3.5,4.5) node {4};
  \draw (0.5,3.5) node {5};
  \draw (1.5,3.5) node {6};
  \draw (2.5,3.5) node {7};
  \draw (3.5,3.5) node {8};
  \draw (0.5,2.5) node {9};
  \draw (1.5,2.5) node {10};
  \draw (2.5,2.5) node {11};
  \draw (3.5,2.5) node {12};
  \draw (0.5,1.5) node {13};
  \draw (0.5,0.5) node {14};
  \draw (4.5,4.5) node {15};
  \draw (1.5,1.5) node [color=red] {16};
  \draw (1.5,0.5) node [color=blue] {17};
  \draw (5.5,4.5) node [color=blue] {18}; 
  \draw (5,1) node {$\mu=(2,1)$};
  \draw (0,-1.5) node {};
  \draw (3.8,6) node{(SYT even, part.)};
\end{tikzpicture}\\
\begin{tikzpicture}[scale=0.35]
 \draw (0,0) node{};
 \draw [<->,decorate,
decoration={snake,amplitude=.4mm,segment length=2mm,post length=1mm, pre length=1mm}] (0,5) -- (2,5) node[midway, above] {Bij. $B$};
\end{tikzpicture}
\begin{tikzpicture}[scale=0.28]
\draw (0,0)--(1,1)--(2,2)--(3,3)--(4,4)--(5,4)--(6,4)--(7,4)--(8,4)--(9,3)--(10,2)--(11,2)--(12,2)--(13,2)--(14,1)--(15,2);
\draw[red] (15,2)--(16,2);
\draw[blue] (16,2)--(17,1)--(18,0);
\draw (4,-1)--(5,0)--(6,-1)--(7,0)--(8,1);
\draw (10,1)--(11,1)--(12,1)--(13,0);
\draw[red] (15,0)--(16,-1);
\draw (13.5,3.5) node {$\mu=(2,1)$};
\draw (8,5.5) node {(vac. tab. shape $\emptyset$, even, part.)};
\end{tikzpicture}
\begin{tikzpicture}[scale=0.35]
 \draw (0,0) node{};
 \draw (0,5) node{};
 \draw [<->,decorate,
decoration={snake,amplitude=.4mm,segment length=2mm,post length=1mm, pre length=1mm}] (0,5) -- (2,5);
\end{tikzpicture}
\begin{tikzpicture}[scale=0.38]
\draw (0,0)--(1,1)--(2,2)--(3,3)--(4,4)--(5,4)--(6,4)--(7,4)--(8,4)--(9,3)--(10,2)--(11,2)--(12,2)--(13,2)--(14,1)--(15,2);
\draw (4,-1)--(5,0)--(6,-1)--(7,0)--(8,1);
\draw (10,1)--(11,1)--(12,1)--(13,0);
\draw (6,5.5) node {vacillating tableaux};
\end{tikzpicture}
\caption{The strategy of our bijection outlined in an even case (with $r=15$ and $k=2$)}
\label{fig:StrategyEven}
\end{figure}
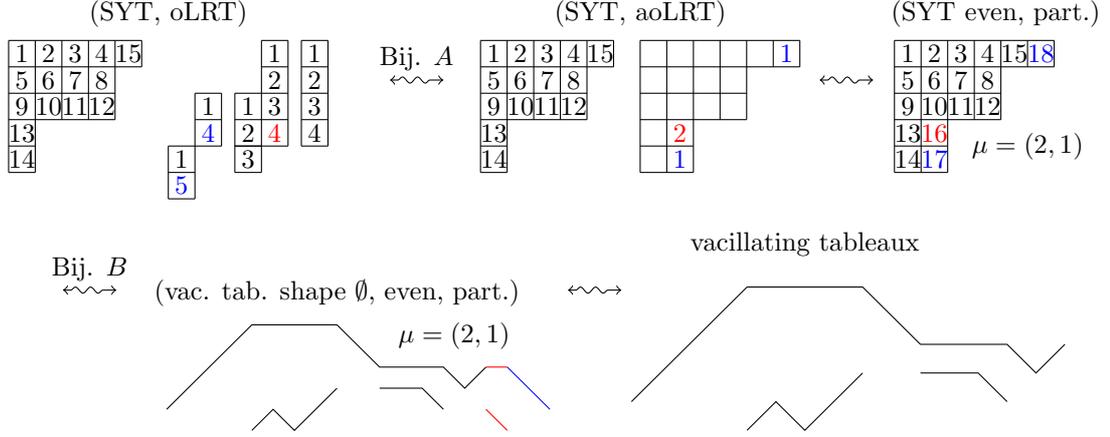

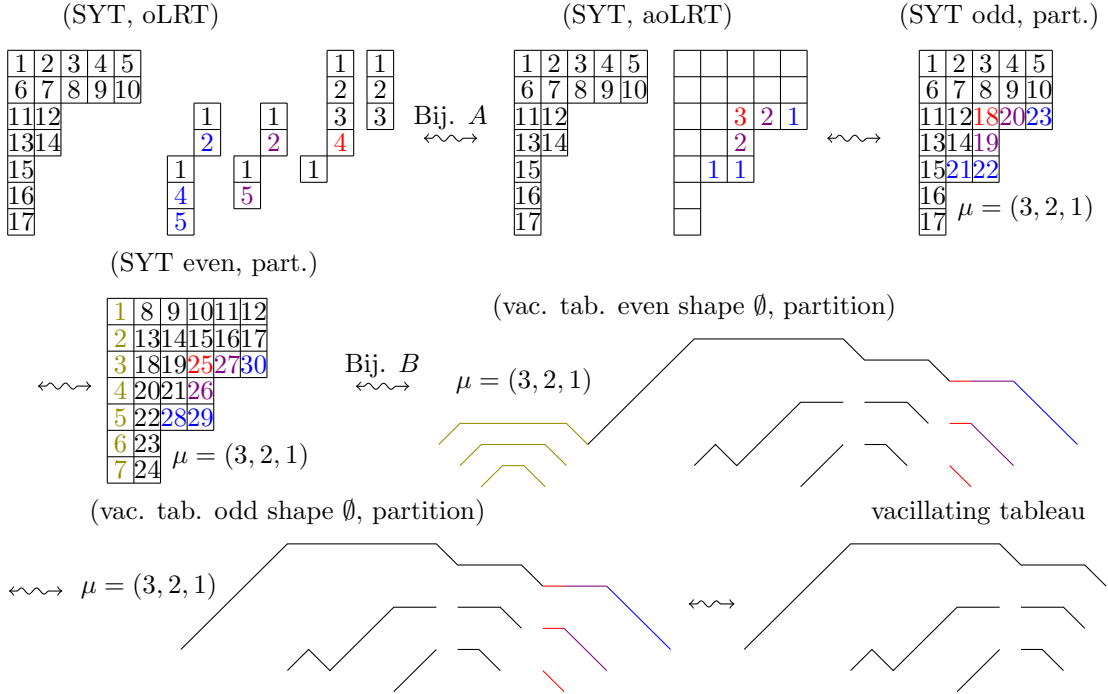
\begin{figure}
\begin{tikzpicture}[scale=0.35]
  \draw (0,7) -- (5,7);
  \draw (0,6) -- (5,6);
  \draw (0,5) -- (5,5);
  \draw (0,4) -- (2,4);
  \draw (0,3) -- (2,3);
  \draw (0,2) -- (1,2);
  \draw (0,1) -- (1,1);
  \draw (0,0) -- (1,0);
  \draw (0,7) -- (0,0);
  \draw (1,7) -- (1,0);
  \draw (2,7) -- (2,3);
  \draw (3,7) -- (3,5);
  \draw (4,7) -- (4,5);
  \draw (5,7) -- (5,5);
  \draw (0.5,6.5) node {1};
  \draw (1.5,6.5) node {2};
  \draw (2.5,6.5) node {3};
  \draw (3.5,6.5) node {4};
  \draw (4.5,6.5) node {5};
  \draw (0.5,5.5) node {6};
  \draw (1.5,5.5) node {7};
  \draw (2.5,5.5) node {8};
  \draw (3.5,5.5) node {9};
  \draw (4.5,5.5) node {10};
  \draw (0.5,4.5) node {11};
  \draw (1.5,4.5) node {12};
  \draw (0.5,3.5) node {13};
  \draw (1.5,3.5) node {14};
  \draw (0.5,2.5) node {15};
  \draw (0.5,1.5) node {16};
  \draw (0.5,0.5) node {17};
  \draw (7,5) -- (8,5);
  \draw (7,4) -- (8,4);
  \draw (6,3) -- (8,3);
  \draw (6,2) -- (7,2);
  \draw (6,1) -- (7,1);
  \draw (6,0) -- (7,0);
  \draw (6,3) -- (6,0);
  \draw (7,5) -- (7,0);
  \draw (8,5) -- (8,3);
  \draw (7.5,4.5) node {1};
  \draw (7.5,3.5)[blue] node {2};
  \draw (6.5,2.5) node {1};
  \draw (6.5,1.5)[blue] node {4};
  \draw (6.5,0.5)[blue] node {5};
  \draw (9.5,5) -- (10.5,5);
  \draw (9.5,4) -- (10.5,4);
  \draw (8.5,3) -- (10.5,3);
  \draw (8.5,2) -- (9.5,2);
  \draw (8.5,1) -- (9.5,1);
  \draw (8.5,3) -- (8.5,1);
  \draw (9.5,5) -- (9.5,1);
  \draw (10.5,5) -- (10.5,3);
  \draw (10,4.5) node {1};
  \draw (10,3.5)[violet] node {2};
  \draw (9,2.5) node {1};
  \draw (9,1.5)[violet] node {5};
  \draw (12,7) -- (13,7);
  \draw (12,6) -- (13,6);
  \draw (12,5) -- (13,5);
  \draw (12,4) -- (13,4);
  \draw (11,3) -- (13,3);
  \draw (11,2) -- (12,2);
  \draw (11,3) -- (11,2);
  \draw (12,7) -- (12,2);
  \draw (13,7) -- (13,3);
  \draw (12.5,6.5) node {1};
  \draw (12.5,5.5) node {2};
  \draw (12.5,4.5) node {3};
  \draw (12.5,3.5)[red] node {4};
  \draw (11.5,2.5) node {1};
  \draw (13.5,7) -- (14.5,7);
  \draw (13.5,6) -- (14.5,6);
  \draw (13.5,5) -- (14.5,5);
  \draw (13.5,4) -- (14.5,4);
  \draw (13.5,7) -- (13.5,4);
  \draw (14.5,7) -- (14.5,4);
  \draw (14,6.5) node {1};
  \draw (14,5.5) node {2};
  \draw (14,4.5) node {3};
  \draw (5,8.3) node {(SYT, oLRT)};
\end{tikzpicture}
\begin{tikzpicture}[scale=0.35]
 \draw (0,0) node{};
 \draw [<->,decorate,
decoration={snake,amplitude=.4mm,segment length=2mm,post length=1mm, pre length=1mm}] (0,3.5) -- (2,3.5) node[midway, above] {Bij. $A$};
\end{tikzpicture}
\begin{tikzpicture}[scale=0.35]
  \draw (0,7) -- (5,7);
  \draw (0,6) -- (5,6);
  \draw (0,5) -- (5,5);
  \draw (0,4) -- (2,4);
  \draw (0,3) -- (2,3);
  \draw (0,2) -- (1,2);
  \draw (0,1) -- (1,1);
  \draw (0,0) -- (1,0);
  \draw (0,7) -- (0,0);
  \draw (1,7) -- (1,0);
  \draw (2,7) -- (2,3);
  \draw (3,7) -- (3,5);
  \draw (4,7) -- (4,5);
  \draw (5,7) -- (5,5);
  \draw (0.5,6.5) node {1};
  \draw (1.5,6.5) node {2};
  \draw (2.5,6.5) node {3};
  \draw (3.5,6.5) node {4};
  \draw (4.5,6.5) node {5};
  \draw (0.5,5.5) node {6};
  \draw (1.5,5.5) node {7};
  \draw (2.5,5.5) node {8};
  \draw (3.5,5.5) node {9};
  \draw (4.5,5.5) node {10};
  \draw (0.5,4.5) node {11};
  \draw (1.5,4.5) node {12};
  \draw (0.5,3.5) node {13};
  \draw (1.5,3.5) node {14};
  \draw (0.5,2.5) node {15};
  \draw (0.5,1.5) node {16};
  \draw (0.5,0.5) node {17};
  \draw (6,7) -- (11,7);
  \draw (6,6) -- (11,6);
  \draw (6,5) -- (11,5);
  \draw (6,4) -- (11,4);
  \draw (6,3) -- (9,3);
  \draw (6,2) -- (9,2);
  \draw (6,1) -- (7,1);
  \draw (6,0) -- (7,0);
  \draw (6,7) -- (6,0);
  \draw (7,7) -- (7,0);
  \draw (8,7) -- (8,2);
  \draw (9,7) -- (9,2);
  \draw (10,7) -- (10,4);
  \draw (11,7) -- (11,4);
  \draw (7.5,2.5)[blue] node {1};
  \draw (8.5,2.5)[blue] node {1};
  \draw (10.5,4.5)[blue] node {1};
  \draw (8.5,3.5)[violet] node {2};
  \draw (9.5,4.5)[violet] node {2};
  \draw (8.5,4.5)[red] node {3};
  \draw (5,8.3) node {(SYT, aoLRT)};
\end{tikzpicture}
\begin{tikzpicture}[scale=0.35]
 \draw (0,0) node{};
 \draw [<->,decorate,
decoration={snake,amplitude=.4mm,segment length=2mm,post length=1mm, pre length=1mm}] (0,3.5) -- (2,3.5);
\end{tikzpicture}
\begin{tikzpicture}[scale=0.35]
  \draw (0,7) -- (5,7);
  \draw (0,6) -- (5,6);
  \draw (0,5) -- (5,5);
  \draw (0,4) -- (5,4);
  \draw (0,3) -- (3,3);
  \draw (0,2) -- (3,2);
  \draw (0,1) -- (1,1);
  \draw (0,0) -- (1,0);
  \draw (0,7) -- (0,0);
  \draw (1,7) -- (1,0);
  \draw (2,7) -- (2,2);
  \draw (3,7) -- (3,2);
  \draw (4,7) -- (4,4);
  \draw (5,7) -- (5,4);
  \draw (0.5,6.5) node {1};
  \draw (1.5,6.5) node {2};
  \draw (2.5,6.5) node {3};
  \draw (3.5,6.5) node {4};
  \draw (4.5,6.5) node {5};
  \draw (0.5,5.5) node {6};
  \draw (1.5,5.5) node {7};
  \draw (2.5,5.5) node {8};
  \draw (3.5,5.5) node {9};
  \draw (4.5,5.5) node {10};
  \draw (0.5,4.5) node {11};
  \draw (1.5,4.5) node {12};
  \draw (0.5,3.5) node {13};
  \draw (1.5,3.5) node {14};
  \draw (0.5,2.5) node {15};
  \draw (0.5,1.5) node {16};
  \draw (0.5,0.5) node {17};
  \draw (2.5,4.5) node[red] {18};
  \draw (2.5,3.5) node[violet] {19};
  \draw (3.5,4.5) node[violet] {20};
  \draw (1.5,2.5) node[blue] {21};
  \draw (2.5,2.5) node[blue] {22};
  \draw (4.5,4.5) node[blue] {23};
  \draw (3,8.3) node {(SYT odd, part.)};
    \draw (4,1) node {$\mu=(3,2,1)$};
\end{tikzpicture}

\begin{tikzpicture}[scale=0.35]
 \draw (0,-0.5) node{};
 \draw (0,5) node{};
 \draw [<->,decorate,
decoration={snake,amplitude=.4mm,segment length=2mm,post length=1mm, pre length=1mm}] (0,3) -- (2,3);
\end{tikzpicture}
\begin{tikzpicture}[scale=0.35]
  \draw (-1,7) -- (5,7);
  \draw (-1,6) -- (5,6);
  \draw (-1,5) -- (5,5);
  \draw (-1,4) -- (5,4);
  \draw (-1,3) -- (3,3);
  \draw (-1,2) -- (3,2);
  \draw (-1,1) -- (1,1);
  \draw (-1,0) -- (1,0);
  \draw (-1,7) -- (-1,0);
  \draw (0,7) -- (0,0);
  \draw (1,7) -- (1,0);
  \draw (2,7) -- (2,2);
  \draw (3,7) -- (3,2);
  \draw (4,7) -- (4,4);
  \draw (5,7) -- (5,4);
  \draw (-0.5,6.5) node[olive] {1};
  \draw (-0.5,5.5) node[olive] {2};
  \draw (-0.5,4.5) node[olive] {3};
  \draw (-0.5,3.5) node[olive] {4};
  \draw (-0.5,2.5) node[olive] {5};
  \draw (-0.5,1.5) node[olive] {6};
  \draw (-0.5,0.5) node[olive] {7};
  \draw (0.5,6.5) node {8};
  \draw (1.5,6.5) node {9};
  \draw (2.5,6.5) node {10};
  \draw (3.5,6.5) node {11};
  \draw (4.5,6.5) node {12};
  \draw (0.5,5.5) node {13};
  \draw (1.5,5.5) node {14};
  \draw (2.5,5.5) node {15};
  \draw (3.5,5.5) node {16};
  \draw (4.5,5.5) node {17};
  \draw (0.5,4.5) node {18};
  \draw (1.5,4.5) node {19};
  \draw (0.5,3.5) node {20};
  \draw (1.5,3.5) node {21};
  \draw (0.5,2.5) node {22};
  \draw (0.5,1.5) node {23};
  \draw (0.5,0.5) node {24};
  \draw (2.5,4.5) node[red] {25};
  \draw (2.5,3.5) node[violet] {26};
  \draw (3.5,4.5) node[violet] {27};
  \draw (1.5,2.5) node[blue] {28};
  \draw (2.5,2.5) node[blue] {29};
  \draw (4.5,4.5) node[blue] {30};
  \draw (3,8.3) node {(SYT even, part.)};
  \draw (4,1) node {$\mu=(3,2,1)$};
\end{tikzpicture}
\begin{tikzpicture}[scale=0.35]
 \draw (0,0) node{};
 \draw [<->,decorate,
decoration={snake,amplitude=.4mm,segment length=2mm,post length=1mm, pre length=1mm}] (0,3.5) -- (2,3.5) node[midway, above] {Bij. $B$};
\end{tikzpicture}
\begin{tikzpicture}[scale=0.28]
\draw[olive] (0,0) -- (1,1)--(2,1)--(3,1)--(4,1)--(5,1)--(6,1)--(7,0);
\draw (7,0)--(8,1)--(9,2)--(10,3)--(11,4)--(12,5)--(13,5)--(14,5)--(15,5)--(16,5)--(17,5)--(18,5)--(19,5)--(20,4)--(21,4)--(22,4)--(23,4)--(24,3);
\draw[red] (24,3)--(25,3);
\draw[violet] (25,3)--(26,3)--(27,3);
\draw[blue] (27,3)--(28,2)--(29,1)--(30,0);
\draw[olive] (1,-1)--(2,0)--(3,0)--(4,0)--(5,0)--(6,-1);
\draw (12,-1)--(13,0)--(14,-1)--(15,0)--(16,1)--(17,2)--(18,2)--(19,2);
\draw (20,2)--(21,2)--(22,2)--(23,1);
\draw [red](24,1)--(25,1);
\draw[violet] (25,1)--(26,0)--(27,-1);
\draw[olive] (2,-2)--(3,-1)--(4,-1)--(5,-2);
\draw (17,-2)--(18,-1)--(19,0);
\draw (20,0)--(21,0)--(22,-1);
\draw[red] (24,-1)--(25,-2);
\draw (12,6.5) node {(vac. tab. even shape $\emptyset$, partition)};
\draw (4,3) node{$\mu=(3,2,1)$};
\end{tikzpicture}
\begin{tikzpicture}[scale=0.35]
 \draw (0,-0.5) node{};
 \draw (0,5) node{};
 \draw [<->,decorate,
decoration={snake,amplitude=.4mm,segment length=2mm,post length=1mm, pre length=1mm}] (0,3) -- (2,3);
\end{tikzpicture}
\begin{tikzpicture}[scale=0.28]
\draw (7,0)--(8,1)--(9,2)--(10,3)--(11,4)--(12,5)--(13,5)--(14,5)--(15,5)--(16,5)--(17,5)--(18,5)--(19,5)--(20,4)--(21,4)--(22,4)--(23,4)--(24,3);
\draw[red] (24,3)--(25,3);
\draw[violet] (25,3)--(26,3)--(27,3);
\draw[blue] (27,3)--(28,2)--(29,1)--(30,0);
\draw (12,-1)--(13,0)--(14,-1)--(15,0)--(16,1)--(17,2)--(18,2)--(19,2);
\draw (20,2)--(21,2)--(22,2)--(23,1);
\draw[red] (24,1)--(25,1);
\draw[violet] (25,1)--(26,0)--(27,-1);
\draw (17,-2)--(18,-1)--(19,0);
\draw (20,0)--(21,0)--(22,-1);
\draw[red] (24,-1)--(25,-2);
\draw (12,6.5) node {(vac. tab. odd shape $\emptyset$, partition)};
\draw (5.5,3) node{$\mu=(3,2,1)$};
\end{tikzpicture}
\begin{tikzpicture}[scale=0.3]
 \draw (0,-0.5) node{};
 \draw (0,5) node{};
 \draw [<->,decorate,
decoration={snake,amplitude=.4mm,segment length=2mm,post length=1mm, pre length=1mm}] (0,3) -- (2,3);
\end{tikzpicture}
\begin{tikzpicture}[scale=0.28]
\draw (7,0)--(8,1)--(9,2)--(10,3)--(11,4)--(12,5)--(13,5)--(14,5)--(15,5)--(16,5)--(17,5)--(18,5)--(19,5)--(20,4)--(21,4)--(22,4)--(23,4)--(24,3);
\draw (12,-1)--(13,0)--(14,-1)--(15,0)--(16,1)--(17,2)--(18,2)--(19,2);
\draw (20,2)--(21,2)--(22,2)--(23,1);
\draw (17,-2)--(18,-1)--(19,0);
\draw (20,0)--(21,0)--(22,-1);
\draw (18,6.5) node {vacillating tableau};
\end{tikzpicture}
\caption{The strategy of our bijection outlined in an odd case (with $r=17$ and $k=3$)}
\label{fig:Strategyodd}
\end{figure}

As we know the inverse of Bijection $A$ and Bijection $B$, the inverse bijection is easily defined:

\begin{definition}[The other direction of our bijection] We start with a vacillating tableau $V$ of shape $\mu$ and length $r$, and add $\mu_k$ $(-k)$'s, $\mu_{k-1}$ $(-k+1)$'s, $\dots$ and $\mu_1$ $(-1)$'s to obtain a vacillating tableau $V$ of shape $\emptyset$ and cut-away-shape $\mu$. If this has odd length we furthermore add $1,2,\dots,k,0,-k,\dots,-2,-1$ in the front. Next we apply the inverse of Bijection $B$ to obtain a standard Young tableau $\tilde{Q}$.

If we added $1,2,\dots,k,0,-k,\dots,-2,-1$ to $\tilde{V}$, we cancel the smallest $n$ entries of $\tilde{Q}$ now. Those are in the first column in increasing order. If we did so, we furthermore reduce each entry of $\tilde{Q}$ by $n$ afterwards, to obtain a standard Young tableau again.

We obtain $Q$ by deleting the $|\mu|$ largest entries in $\tilde{Q}$. $Q$ is a standard Young tableau of shape $\lambda$. Moreover we define $\tilde{L}$ to be the reverse skew semistandard tableau of the same outer shape as $\tilde{Q}$ and inner shape $\lambda$. We fill cells, where entries of $\mu_j$ are in $\tilde{Q}$ with $j$. Due to the properties of $\mu$-horizontal strips, $\tilde{L}$ is an alternative orthogonal Littlewood-Richardson tableau.

Finally we apply the inverse of Bijection $A$ to obtain $L$, an orthogonal Littlewood-Richardson tableau in $\LRtabs$ where $\lambda$ is a partition such that $\lambda\vdash r$ and $\mu\leq \lambda$.
\end{definition}

The strategy we use is similar as in the case $n=3$ in \cite{3erAlgo}. There are two main differences. The first is, that in \cite{3erAlgo} we do not calculate the alternative Littlewood-Richardson tableau but go directly to the $\mu$-horizontal strip. The second is, that we attach numbers for odd tableaux to the left of $Q$ in order to obtain an all even rowed tableau. However, in case $n=3$ we know, that concatenating of standard Young tableaux where all row lengths have the same parity corresponds to concatenation of vacillating tableaux with shape $\emptyset$. Therefore, for $n=3$ both strategies are the same.

\begin{theorem}
Let $\lambda\vdash r$, $\ell(\lambda)\leq n(=2k+1)$, $\ell(\mu)\leq k$.
The map defined in this section maps a pair $(Q,L)$ consisting of a standard Young tableau $Q$ in $\SYT(\lambda)$ and an orthogonal Littlewood-Richardson tableau $L$ in $\LRtabs$ to a vacillating tableau of length $r$ and shape $\mu$. Moreover it is well-defined, bijective and descent preserving.
\end{theorem}

\begin{proof}
We prove that every algorithm we use defines a well-defined mapping in Theorems~\ref{theo:algo1WellDef}, \ref{theo:ualgo2WellDef},  \ref{theo:algo2welldefvactab} and~\ref{theo:ualgo2welldefSYT}. Those also show, together with Theorems~\ref{theo:muhorizontal} and~\ref{theo:ualgo1LRtabs}, that they produce the desired objects.

To see that it is bijective, we argue that the algorithms we use are inverse in Theorems~\ref{theo:algo1ualgo1Inverse} and~\ref{theo:algo2ualgo2Inverse}. Moreover the procedure we describe between alternative orthogonal Littlewood-Richardson tableaux and $\mu$-horizontal strips is inverse by definition. The procedure we describe by adding and deleting the first positions is inverse by Theorem~\ref{theo:10-1}.

It is descent preserving as Bijection $B$ is descent-preserving (see Theorem~\ref{theo:algo2DescPres}).
\end{proof}

\section{Bijection $B$}
\label{sec:BijB}

\subsection{Formulation of Bijection $B$}

Bijection $B$ is formulated by Algorithms~\ref{alg:2} and~\ref{alg:u2} which are inverse. It maps a standard Young tableau with $n=2k+1$ possibly empty rows, whose lengths are even, containing a $\mu$-horizontal strip, to a vacillating tableau of dimension $k$, shape $\emptyset$ and cut-away-shape $\mu$.

To formulate those algorithms we introduce some notation in Table~\ref{tab:Notation}. Note that at some points we have left to right opposites for those algorithms. When looking at weight $\emptyset$ words, which will not always be the case while executing Algorithms~\ref{alg:2} and~\ref{alg:u2}, the definitions are the same.

We refer to parts or operations in the algorithms by the comments placed next to them.

\begin{table}[h]
\centering
\caption{Notation for Algorithms~\ref{alg:2} and~\ref{alg:u2}}
\label{tab:Notation}
\begin{tabular}{|p{3.3cm}|p{10.7cm}|}
\hline
A \emph{labeled word} $w$ with letters in $\{\pm 1, \dots, \pm k,0\}$. & A word, where each letter is labeled by an integer $1 \leq i \leq r$ strictly increasing from left to right.
Each position consists of a label and an entry. We denote by $w(p)$ the entry of $w$ labeled with $p$.\\
\hline
A position $q$ is on $l$-level $m$ in Algorithm~\ref{alg:2} (respectively Algorithm~\ref{alg:u2}). & The minimum of the following to sums over entries with absolute value $l$ is $-l\cdot m$ (respectively $+l\cdot m$). For the first sum we consider entries strictly to the \emph{right} (respectively \emph{left}) of $q$. For the second one we consider entries to the \emph{right} (respectively \emph{left}) including $q$.

Illustration of positions on level $m$:
\begin{tikzpicture}[scale=0.35]
\draw[dotted] (1,0)--(6,0);
\draw (0.5,0) node{$m$};
\draw (1.5,1)--(2.5,0);
\draw (3,0)--(4,1);
\draw (4.5,0) -- (5.5,0);
\end{tikzpicture}\\
\hline
A position $q$ is a height violation in $l$. & The $l$-level of $q$ is smaller than the $(l+1)$-level of $q$. If $w(q)=\pm (l+1)$ we take  the $(l+1)$-level plus one instead.
\\
\hline
Insert $q$ with $l$. & We insert a new position with entry $l$ and label $q$ in that way, that the labels are still sorted.\\
\hline
Ignore $q$. & Act as if this position was not here, for example in level calculations.\\
\hline
A position $p$ is a 3-row-position in $j$. & $p$ is either the rightmost $0$ of an odd sequence of $0$'s on $j$-level one or a $0$ that is on $j$-level two or higher.
\\
\hline
A position $p$ is a 2-row-position in $j$. & $p$ is either a $j$ on $j$-level one or the leftmost $0$ of a sequence of $0$'s.
\\
\hline
A position $p$ is in a $j$-even (respectively odd) position & The number of positions $q$ strictly to the left with $w(q)\in\{0,\pm j\}$ is even (respectively odd).\\
\hline
\end{tabular}
\end{table}

\begin{algorithm}
\label{alg:2}
\SetKwInOut{Input}{input}\SetKwInOut{Output}{output}
\Input{$n=2k+1$, standard Young tableau $Q$ of at most $n$ rows, all rows of even length}
\Output{vacillating tableau $V$, dimension $k$, weight $\emptyset$, same number of entries as $Q$}

\SetKwIF{FOR}{}{}{for}{do}{}{}{}
\SetKwIF{WHILE}{}{}{while}{do}{}{}{}

let $w$ be word $(1,-1,\dots,1,-1)$ labeled by first row elements of $Q$\tcc*{insert row 1} 

\uFOR(\tcc*[f]{insert row $i$}){$i=2,3,\dots,n$}{
$j:=\lfloor i/2 \rfloor$; unmark everything\;
  \lIf(\tcc*[f]{initialize $j$}){$i$ even}
    {change $0$-entries of $w$ into $j,-j,\dots,j,-j$}
  \uFOR{pairs of elements $a,b$ in row $i$, start with the rightmost, go to left}
    {
    $a_1:=a$, $b_1:=b$, $a_l:=b_l:=0$ for $l=2,3,\dots,j+1$\; 
    \lIf(\tcc*[f]{$b$}){$b$ is largest position so far}{insert $b_1$ with $-1$}
    let $p$ be rightmost position so far, $\tilde{p}$ be next position left of $p$ with $w(\tilde{p})\in \{0,\pm j\}$\;   
    \uWHILE{$a_{j+1}<p$ {\normalfont or} $w(p)\notin\{0,\pm j\}$}
      {
      \If{$p<b_l$, $p\neq a_l$, $w(p)=-l$ for an $l<j$, $a_{l+1}=0$}
      {
        \lIf(\tcc*[f]{$b_{l+1}$}){$p$ not marked, $b_{l+1}=0$}
        {$w(p):=-l-1$, $b_{l+1}:=p$}
     \lElseIf(\tcc*[f]{$a_{l+1}$}){$p<a_l$, $p<b_{l+1}$}
        {$w(p):=-l-1$, $a_{l+1}:=p$} 
      }
      \If(\tcc*[f]{$i$ even}){$i$ is even, $w(p)\in\{0,\pm j\}$}
      {  
        \uIf(\tcc*[f]{adjust separation point}){$b_j<p$, $w(\tilde{p}),w(p)=j,-j$}
        {
        for $l<j$ change $\pm l$ on $l$-level $0$ between $p$ and $\tilde{p}$ into $\pm (l+1)$, if $p<b_l$, $b_{l+1}=0$ ignore $b_l$, if $p<a_l$, $a_{l+1}=0$ ignore $a_l$; mark changed positions; change $-j,j$ between $p$ and $\tilde{p}$ into $0,0$\;        
        }
        \uElseIf(\tcc*[f]{mark it + connect}){$a_j<p$, $w(\tilde{p}),w(p)=j,-j$}{$w(\tilde{p}),w(p):=0,0$; for $l<j$ mark $\pm l$ on $l$-level $0$ between $p$ and $\tilde{p}$, if $p<a_l$, $a_{l+1}=0$ ignore $a_l$\;}
        \uElseIf(\tcc*[f]{$a_{j+1}$ 1}){$p=a_j$, $w(\tilde{p})=0$ on $j$-level $1$}{$w(\tilde{p}),w(p):=j,0$, $a_{j+1}:=\tilde{p}$\;}       
        \ElseIf(\tcc*[f]{$a_{j+1}$ 2}){$p<a_j$, $w(p)=-j$, $a_{j+1}=0$}{$w(p):=j$, $a_{j+1}:=p$\;}
        \lIf(\tcc*[f]{$b_{j+1}$}){$p<b_j$, $b_{j+1}=0$}{$b_{j+1}:=p$}
      }
      \If(\tcc*[f]{$i$ odd}){$i$ is odd, $w(p)\in\{0,\pm j\}$}
      {
      \uIf(\tcc*[f]{adjust separation point}){$b_{j+1}<p$, $w(p),w(\tilde{p})=0,0$, $p$ $j$-even position on $j$-level $1$ if $b_j<p$ or $2$ if $p<b_j$}
      {for $l<j$ change $\pm l$ on $l$-level $0$ between $p$ and $\tilde{p}$ into $\pm (l+1)$, if $p<b_l$, $b_{l+1}=0$ ignore $b_l$, if $p<a_l$, $a_{l+1}=0$ ignore $a_l$; mark changed positions\;}
        \lElseIf(\tcc*[f]{connect}){$a_{j+1}<p<b_{j+1}$, $w(p)=j$ on $j$-level $1$ for $p<a_j$ or $0$ for $a_j<p$}
          {$w(\tilde{p}),w(p):=0,0$}
      \uElseIf(\tcc*[f]{mark it + separate}){$a_{j+1}<p<b_{j+1}$, $w(\tilde{p}),w(p)=0,0$, $p$ $j$-even position on $j$-level $2$ if $p<a_j$ or $1$ if $a_j<p$}{$w(\tilde{p}),w(p):=-j,j$; for $l<j$ mark $\pm l$ on $l$-level $0$ between $p$ and $\tilde{p}$, if $p<a_l$, $a_{l+1}=0$ ignore $a_l$\;}
       
      \ElseIf{$p<b_j$, $p\neq a_j$, $w(p)=-j$, $a_{j+1}=0$}
      {
        \lIf(\tcc*[f]{$b_{j+1}$}){$p$ not marked, $b_{j+1}=0$}
        {$w(p):=0$, $b_{j+1}:=p$}
     \lElseIf(\tcc*[f]{$a_{j+1}$}){$p<a_j$, $p<b_{j+1}$}
        {$w(p):=0$, $a_{j+1}:=p$} 
      }
    } 
        \If{$p=a_l$ on $l$-level $0$, for an $l<j$, the $l$ to the right is marked}{mark $a_l$\tcc*{mark $a_l$}}    
}
}
}
\caption{Standard Young Tableaux to Vacillating Tableaux}
\end{algorithm} 


\setlength{\interspacetitleruled}{0pt}%
\setlength{\algotitleheightrule}{0pt}%

\begin{algorithm*}[h]
\SetKwIF{FOR}{}{}{\vspace{-0.4cm}}{}{}{}{}
\SetKwIF{WHILE}{}{}{\vspace{-0.4cm}}{}{}{}{}      
      
\FOR{}{
\FOR{}{
\WHILE{}{

     \If(\tcc*[f]{height violation}){$p$ height violation in $l$ for an $l<j$, ($p<a_l$ or $p$ not marked), if $p<a_l$, $a_{l+1}=0$ ignore $a_l$}
       {
        $w(p):=l+1$\;
        \leIf{$a_{l+1}=0$}{$b_{l+1}:=0$}{$a_{l+1}:=0$}
        \lIf{$i$ is even, $a_{j+1}\neq 0$}{$w(a_{j+1}),w(p):=0,0$, $a_{j+1}:=0$}
        \lIf{$i$ is odd, $w(\tilde{p})=0$ on $j$-level $0$}{$w(\tilde{p}):=-j$, $b_{j+1}:=0$}
      }

    \lIf(\tcc*[f]{$b$}){$b$ is between $p$ and the position to the left}{insert $b_1$ with $-1$}    
    \lElseIf(\tcc*[f]{$a$}){$a$ is between those}{insert $a_1$ with $-1$}    
    let $p$ be one position to the left in $w$, change $\tilde{p}$ according to it\;
    }
  }
  do one additional iteration of the inner for-loop with $a=b=0$\;
  }
forget the labels of $w$, set $V=w$\;
\Return $V$\;
\end{algorithm*}

\setlength{\interspacetitleruled}{3pt}%
\setlength{\algotitleheightrule}{0.5pt}%
\begin{algorithm}
\caption{Vacillating Tableaux to Standard Young Tableaux}
\label{alg:u2}
\SetKwInOut{Input}{input}\SetKwInOut{Output}{output}\SetKw{Break}{break}
\Input{$n=2k+1$, vacillating tableau $V$ of dimension $k$, weight $\emptyset$, even length}
\Output{standard Young tableau $Q$, same number of entries as $V$, $n$ rows, all of even length}

\SetKwIF{FOR}{}{}{for}{do}{}{}{}
\SetKwIF{WHILE}{}{}{while}{do}{}{}{}

label $V$ with $1,2,\dots,r$ to obtain a labeled word $w$\;
\uFOR(\tcc*[f]{extract row $i$}){$i=n, n-1,\dots,2$}
  {
  $j:=\lfloor i/2 \rfloor$, unmark everything\;
  \uWHILE{in word $j$ are $\lceil i/j \rceil$-row-positions}
    {
    let $p$ be the second position from left with $w(p)\in\{0,\pm j\}$, let $\tilde{p}$ be the next position left of $p$ with $w(\tilde{p})\in\{0,\pm j\}$\; $a=a_l=b=b_l=r$ for $l=1,2,\dots,j+1$\;
    \uWHILE{$p<b$}
      {
          \If(\tcc*[f]{height violation}){$p$ height violation in $l$, for an $l<j$, ($p<a_l$ or not marked), if $b_{l+1}<p$ and $b_l=r$ ignore next $l$ left of $p$}
        {
        $w(p):=l$\; \leIf{$b_{l+1}=r$}{$a_{l+1}=r$}{$b_{l+1}=r$}
        }
      \If(\tcc*[f]{$i$ odd}){$i$ is odd, $w(p)\in\{\pm j,0\}$}
      { 
        \If(\tcc*[f]{height violation special case}){$w(p)=0$ on $j$-level $0$}
        {$w(p):=j$, $b_{j+1}=r$\;}          
        \uIf{$p$ is 3-row-position, $p<b_{j+1}$}
        {
          \If(\tcc*[f]{mark separation point}){$p$ is on level $1$}
            {mark $p$ and next $0$ right, for $l\leq j$ mark $\pm l$ on $l$-level $0$ between them\;}
          \uIf(\tcc*[f]{adjust it}){$p<a_{j+1}$, right of $p$ for no $l$ there is an unmarked $\pm(l-1)$ between a marked $-l$ and a marked $l$ and no $-1$ is marked}
            {change marked $\pm l$ into $\pm (l-1)$\;}
          \lElseIf(\tcc*[f]{$a_{j+1}$}){$a_{j+1}=r$}{$w(p):=-j$, $a_{j+1}:=p$}     
          \lElseIf(\tcc*[f]{$b_{j+1}$}){$p$ not marked, $b_{j+1}=r$}{$w(p):=-j$, $b_{j+1}:=p$}
        }
        \lElseIf(\tcc*[f]{connect}){$a_{j+1}<p$, $p<b_{j+1}$, $w(p),w(\tilde{p})=0,0$, $p$ $j$-even position, on $j$-level 1 for $a_j<p$ or 2 for $p<a_j$}{$w(p),w(\tilde{p}):=-j,j$}
        \ElseIf{$a_{j+1}<p$, $p<b_{j+1}$, $w(p),w(\tilde{p})=-j,j$ on $j$-level 0}{$w(p),w(\tilde{p})=0,0$\tcc*{separate}}
      }   
            \If(\tcc*[f]{$i$ even}){$i$ is even, $w(p)\in\{\pm j,0\}$}
      {
        \If(\tcc*[f]{height violation}){$w(p)=0$ on $(j-1)$-level $0$}
        {
        $w(p):=j-1$, $w(a_{j+1}):=j$, $a_{j+1}:=a_j:=r$\tcc*{special case}
        }
      \lIf(\tcc*[f]{$b_{j+1}$}){$a_j\leq p$, $b_{j+1}=r$}{$b_{j+1}:=p$} 
        \uIf{$p$ is 2-row-position, $p<a_{j+1}$}
        {
          \If(\tcc*[f]{mark separation point}){$w(p)=0$}
            {mark $p$ and $0$'s directly right, for $l<j$ mark $\pm l$ on $l$-level $0$ between them\;}
          \uIf(\tcc*[f]{adjust it}){$p<a_{j+1}$, right of $p$ for no $l$ there is an unmarked $\pm(l-1)$ between a marked $-l$ and a marked $l$, there is not an unmarked $\pm(j-1)$ between two marked $0$'s and no $-1$ is marked}
            {change marked $\pm l$ into $\pm (l-1)$, change marked $0,0$ into $-j,j$\;}
          \ElseIf{$a_{j+1}=r$}
            {\lIf(\tcc*[f]{$a_{j+1}$ 1}){$w(p)=0$}{$a_{j+1}:=\tilde{p}$, $w(\tilde{p}),w(p):=0,-j$}
            \lElse(\tcc*[f]{$a_{j+1}$ 2}){$a_{j+1}:=p$, $w(p):=-j$}}     
        }    
        \ElseIf(\tcc*[f]{connect}){$w(\tilde{p}),w(p)=0,0$ on $j$-level $0$}
        {$w(\tilde{p}),w(p):=j,-j$\;}
      } 
            
}
}
}
\end{algorithm}

\setlength{\interspacetitleruled}{0pt}%
\setlength{\algotitleheightrule}{0pt}%

\begin{algorithm*}[h]
\SetKwIF{FOR}{}{}{\vspace{-0.4cm}}{}{}{}{}
\SetKwIF{WHILE}{}{}{\vspace{-0.4cm}}{}{}{}{}      
      
\FOR{}{
\FOR{}{
\WHILE{}{

\If{$w(p)=-l$, $a_{l+1}<p$ for an $l$ with $j\geq l>1$, $b_{l}=r$}
      {
       \lIf(\tcc*[f]{$a_l$}){$a_{l}=r$}{$w(p)=-l+1$, $a_l:=p$}
       \lElseIf(\tcc*[f]{$b_l$}){$b_{l+1}<p$, $a_l<p$, not marked}
        {$w(p):=-l+1$, $b_l:=p$} 
       }
      
     \lIf(\tcc*[f]{mark $a_l$}){$p=a_l$ on $l$-level $0$, marked}{mark the next $l$ to the right} 
      \lIf(\tcc*[f]{$a$}){$a_2<p$, $w(p)=-1$, $a_1=r$}{delete $p$, $a_1=a:=p$}
      \lElseIf(\tcc*[f]{$b$}){$a_1<p$, $b_2<p$, $w(p)=-1$}{delete $p$, $b_1=b:=p$}
      let $p$ be one position to the right, change $\tilde{p}$ according to it\;
      }
    \lIf{$a\neq r$}{put $a,b$ in row $j$ of $Q$}
    }
  \lIf(\tcc*[f]{initialize $j$}){$i$ is even}{change entries $\pm j$ of $w$ into $0$}  
  }
put the labels still in $V$ in the first row of $Q$ \tcc*{extract row 1}
\Return $Q$\;
\end{algorithm*}

\subsection{Examples explaining Bijection $B$}

We can draw labeled words like we draw vacillating tableaux as tuple of paths, compare with Example~\ref{ex:VacTab}.

\begin{example}[An easy example to motivate the Algorithm]
We consider the following tableau for $n=2k+1=7$, thus we are going to create $k=3$ paths:

\hspace{5cm}\begin{tikzpicture}[scale=0.35]
  \draw (0,7) -- (6,7);
  \draw (0,6) -- (6,6);
  \draw (0,5) -- (6,5);
  \draw (0,4) -- (4,4);
  \draw (0,3) -- (2,3);
  \draw (0,2) -- (2,2);
  \draw (0,1) -- (2,1);
  \draw (0,7) -- (0,1);
  \draw (1,7) -- (1,1);
  \draw (2,7) -- (2,1);
  \draw (3,7) -- (3,4);
  \draw (4,7) -- (4,4);
  \draw (5,7) -- (5,5);
  \draw (6,7) -- (6,5);
  \draw (0.5,6.5) node {1};
  \draw (1.5,6.5) node {2};
  \draw (2.5,6.5) node {3};
  \draw (0.5,5.5) node {4};
  \draw (3.5,6.5) node {5};
  \draw (1.5,5.5) node {6};
  \draw (0.5,4.5) node {7};
  \draw (2.5,5.5) node {8};
  \draw (1.5,4.5) node {9};
  \draw (2.5,4.5) node {10};
  \draw (0.5,3.5) node {11};
  \draw (1.5,3.5) node {12};
  \draw (0.5,2.5) node {13};
  \draw (1.5,2.5) node {14};
  \draw (0.5,1.5) node {15};
  \draw (3.5,5.5) node {16};
  \draw (3.5,4.5) node {17};
  \draw (1.5,1.5) node {18};
  \draw (4.5,6.5) node {19}; 
  \draw (5.5,6.5) node {20}; 
  \draw (4.5,5.5) node {21}; 
  \draw (5.5,5.5) node {22};     
\end{tikzpicture}
\begin{itemize}
\item We initialize the first path with up, down, up, down, $\dots$, up, down-steps, labeled with the elements of the first row.

\item We insert rows $2$ up to $2k+1$ from top to bottom. For each row we insert pairs of two elements, starting with the rightmost pair, into the topmost path.

\item When we insert a pair $a,b$ into a path, we insert the new positions $a$ and $b$ with down-steps and 
\begin{itemize}
\item if we have not inserted into this path during the insertion process of the previous row, we change the down-step left of a pair $(a,b)$, into an up-step.
\item otherwise we change the next down-step to the left of each, $a$ and $b$ into a horizontal step.
\end{itemize}

If there is a path beneath, we insert these new horizontal-steps as a pair $a,b$ into this path according to this rule.

\item Whenever we finish inserting an odd row, we initialize a new path below (with up- and down steps) and label it with the horizontal steps of the bottommost path so far.
\end{itemize}

So the core of Algorithm~\ref{alg:2} is an an insertion algorithm from standard Young tableaux into vacillating tableaux. Some insertions create horizontal steps and to preserve descents, these are bumped into lower paths.\\
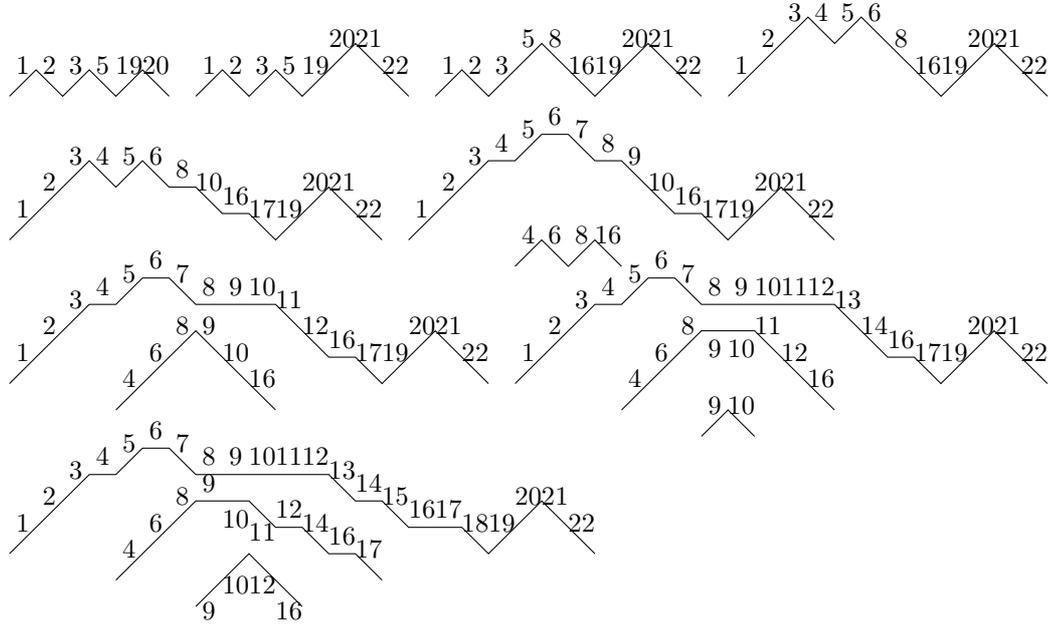
\begin{tikzpicture}[scale=0.35]
\draw (0,0)  --
(1,1) node[midway, above] {1}--
(2,0) node[midway, above] {2}--
(3,1) node[midway, above] {3}--
(4,0) node[midway, above] {5}--
(5,1) node[midway, above] {19}--
(6,0) node[midway, above] {20};
\draw (0,-1.1) node {};
\end{tikzpicture}
\begin{tikzpicture}[scale=0.35]
\draw (0,0)  --
(1,1) node[midway, above] {1}--
(2,0) node[midway, above] {2}--
(3,1) node[midway, above] {3}--
(4,0) node[midway, above] {5}--
(5,1) node[midway, above] {19}--
(6,2) node[midway, above] {20}--
(7,1) node[midway, above] {21}--
(8,0) node[midway, above] {22};
\draw (0,-1.1) node {};
\end{tikzpicture}
\begin{tikzpicture}[scale=0.35]
\draw (0,0)  --
(1,1) node[midway, above] {1}--
(2,0) node[midway, above] {2}--
(3,1) node[midway, above] {3}--
(4,2) node[midway, above] {5}--
(5,1) node[midway, above] {8}--
(6,0) node[midway, above] {16}--
(7,1) node[midway, above] {19}--
(8,2) node[midway, above] {20}--
(9,1) node[midway, above] {21}--
(10,0) node[midway, above] {22};
\draw (0,-1.1) node {};
\end{tikzpicture}
\begin{tikzpicture}[scale=0.35]
\draw (0,0)  --
(1,1) node[midway, above] {1}--
(2,2) node[midway, above] {2}--
(3,3) node[midway, above] {3}--
(4,2) node[midway, above] {4}--
(5,3) node[midway, above] {5}--
(6,2) node[midway, above] {6}--
(7,1) node[midway, above] {8}--
(8,0) node[midway, above] {16}--
(9,1) node[midway, above] {19}--
(10,2) node[midway, above] {20}--
(11,1) node[midway, above] {21}--
(12,0) node[midway, above] {22};
\draw (0,-1.1) node {};
\end{tikzpicture}\vspace{-0.5cm}\\
\begin{tikzpicture}[scale=0.35]
\draw (0,0)  --
(1,1) node[midway, above] {1}--
(2,2) node[midway, above] {2}--
(3,3) node[midway, above] {3}--
(4,2) node[midway, above] {4}--
(5,3) node[midway, above] {5}--
(6,2) node[midway, above] {6}--
(7,2) node[midway, above] {8}--
(8,1) node[midway, above] {10}--
(9,1) node[midway, above] {16}--
(10,0) node[midway, above] {17}--
(11,1) node[midway, above] {19}--
(12,2) node[midway, above] {20}--
(13,1) node[midway, above] {21}--
(14,0) node[midway, above] {22};
\draw (0,-1.1) node {};
\end{tikzpicture}
\begin{tikzpicture}[scale=0.35]
\draw (0,0)  --
(1,1) node[midway, above] {1}--
(2,2) node[midway, above] {2}--
(3,3) node[midway, above] {3}--
(4,3) node[midway, above] {4}--
(5,4) node[midway, above] {5}--
(6,4) node[midway, above] {6}--
(7,3) node[midway, above] {7}--
(8,3) node[midway, above] {8}--
(9,2) node[midway, above] {9}--
(10,1) node[midway, above] {10}--
(11,1) node[midway, above] {16}--
(12,0) node[midway, above] {17}--
(13,1) node[midway, above] {19}--
(14,2) node[midway, above] {20}--
(15,1) node[midway, above] {21}--
(16,0) node[midway, above] {22};
\draw (4,-1)  --
(5,0) node[midway, above] {4}--
(6,-1) node[midway, above] {6}--
(7,0) node[midway, above] {8}--
(8,-1) node[midway, above] {16};
\draw (0,-1.1) node {};
\end{tikzpicture}\vspace{-0.5cm}\\
\begin{tikzpicture}[scale=0.35]
\draw (0,0)  --
(1,1) node[midway, above] {1}--
(2,2) node[midway, above] {2}--
(3,3) node[midway, above] {3}--
(4,3) node[midway, above] {4}--
(5,4) node[midway, above] {5}--
(6,4) node[midway, above] {6}--
(7,3) node[midway, above] {7}--
(8,3) node[midway, above] {8}--
(9,3) node[midway, above] {9}--
(10,3) node[midway, above] {10}--
(11,2) node[midway, above] {11}--
(12,1) node[midway, above] {12}--
(13,1) node[midway, above] {16}--
(14,0) node[midway, above] {17}--
(15,1) node[midway, above] {19}--
(16,2) node[midway, above] {20}--
(17,1) node[midway, above] {21}--
(18,0) node[midway, above] {22};
\draw (4,-1)  --
(5,0) node[midway, above] {4}--
(6,1) node[midway, above] {6}--
(7,2) node[midway, above] {8}--
(8,1) node[midway, above] {9}--
(9,0) node[midway, above] {10}--
(10,-1) node[midway, above] {16};
\draw (0,-2.1) node {};
\end{tikzpicture}
\begin{tikzpicture}[scale=0.35]
\draw (0,0)  --
(1,1) node[midway, above] {1}--
(2,2) node[midway, above] {2}--
(3,3) node[midway, above] {3}--
(4,3) node[midway, above] {4}--
(5,4) node[midway, above] {5}--
(6,4) node[midway, above] {6}--
(7,3) node[midway, above] {7}--
(8,3) node[midway, above] {8}--
(9,3) node[midway, above] {9}--
(10,3) node[midway, above] {10}--
(11,3) node[midway, above] {11}--
(12,3) node[midway, above] {12}--
(13,2) node[midway, above] {13}--
(14,1) node[midway, above] {14}--
(15,1) node[midway, above] {16}--
(16,0) node[midway, above] {17}--
(17,1) node[midway, above] {19}--
(18,2) node[midway, above] {20}--
(19,1) node[midway, above] {21}--
(20,0) node[midway, above] {22};
\draw (4,-1)  --
(5,0) node[midway, above] {4}--
(6,1) node[midway, above] {6}--
(7,2) node[midway, above] {8}--
(8,2) node[midway, below] {9}--
(9,2) node[midway, below] {10}--
(10,1) node[midway, above] {11}--
(11,0) node[midway, above] {12}--
(12,-1) node[midway, above] {16};
\draw (7,-2)  --
(8,-1) node[midway, above] {9}--
(9,-2) node[midway, above] {10};
\draw (0,-2.1) node {};
\end{tikzpicture}\vspace{-0.5cm}\\
\begin{tikzpicture}[scale=0.35]
\draw (0,0)  --
(1,1) node[midway, above] {1}--
(2,2) node[midway, above] {2}--
(3,3) node[midway, above] {3}--
(4,3) node[midway, above] {4}--
(5,4) node[midway, above] {5}--
(6,4) node[midway, above] {6}--
(7,3) node[midway, above] {7}--
(8,3) node[midway, above] {8}--
(9,3) node[midway, above] {9}--
(10,3) node[midway, above] {10}--
(11,3) node[midway, above] {11}--
(12,3) node[midway, above] {12}--
(13,2) node[midway, above] {13}--
(14,2) node[midway, above] {14}--
(15,1) node[midway, above] {15}--
(16,1) node[midway, above] {16}--
(17,1) node[midway, above] {17}--
(18,0) node[midway, above] {18}--
(19,1) node[midway, above] {19}--
(20,2) node[midway, above] {20}--
(21,1) node[midway, above] {21}--
(22,0) node[midway, above] {22};
\draw (4,-1)  --
(5,0) node[midway, above] {4}--
(6,1) node[midway, above] {6}--
(7,2) node[midway, above] {8}--
(8,2) node[midway, above] {9}--
(9,2) node[midway, below] {10}--
(10,1) node[midway, below] {11}--
(11,1) node[midway, above] {12}--
(12,0) node[midway, above] {14}--
(13,0) node[midway, above] {16}--
(14,-1) node[midway, above] {17};
\draw (7,-2)  --
(8,-1) node[midway, below] {9}--
(9,0) node[midway, below] {10}--
(10,-1) node[midway, below] {12}--
(11,-2) node[midway, below] {16};
\draw (0,-3.1) node {};
\end{tikzpicture}

This description considers only the main cases of our Algorithm (in the algorithmic description these are those commented by: \emph{$a_l$, $b_l$, $i$ even $a_{j+1}$ 2, $i$ odd $b_{j+1}$, $i$ odd $a_{j+1}$}). The other cases are explained in the examples below.
\end{example}

\begin{example}[A more complicated example to understand most common special cases]
We consider the following tableau:

\hspace{5cm}
\begin{tikzpicture}[scale=0.35]
  \draw (0,7) -- (8,7);
  \draw (0,6) -- (8,6);
  \draw (0,5) -- (6,5);
  \draw (0,4) -- (4,4);
  \draw (0,3) -- (2,3);
  \draw (0,2) -- (2,2);
  \draw (0,7) -- (0,2);
  \draw (1,7) -- (1,2);
  \draw (2,7) -- (2,2);
  \draw (3,7) -- (3,4);
  \draw (4,7) -- (4,4);
  \draw (5,7) -- (5,5);
  \draw (6,7) -- (6,5);
  \draw (7,7) -- (7,6);
  \draw (8,7) -- (8,6);
  \draw (0.5,6.5) node {1};
  \draw (0.5,5.5) node {2};
  \draw (0.5,4.5) node {3};
  \draw (0.5,3.5) node {4};
  \draw (0.5,2.5) node {5};
  \draw (1.5,6.5) node {6};
  \draw (1.5,5.5) node {7};
  \draw (2.5,6.5) node {8};
  \draw (3.5,6.5) node {9};
  \draw (2.5,5.5) node {10};
  \draw (4.5,6.5) node {11};
  \draw (1.5,4.5) node {12};
  \draw (5.5,6.5) node {13};
  \draw (2.5,4.5) node {14};
  \draw (3.5,5.5) node {15};
  \draw (6.5,6.5) node {16};
  \draw (4.5,5.5) node {17};
  \draw (7.5,6.5) node {18};
  \draw (1.5,3.5) node {19};
  \draw (5.5,5.5) node {20};
  \draw (3.5,4.5) node {21};
  \draw (1.5,2.5) node {22};  
\end{tikzpicture}

When we insert the first row, we see that our Algorithm is not descent-preserving and gives no sensible output if we follow our rules form the previous example strictly. Therefore we create another case for inserting something the first time into a path and change pairs of up-down-steps between them into horizontal steps. This refers to \emph{ $i$ even $a_{j+1}$ 1} (at $2,6$ and $17,18$) and \emph{$i$ even connect} (at $11,12$):

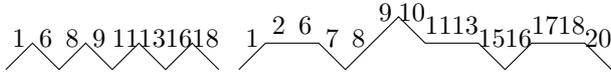
\begin{tikzpicture}[scale=0.35]
\draw (0,0)  --
(1,1) node[midway, above] {1}--
(2,0) node[midway, above] {6}--
(3,1) node[midway, above] {8}--
(4,0) node[midway, above] {9}--
(5,1) node[midway, above] {11}--
(6,0) node[midway, above] {13}--
(7,1) node[midway, above] {16}--
(8,0) node[midway, above] {18};
\end{tikzpicture}
\begin{tikzpicture}[scale=0.35]
\draw (0,0)  --
(1,1) node[midway, above] {1}--
(2,1) node[midway, above] {2}--
(3,1) node[midway, above] {6}--
(4,0) node[midway, above] {7}--
(5,1) node[midway, above] {8}--
(6,2) node[midway, above] {9}--
(7,1) node[midway, above] {10}--
(8,1) node[midway, above] {11}--
(9,1) node[midway, above] {13}--
(10,0) node[midway, above] {15}--
(11,1) node[midway, above] {16}--
(12,1) node[midway, above] {17}--
(13,1) node[midway, above] {18}--
(14,0) node[midway, above] {20};
\end{tikzpicture}

When inserting the third row we have to adjust this rules once again in order to preserve descents and that concatenated tableaux are mapped to concatenated paths. (A property that is only proven for $n=3$, but conjectured otherwise, see Conjecture~\ref{con:ConCat}.) Therefore we have to introduce \emph{$i$ odd connect} and \emph{$i$ odd separate}. Between $a$ and $b$ two horizontal steps on level $1$ are changed into a down-step and an up-step and a down-step and an up-step on level $0$ are changed into two horizontal steps. (In our example we do \emph{$i$ odd separate} at $2,6$, when inserting $3$ and $12$ and at $11,13$ and $17,18$ when inserting $14$ and $21$. We do \emph{$i$ odd connect} at $7,8$ when inserting $3$ and $12$ and at $15,16$ when inserting $14$ and $21$.) The corresponding positions are cycled.

Until this point our algorithm works exactly as in~\cite{3erAlgo}.

Now we initialize the second path. We see that where we did \emph{separate} on our first path, there are some steps that do not observe the rules for vacillating tableau. We will deal with those in the next insertions.

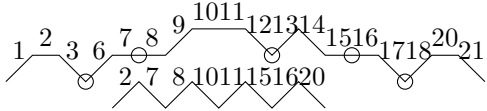
\begin{tikzpicture}[scale=0.35]
\draw (0,0)  --
(1,1) node[midway, above] {1}--
(2,1) node[midway, above] {2}--
(3,0) node[midway, above] {3} node[draw=black,circle,inner sep=0cm,minimum size=0.2cm] {}--
(4,1) node[midway, above] {6}--
(5,1) node[midway, above] {7} node[draw=black,circle,inner sep=0cm,minimum size=0.2cm] {}--
(6,1) node[midway, above] {8}--
(7,2) node[midway, above] {9}--
(8,2) node[midway, above] {10}--
(9,2) node[midway, above] {11}--
(10,1) node[midway, above] {12} node[draw=black,circle,inner sep=0cm,minimum size=0.2cm] {}--
(11,2) node[midway, above] {13}--
(12,1) node[midway, above] {14}--
(13,1) node[midway, above] {15} node[draw=black,circle,inner sep=0cm,minimum size=0.2cm] {}--
(14,1) node[midway, above] {16}--
(15,0) node[midway, above] {17} node[draw=black,circle,inner sep=0cm,minimum size=0.2cm] {}--
(16,1) node[midway, above] {18}--
(17,1) node[midway, above] {20}--
(18,0) node[midway, above] {21};
\draw (4,-1)--
(5,0) node[midway, above] {2}--
(6,-1) node[midway, above] {7}--
(7,0) node[midway, above] {8}--
(8,-1) node[midway, above] {10}--
(9,0) node[midway, above] {11}--
(10,-1) node[midway, above] {15}--
(11,0) node[midway, above] {16}--
(12,-1) node[midway, above] {20};
\end{tikzpicture}

Now several things happen at once:

As mentioned above, we have to deal with this rule violations that we have noticed before.
However we will see, that if $a$ is inserted completely left of such a violation and $b$ completely to the right, the rules we have introduced so far deal with that and two separate paths will be formed. We just mark them as \enquote{allowed hight violations} in \emph{mark it + connect}. We do so between $3$ and $6$.

However between $17$ and $18$ we have to intervene and use \emph{adjust separation point} (we see this at the right dashed circles). When doing so we ignore $b_1$ that is $19$ and act as if $17$ and $18$ are still on level zero. 

Moreover when inserting according to the simple rules we come to another point where the paths do not observe the rules of a vacillating tableau. Again we deal with this and use \emph{height violation}, as this is not marked (we see this at the left dashed circles).

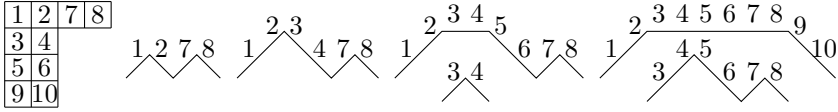
\begin{tikzpicture}[scale=0.35]
\draw (0,0)  --
(1,1) node[midway, above] {1}--
(2,1) node[midway, above] {2}--
(3,1) node[midway, above] {3}--
(4,0) node[midway, above] {4}--
(5,1) node[midway, above] {6}--
(6,1) node[midway, above] {7}--
(7,1) node[midway, above] {8}--
(8,2) node[midway, above] {9}--
(9,2) node[midway, above] {10}--
(10,2) node[midway, above] {11}--
(11,2) node[midway, above] {12}--
(12,2) node[midway, above] {13}--
(13,2) node[midway, above] {14}--
(14,2) node[midway, above] {15}--
(15,2) node[midway, above] {16}--
(16,2) node[midway, above] {17}--
(17,2) node[midway, above] {18}--
(18,1) node[midway, above] {19}--
(19,1) node[midway, above] {20}--
(20,0) node[midway, above] {21};
\draw (5,-1)--
(6,0) node[midway, above] {2}--
(7,0) node[midway, below] {3}--
(8,0) node[midway, above] {7}--
(9,0) node[midway, above] {8}--
(10,0) node[midway, above] {10}--
(11,1) node[midway, above] {11}--
(12,0) node[midway,above] {12}--
(13,1) node[midway,above] {13}--
(14,0) node[midway,above] {14}--
(15,-1) node[midway, above] {15}--
(16,0) node[midway, above] {16}--
(17,0) node[midway, above] {17}--
(18,0) node[midway, above] {18}--
(19,-1) node[midway, above] {20};
  \draw[dashed] (12,0.7) circle (30pt) node{};
  \draw[dashed] (11,2.7) circle (30pt) node{};
  \draw[dashed] (17,-0.3) circle (30pt) node{};
  \draw[dashed] (16,2.7) circle (30pt) node{};
\end{tikzpicture}

In the last insertion no new rule is introduced. However we can see how the two separate, concatenated paths have developed according to the two separate, concatenated tableaux our tableau consists of.

\begin{tikzpicture}[scale=0.35]
\draw (0,0)  --
(1,1) node[midway, above] {1}--
(2,1) node[midway, above] {2}--
(3,1) node[midway, above] {3}--
(4,1) node[midway, above] {4}--
(5,0) node[midway, above] {5}--
(6,1) node[midway, above] {6}--
(7,1) node[midway, above] {7}--
(8,1) node[midway, above] {8}--
(9,2) node[midway, above] {9}--
(10,2) node[midway, above] {10}--
(11,2) node[midway, above] {11}--
(12,2) node[midway, above] {12}--
(13,2) node[midway, above] {13}--
(14,2) node[midway, above] {14}--
(15,2) node[midway, above] {15}--
(16,2) node[midway, above] {16}--
(17,2) node[midway, above] {17}--
(18,2) node[midway, above] {18}--
(19,1) node[midway, above] {19}--
(20,1) node[midway, above] {20}--
(21,1) node[midway, above] {21}--
(22,0) node[midway, above] {22};
\draw (1,-1)--
(2,0) node[midway, above] {2}--
(3,0) node[midway, below] {3}--
(4,-1) node[midway, above] {4};
\draw (6,-1)--
(7,0) node[midway, above] {7}--
(8,-1) node[midway, above] {8}--
(9,0) node[midway, above] {10}--
(10,1) node[midway, above] {11}--
(11,0) node[midway,above] {12}--
(12,1) node[midway,above] {13}--
(13,0) node[midway,above] {14}--
(14,0) node[midway, above] {15}--
(15,0) node[midway, above] {16}--
(16,-1) node[midway, above] {17}--
(17,0) node[midway, above] {18}--
(18,0) node[midway, above] {20}--
(19,-1) node[midway, above] {21};
\end{tikzpicture}

Note that this example is also an example for Algorithm~\ref{alg:u2}, if one reads it from bottom to top.
\end{example}

\begin{example}[Illustrating additional special cases of Algorithm~\ref{alg:2} and~\ref{alg:u2}]
\label{ex:SpecialCases}
We consider several different tableaux to illustrate \emph{ignore inserted $a_l$} in \emph{height violation} (respectively \emph{ignore next $l$}) and the two special cases of \emph{height violation} \emph{$i$ even, $a_{j+1}\neq 0$} and \emph{$i$ odd, $w(p_j)=0$ on $j$-level $0$}.

We start with the following tableau:\\
\begin{tikzpicture}[scale=0.35]
  \draw (0,7) -- (4,7);
  \draw (0,6) -- (4,6);
  \draw (0,5) -- (4,5);
  \draw (0,4) -- (2,4);
  \draw (0,3) -- (2,3);
  \draw (0,7) -- (0,3);
  \draw (1,7) -- (1,3);
  \draw (2,7) -- (2,3);
  \draw (3,7) -- (3,5);
  \draw (4,7) -- (4,5);
  \draw (0.5,6.5) node {1};
  \draw (1.5,6.5) node {2};
  \draw (0.5,5.5) node {3};
  \draw (1.5,5.5) node {4};
  \draw (0.5,4.5) node {5};
  \draw (1.5,4.5) node {6};
  \draw (2.5,6.5) node {7};
  \draw (3.5,6.5) node {8};
  \draw (2.5,5.5) node {9};
  \draw (3.5,5.5) node {10};
  \draw (0.5,3.5) node {11};
  \draw (1.5,3.5) node {12}; 
\end{tikzpicture}
\begin{tikzpicture}[scale=0.35]
\draw (0,0)  --
(1,1) node[midway, above] {1}--
(2,0) node[midway, above] {2}--
(3,1) node[midway, above] {7}--
(4,0) node[midway, above] {8};
\draw(1,-1);
\end{tikzpicture}
\begin{tikzpicture}[scale=0.35]
\draw (0,0)  --
(1,1) node[midway, above] {1}--
(2,2) node[midway, above] {2}--
(3,1) node[midway, above] {3}--
(4,0) node[midway, above] {4}--
(5,1) node[midway, above] {7}--
(6,2) node[midway, above] {8}--
(7,1) node[midway, above] {9}--
(8,0) node[midway, above] {10};
\draw(1,-1);
\end{tikzpicture}
\begin{tikzpicture}[scale=0.35]
\draw (0,0)  --
(1,1) node[midway, above] {1}--
(2,2) node[midway, above] {2}--
(3,2) node[midway, above] {3}--
(4,2) node[midway, above] {4}--
(5,1) node[midway, above] {5}--
(6,0) node[midway, above] {6}--
(7,1) node[midway, above] {7}--
(8,2) node[midway, above] {8}--
(9,1) node[midway, above] {9}--
(10,0) node[midway, above] {10};
\draw (2,-1)  --
(3,0) node[midway, above] {3}--
(4,-1) node[midway, above] {4};
\draw(1,-1);
\end{tikzpicture}
\begin{tikzpicture}[scale=0.35]
\draw (0,0)  --
(1,1) node[midway, above] {1}--
(2,2) node[midway, above] {2}--
(3,2) node[midway, above] {3}--
(4,2) node[midway, above] {4}--
(5,2) node[midway, above] {5}--
(6,2) node[midway, above] {6}--
(7,2) node[midway, above] {7}--
(8,2) node[midway, above] {8}--
(9,2) node[midway, above] {9}--
(10,2) node[midway, above] {10}--
(11,1) node[midway, above] {11}--
(12,0) node[midway, above] {12};
\draw (2,-1)  --
(3,0) node[midway, above] {3}--
(4,1) node[midway, above] {4}--
(5,0) node[midway, above] {5}--
(6,-1) node[midway, above] {6}--
(7,0) node[midway, above] {7}--
(8,1) node[midway, above] {8}--
(9,0) node[midway, above] {9}--
(10,-1) node[midway, above] {10};
\end{tikzpicture}

Inserting the first three rows works as before. However when inserting the fourth row we see in which cases we need \emph{ignore inserted $a_l$ if $a_{l+1}=0$} in \emph{height violation}. When we insert $11,12$, we have a \emph{height violation} at $p=8$. At $p=7$ we have again a \emph{height violation} as we ignore the inserted $11$.

\begin{tikzpicture}[scale=0.35]
\draw (7,1)  --
(8,2) node[midway, above] {8}--
(9,2) node[midway, above] {9}--
(10,2) node[midway, above] {10}--
(11,1) node[midway, above] {11}--
(12,0) node[midway, above] {12};
\draw (8,1)  --
(9,0) node[midway, above] {9}--
(10,-1) node[midway, above] {10};
\end{tikzpicture}
\hspace{0.2cm}
\raisebox{0.7cm}{$\rightarrow$}
\hspace{0.2cm}
\begin{tikzpicture}[scale=0.35]
\draw (6,1)  --
(7,2) node[midway, above] {7}--
(8,2) node[midway, above] {8}--
(9,2) node[midway, above] {9}--
(10,2) node[midway, above] {10}--
(11,1) node[midway, above] {11}--
(12,0) node[midway, above] {12};
\draw (7,0)  --
(8,1) node[midway, above] {8}--
(9,0) node[midway, above] {9}--
(10,-1) node[midway, above] {10};
\end{tikzpicture}
\hspace{0.2cm}
\raisebox{0.7cm}{$\rightarrow$}
\hspace{0.2cm}
\begin{tikzpicture}[scale=0.35]
\draw (6,2)  --
(7,2) node[midway, above] {7}--
(8,2) node[midway, above] {8}--
(9,2) node[midway, above] {9}--
(10,2) node[midway, above] {10}--
(11,1) node[midway, above] {11}--
(12,0) node[midway, above] {12};
\draw (6,-1)  --
(7,0) node[midway, above] {7}--
(8,1) node[midway, above] {8}--
(9,0) node[midway, above] {9}--
(10,-1) node[midway, above] {10};
\end{tikzpicture}

Now we consider the resulting vacillating tableau apply Algorithm~\ref{alg:u2}. We get the labeled words in the opposite direction and obtain elements of our standard Young tableau two by two. We have a closer look at the first extraction as here again happens a special case. We extract $5,6$ as $a_2,b_2$ and get a \emph{height violation} at $7$. We correct it and continue. At $8$ we have again a \emph{height violation}, as we ignore $7$. Again we correct it and continue.

\begin{tikzpicture}[scale=0.35]
\draw (0,0)  --
(1,1) node[midway, above] {1}--
(2,2) node[midway, above] {2}--
(3,2) node[midway, above] {3}--
(4,2) node[midway, above] {4}--
(5,1) node[midway, above] {5}--
(6,0) node[midway, above] {6}--
(7,0) node[midway, above] {7};
\draw (2,-1)  --
(3,0) node[midway, above] {3}--
(4,-1) node[midway, above] {4}--
(5,0) node[midway, above] {7};
\end{tikzpicture}
\hspace{0.2cm}
\raisebox{0.7cm}{$\rightarrow$}
\hspace{0.2cm}
\begin{tikzpicture}[scale=0.35]
\draw (0,0)  --
(1,1) node[midway, above] {1}--
(2,2) node[midway, above] {2}--
(3,2) node[midway, above] {3}--
(4,2) node[midway, above] {4}--
(5,1) node[midway, above] {5}--
(6,0) node[midway, above] {6}--
(7,1) node[midway, above] {7}--
(8,1) node[midway, above] {8};
\draw (2,-1)  --
(3,0) node[midway, above] {3}--
(4,-1) node[midway, above] {4}--
(5,0) node[midway, above] {8};
\end{tikzpicture}
\hspace{0.2cm}
\raisebox{0.7cm}{$\rightarrow$}
\hspace{0.2cm}
\begin{tikzpicture}[scale=0.35]
\draw (0,0)  --
(1,1) node[midway, above] {1}--
(2,2) node[midway, above] {2}--
(3,2) node[midway, above] {3}--
(4,2) node[midway, above] {4}--
(5,1) node[midway, above] {5}--
(6,0) node[midway, above] {6}--
(7,1) node[midway, above] {7}--
(8,2) node[midway, above] {8};
\draw (2,-1)  --
(3,0) node[midway, above] {3}--
(4,-1) node[midway, above] {4};
\end{tikzpicture}

We point out, that this is also an example where a \emph{height violation} of $a$ overlaps with one of $b$. If this is not the case, then those special cases can also occur just in one of the two algorithms.

In the following Example, during insert row 4, we have to ignore 9 at $p=7$ and get a \emph{height violation} there.

\noindent \begin{tikzpicture}[scale=0.35]
  \draw (0,7) -- (4,7);
  \draw (0,6) -- (4,6);
  \draw (0,5) -- (2,5);
  \draw (0,4) -- (2,4);
  \draw (0,3) -- (2,3);
  \draw (0,7) -- (0,3);
  \draw (1,7) -- (1,3);
  \draw (2,7) -- (2,3);
  \draw (3,7) -- (3,6);
  \draw (4,7) -- (4,6);
  \draw (0.5,6.5) node {1};
  \draw (1.5,6.5) node {2};
  \draw (0.5,5.5) node {3};
  \draw (1.5,5.5) node {4};
  \draw (0.5,4.5) node {5};
  \draw (1.5,4.5) node {6};
  \draw (2.5,6.5) node {7};
  \draw (3.5,6.5) node {8};
  \draw (0.5,3.5) node {9};
  \draw (1.5,3.5) node {10};
\end{tikzpicture}
\begin{tikzpicture}[scale=0.31]
\draw (0,0)  --
(1,1) node[midway, above] {1}--
(2,0) node[midway, above] {2}--
(3,1) node[midway, above] {7}--
(4,0) node[midway, above] {8};
\draw(1,-1.4);
\end{tikzpicture}
\begin{tikzpicture}[scale=0.31]
\draw (0,0)  --
(1,1) node[midway, above] {1}--
(2,2) node[midway, above] {2}--
(3,1) node[midway, above] {3}--
(4,0) node[midway, above] {4}--
(5,1) node[midway, above] {7}--
(6,0) node[midway, above] {8};
\draw(1,-1.4);
\end{tikzpicture}
\begin{tikzpicture}[scale=0.31]
\draw (0,0)  --
(1,1) node[midway, above] {1}--
(2,2) node[midway, above] {2}--
(3,2) node[midway, above] {3}--
(4,2) node[midway, above] {4}--
(5,1) node[midway, above] {5}--
(6,0) node[midway, above] {6}--
(7,1) node[midway, above] {7}--
(8,0) node[midway, above] {8};
\draw (2,-1)  --
(3,0) node[midway, above] {3}--
(4,-1) node[midway, above] {4};
\draw(1,-1.4);
\end{tikzpicture}
\begin{tikzpicture}[scale=0.31]
\draw (0,0)  --
(1,1) node[midway, above] {1}--
(2,2) node[midway, above] {2}--
(3,2) node[midway, above] {3}--
(4,2) node[midway, above] {4}--
(5,2) node[midway, above] {5}--
(6,2) node[midway, above] {6}--
(7,2) node[midway, above] {7}--
(8,2) node[midway, above] {8}--
(9,1) node[midway, above] {9}--
(10,0) node[midway, above] {10};
\draw (2,-1)  --
(3,0) node[midway, above] {3}--
(4,1) node[midway, above] {4}--
(5,0) node[midway, above] {5}--
(6,-1) node[midway, above] {6}--
(7,0) node[midway, above] {7}--
(8,-1) node[midway, above] {8};
\draw(1,-1.4);
\end{tikzpicture}

In the following Example, during extract row $5$, we have to ignore $2$ at $p=8$ and get a hight violation there.\\
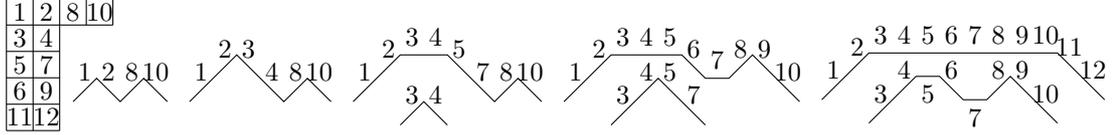
\begin{tikzpicture}[scale=0.35]
  \draw (0,7) -- (4,7);
  \draw (0,6) -- (4,6);
  \draw (0,5) -- (2,5);
  \draw (0,4) -- (2,4);
  \draw (0,3) -- (2,3);
  \draw (0,2) -- (2,2);
  \draw (0,7) -- (0,2);
  \draw (1,7) -- (1,2);
  \draw (2,7) -- (2,2);
  \draw (3,7) -- (3,6);
  \draw (4,7) -- (4,6);
  \draw (0.5,6.5) node {1};
  \draw (1.5,6.5) node {2};
  \draw (0.5,5.5) node {3};
  \draw (1.5,5.5) node {4};
  \draw (0.5,4.5) node {5};
  \draw (0.5,3.5) node {6};
  \draw (1.5,4.5) node {7};
  \draw (2.5,6.5) node {8};
  \draw (1.5,3.5) node {9};
  \draw (3.5,6.5) node {10};
  \draw (0.5,2.5) node {11};
  \draw (1.5,2.5) node {12};
\end{tikzpicture}
\hspace{-0.9cm}
\begin{tikzpicture}[scale=0.31]
\draw (0,0)  --
(1,1) node[midway, above] {1}--
(2,0) node[midway, above] {2}--
(3,1) node[midway, above] {8}--
(4,0) node[midway, above] {10};
\draw(1,-1.4);
\end{tikzpicture}
\begin{tikzpicture}[scale=0.31]
\draw (0,0)  --
(1,1) node[midway, above] {1}--
(2,2) node[midway, above] {2}--
(3,1) node[midway, above] {3}--
(4,0) node[midway, above] {4}--
(5,1) node[midway, above] {8}--
(6,0) node[midway, above] {10};
\draw(1,-1.4);
\end{tikzpicture}
\begin{tikzpicture}[scale=0.31]
\draw (0,0)  --
(1,1) node[midway, above] {1}--
(2,2) node[midway, above] {2}--
(3,2) node[midway, above] {3}--
(4,2) node[midway, above] {4}--
(5,1) node[midway, above] {5}--
(6,0) node[midway, above] {7}--
(7,1) node[midway, above] {8}--
(8,0) node[midway, above] {10};
\draw (2,-1)  --
(3,0) node[midway, above] {3}--
(4,-1) node[midway, above] {4};
\draw(1,-1.4);
\end{tikzpicture}
\begin{tikzpicture}[scale=0.31]
\draw (0,0)  --
(1,1) node[midway, above] {1}--
(2,2) node[midway, above] {2}--
(3,2) node[midway, above] {3}--
(4,2) node[midway, above] {4}--
(5,2) node[midway, above] {5}--
(6,1) node[midway, above] {6}--
(7,1) node[midway, above] {7}--
(8,2) node[midway, above] {8}--
(9,1) node[midway, above] {9}--
(10,0) node[midway, above] {10};
\draw (2,-1)  --
(3,0) node[midway, above] {3}--
(4,1) node[midway, above] {4}--
(5,0) node[midway, above] {5}--
(6,-1) node[midway, above] {7};
\draw(1,-1.4);
\end{tikzpicture}
\begin{tikzpicture}[scale=0.31]
\draw (0,0)  --
(1,1) node[midway, above] {1}--
(2,2) node[midway, above] {2}--
(3,2) node[midway, above] {3}--
(4,2) node[midway, above] {4}--
(5,2) node[midway, above] {5}--
(6,2) node[midway, above] {6}--
(7,2) node[midway, above] {7}--
(8,2) node[midway, above] {8}--
(9,2) node[midway, above] {9}--
(10,2) node[midway, above] {10}--
(11,1) node[midway, above] {11}--
(12,0) node[midway, above] {12};
\draw (2,-1)  --
(3,0) node[midway, above] {3}--
(4,1) node[midway, above] {4}--
(5,1) node[midway, below] {5}--
(6,0) node[midway, above] {6}--
(7,0) node[midway, below] {7}--
(8,1) node[midway, above] {8}--
(9,0) node[midway, above] {9}--
(10,-1) node[midway, above] {10};
\draw(1,-1.4);
\end{tikzpicture}

Next we consider the following tableau:\\
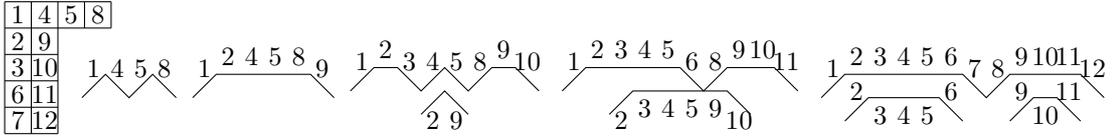
\begin{tikzpicture}[scale=0.35]
  \draw (0,7) -- (4,7);
  \draw (0,6) -- (4,6);
  \draw (0,5) -- (2,5);
  \draw (0,4) -- (2,4);
  \draw (0,3) -- (2,3);
  \draw (0,2) -- (2,2);
  \draw (0,7) -- (0,2);
  \draw (1,7) -- (1,2);
  \draw (2,7) -- (2,2);
  \draw (3,7) -- (3,6);
  \draw (4,7) -- (4,6);
  \draw (0.5,6.5) node {1};
  \draw (0.5,5.5) node {2};
  \draw (0.5,4.5) node {3};
  \draw (1.5,6.5) node {4};
  \draw (2.5,6.5) node {5};
  \draw (0.5,3.5) node {6};
  \draw (0.5,2.5) node {7};
  \draw (3.5,6.5) node {8};
  \draw (1.5,5.5) node {9};
  \draw (1.5,4.5) node {10};
  \draw (1.5,3.5) node {11};
  \draw (1.5,2.5) node {12};
\end{tikzpicture}
\hspace{-0.7cm}
\begin{tikzpicture}[scale=0.31]
\draw (0,0)  --
(1,1) node[midway, above] {1}--
(2,0) node[midway, above] {4}--
(3,1) node[midway, above] {5}--
(4,0) node[midway, above] {8};
\draw(1,-1.7);
\end{tikzpicture}
\begin{tikzpicture}[scale=0.31]
\draw (0,0)  --
(1,1) node[midway, above] {1}--
(2,1) node[midway, above] {2}--
(3,1) node[midway, above] {4}--
(4,1) node[midway, above] {5}--
(5,1) node[midway, above] {8}--
(6,0) node[midway, above] {9};
\draw(1,-1.7);
\end{tikzpicture}
\begin{tikzpicture}[scale=0.31]
\draw (0,0)  --
(1,1) node[midway, above] {1}--
(2,1) node[midway, above] {2}--
(3,0) node[midway, above] {3}--
(4,1) node[midway, above] {4}--
(5,0) node[midway, above] {5}--
(6,1) node[midway, above] {8}--
(7,1) node[midway, above] {9}--
(8,0) node[midway, above] {10};
\draw (3,-1)  --
(4,0) node[midway, below] {2}--
(5,-1) node[midway, below] {9};
\draw(1,-1.7);
\end{tikzpicture}
\begin{tikzpicture}[scale=0.31]
\draw (0,0)  --
(1,1) node[midway, above] {1}--
(2,1) node[midway, above] {2}--
(3,1) node[midway, above] {3}--
(4,1) node[midway, above] {4}--
(5,1) node[midway, above] {5}--
(6,0) node[midway, above] {6}--
(7,1) node[midway, above] {8}--
(8,1) node[midway, above] {9}--
(9,1) node[midway, above] {10}--
(10,0) node[midway, above] {11};
\draw (2,-1)  --
(3,0) node[midway, below] {2}--
(4,0) node[midway, below] {3}--
(5,0) node[midway, below] {4}--
(6,0) node[midway, below] {5}--
(7,0) node[midway, below] {9}--
(8,-1) node[midway, below] {10};
\draw(1,-1.7);
\end{tikzpicture}
\begin{tikzpicture}[scale=0.31]
\draw (0,0)  --
(1,1) node[midway, above] {1}--
(2,1) node[midway, above] {2}--
(3,1) node[midway, above] {3}--
(4,1) node[midway, above] {4}--
(5,1) node[midway, above] {5}--
(6,1) node[midway, above] {6}--
(7,0) node[midway, above] {7}--
(8,1) node[midway, above] {8}--
(9,1) node[midway, above] {9}--
(10,1) node[midway, above] {10}--
(11,1) node[midway, above] {11}--
(12,0) node[midway, above] {12};
\draw (1,-1)  --
(2,0) node[midway, above] {2}--
(3,0) node[midway, below] {3}--
(4,0) node[midway, below] {4}--
(5,0) node[midway, below] {5}--
(6,-1) node[midway, above] {6};
\draw (8,-1)--
(9,0) node[midway, above] {9}--
(10,0) node[midway, below] {10}--
(11,-1) node[midway, above] {11};
\draw(1,-1.7);
\end{tikzpicture}

Again inserting the first three rows works as before. When inserting the fourth row, thus $6,11$, we come into the special case of \emph{height violation} \emph{$i$ even, $a_{j+1}\neq 0$}. The way we deal with this ensures, that we can continue normally after adjusting the \emph{height violation}. At $p=4$ we have a \emph{height violation} that is marked but $p<a_2$. At this point $a_{3}$ is already defined to be $2$. We set $a_2$ and $a_3$ back to $0$ and search for them anew. Later we define them to be $3$ and $2$.

\noindent 
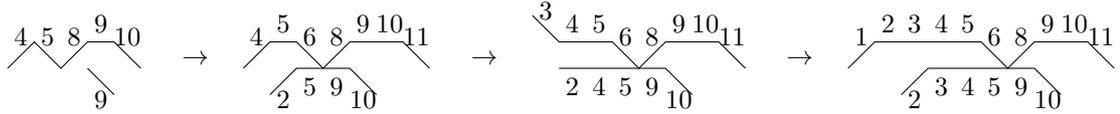
\begin{tikzpicture}[scale=0.35]
\draw (3,0)--
(4,1) node[midway, above] {4}--
(5,0) node[midway, above] {5}--
(6,1) node[midway, above] {8}--
(7,1) node[midway, above] {9}--
(8,0) node[midway, above] {10};
\draw (6,0)--
(7,-1) node[midway, below] {9};
\draw(3,-1.7);
\end{tikzpicture}
\hspace{0.2cm}
\raisebox{0.7cm}{$\rightarrow$}
\hspace{0.2cm}
\begin{tikzpicture}[scale=0.35]
\draw (3,0)  --
(4,1) node[midway, above] {4}--
(5,1) node[midway, above] {5}--
(6,0) node[midway, above] {6}--
(7,1) node[midway, above] {8}--
(8,1) node[midway, above] {9}--
(9,1) node[midway, above] {10}--
(10,0) node[midway, above] {11};
\draw (4,-1)  --
(5,0) node[midway, below] {2}--
(6,0) node[midway, below] {5}--
(7,0) node[midway, below] {9}--
(8,-1) node[midway, below] {10};
\end{tikzpicture}
\hspace{0.2cm}
\raisebox{0.7cm}{$\rightarrow$}
\hspace{0.2cm}
\begin{tikzpicture}[scale=0.35]
\draw (2,2)  --
(3,1) node[midway, above] {3}--
(4,1) node[midway, above] {4}--
(5,1) node[midway, above] {5}--
(6,0) node[midway, above] {6}--
(7,1) node[midway, above] {8}--
(8,1) node[midway, above] {9}--
(9,1) node[midway, above] {10}--
(10,0) node[midway, above] {11};
\draw (3,0)  --
(4,0) node[midway, below] {2}--
(5,0) node[midway, below] {4}--
(6,0) node[midway, below] {5}--
(7,0) node[midway, below] {9}--
(8,-1) node[midway, below] {10};
\end{tikzpicture}
\hspace{0.2cm}
\raisebox{0.7cm}{$\rightarrow$}
\hspace{0.2cm}
\begin{tikzpicture}[scale=0.35]
\draw (0,0)  --
(1,1) node[midway, above] {1}--
(2,1) node[midway, above] {2}--
(3,1) node[midway, above] {3}--
(4,1) node[midway, above] {4}--
(5,1) node[midway, above] {5}--
(6,0) node[midway, above] {6}--
(7,1) node[midway, above] {8}--
(8,1) node[midway, above] {9}--
(9,1) node[midway, above] {10}--
(10,0) node[midway, above] {11};
\draw (2,-1)  --
(3,0) node[midway, below] {2}--
(4,0) node[midway, below] {3}--
(5,0) node[midway, below] {4}--
(6,0) node[midway, below] {5}--
(7,0) node[midway, below] {9}--
(8,-1) node[midway, below] {10};
\end{tikzpicture}

Again we consider the resulting vacillating tableau and apply Algorithm~\ref{alg:u2}. Again we obtain the same sequence of labeled words but the other way around and extract elements of our standard Young tableau in pairs. We obtain our special case when extracting row $4$, thus an even row. $4$ is a $0$ on $1$-level $0$. We set $a_2$ back to $r$ and change $4$ into a $-1$. Later we extract $5$ as a new $a_2$.

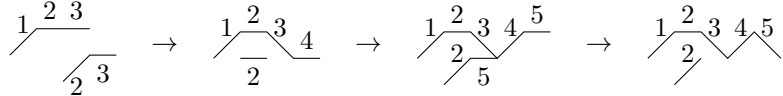
\begin{tikzpicture}[scale=0.35]
\draw (0,0)  --
(1,1) node[midway, above] {1}--
(2,1) node[midway, above] {2}--
(3,1) node[midway, above] {3};
\draw (2,-1)  --
(3,0) node[midway, below] {2}--
(4,0) node[midway, below] {3};
\draw(1,-1.7);
\end{tikzpicture}
\hspace{0.2cm}
\raisebox{0.7cm}{$\rightarrow$}
\hspace{0.2cm}
\begin{tikzpicture}[scale=0.35]
\draw (0,0)  --
(1,1) node[midway, above] {1}--
(2,1) node[midway, above] {2}--
(3,0) node[midway, above] {3}--
(4,0) node[midway, above] {4};
\draw (1,0)  --
(2,0) node[midway, below] {2};
\draw(1,-1.7);
\end{tikzpicture}
\hspace{0.2cm}
\raisebox{0.7cm}{$\rightarrow$}
\hspace{0.2cm}
\begin{tikzpicture}[scale=0.35]
\draw (0,0)  --
(1,1) node[midway, above] {1}--
(2,1) node[midway, above] {2}--
(3,0) node[midway, above] {3}--
(4,1) node[midway, above] {4}--
(5,1) node[midway, above] {5};
\draw (1,-1)  --
(2,0) node[midway, above] {2}--
(3,0) node[midway, below] {5};
\draw(1,-1.7);
\end{tikzpicture}
\hspace{0.2cm}
\raisebox{0.7cm}{$\rightarrow$}
\hspace{0.2cm}
\begin{tikzpicture}[scale=0.35]
\draw (0,0)  --
(1,1) node[midway, above] {1}--
(2,1) node[midway, above] {2}--
(3,0) node[midway, above] {3}--
(4,1) node[midway, above] {4}--
(5,0) node[midway, above] {5};
\draw (1,-1)  --
(2,0) node[midway, above] {2};
\draw(1,-1.7);
\end{tikzpicture}

Finally we consider the following tableau:\\
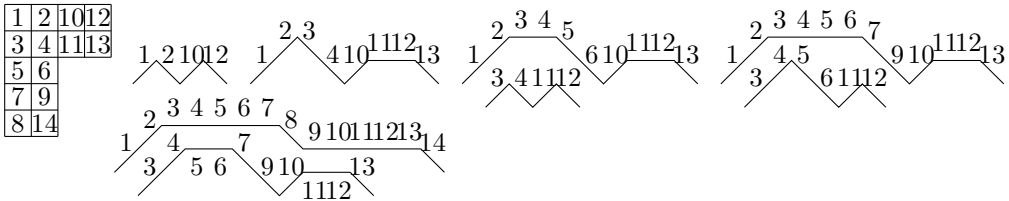
\begin{tikzpicture}[scale=0.35]
  \draw (0,7) -- (4,7);
  \draw (0,6) -- (4,6);
  \draw (0,5) -- (4,5);
  \draw (0,4) -- (2,4);
  \draw (0,3) -- (2,3);
  \draw (0,2) -- (2,2);
  \draw (0,7) -- (0,2);
  \draw (1,7) -- (1,2);
  \draw (2,7) -- (2,2);
  \draw (3,7) -- (3,5);
  \draw (4,7) -- (4,5);
  \draw (0.5,6.5) node {1};
  \draw (1.5,6.5) node {2};
  \draw (0.5,5.5) node {3};
  \draw (1.5,5.5) node {4};
  \draw (0.5,4.5) node {5};
  \draw (1.5,4.5) node {6};
  \draw (0.5,3.5) node {7};
  \draw (0.5,2.5) node {8};
  \draw (1.5,3.5) node {9};
  \draw (2.5,6.5) node {10};
  \draw (2.5,5.5) node {11};
  \draw (3.5,6.5) node {12};
  \draw (3.5,5.5) node {13};
  \draw (1.5,2.5) node {14}; 
\end{tikzpicture}
\begin{tikzpicture}[scale=0.31]
\draw (0,0)  --
(1,1) node[midway, above] {1}--
(2,0) node[midway, above] {2}--
(3,1) node[midway, above] {10}--
(4,0) node[midway, above] {12};
\draw(1,-2.4);
\end{tikzpicture}
\begin{tikzpicture}[scale=0.31]
\draw (0,0)  --
(1,1) node[midway, above] {1}--
(2,2) node[midway, above] {2}--
(3,1) node[midway, above] {3}--
(4,0) node[midway, above] {4}--
(5,1) node[midway, above] {10}--
(6,1) node[midway, above] {11}--
(7,1) node[midway, above] {12}--
(8,0) node[midway, above] {13};
\draw(1,-2.4);
\end{tikzpicture}
\begin{tikzpicture}[scale=0.31]
\draw (0,0)  --
(1,1) node[midway, above] {1}--
(2,2) node[midway, above] {2}--
(3,2) node[midway, above] {3}--
(4,2) node[midway, above] {4}--
(5,1) node[midway, above] {5}--
(6,0) node[midway, above] {6}--
(7,1) node[midway, above] {10}--
(8,1) node[midway, above] {11}--
(9,1) node[midway, above] {12}--
(10,0) node[midway, above] {13};
\draw (1,-1)  --
(2,0) node[midway, above] {3}--
(3,-1) node[midway, above] {4}--
(4,0) node[midway, above] {11}--
(5,-1) node[midway, above] {12};
\draw(1,-2.4);
\end{tikzpicture}
\begin{tikzpicture}[scale=0.31]
\draw (0,0)  --
(1,1) node[midway, above] {1}--
(2,2) node[midway, above] {2}--
(3,2) node[midway, above] {3}--
(4,2) node[midway, above] {4}--
(5,2) node[midway, above] {5}--
(6,2) node[midway, above] {6}--
(7,1) node[midway, above] {7}--
(8,0) node[midway, above] {9}--
(9,1) node[midway, above] {10}--
(10,1) node[midway, above] {11}--
(11,1) node[midway, above] {12}--
(12,0) node[midway, above] {13};
\draw (1,-1)  --
(2,0) node[midway, above] {3}--
(3,1) node[midway, above] {4}--
(4,0) node[midway, above] {5}--
(5,-1) node[midway, above] {6}--
(6,0) node[midway, above] {11}--
(7,-1) node[midway, above] {12};
\draw(1,-2.4);
\end{tikzpicture}\vspace{-0.7cm}

\hspace{1cm}\begin{tikzpicture}[scale=0.31]
\draw (0,0)  --
(1,1) node[midway, above] {1}--
(2,2) node[midway, above] {2}--
(3,2) node[midway, above] {3}--
(4,2) node[midway, above] {4}--
(5,2) node[midway, above] {5}--
(6,2) node[midway, above] {6}--
(7,2) node[midway, above] {7}--
(8,1) node[midway, above] {8}--
(9,1) node[midway, above] {9}--
(10,1) node[midway, above] {10}--
(11,1) node[midway, above] {11}--
(12,1) node[midway, above] {12}--
(13,1) node[midway, above] {13}--
(14,0) node[midway, above] {14};
\draw (1,-1)  --
(2,0) node[midway, above] {3}--
(3,1) node[midway, above] {4}--
(4,1) node[midway, below] {5}--
(5,1) node[midway, below] {6}--
(6,0) node[midway, above] {7}--
(7,-1) node[midway, above] {9}--
(8,0) node[midway, above] {10}--
(9,0) node[midway, below] {11}--
(10,0) node[midway, below] {12}--
(11,-1) node[midway, above] {13};
\draw(1,-1.4);
\end{tikzpicture}

This time inserting the first four rows works as before, however when inserting the fifth row we come the special case of \emph{height violation} \emph{$i$ odd, $w(p_j)=0$ on $j$-level $0$}. Again the way we deal with this ensures, that we can continue normally after \emph{adjusting the height violation}. This happens while inserting $8,14$. At $p=10$ we have a \emph{height violation}, where we had \emph{$i$ odd connect} at $6,11$. We insert $b_2$ again with $9$.

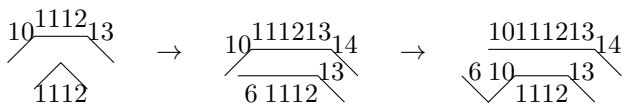
\begin{tikzpicture}[scale=0.35]
\draw (8,0) --
(9,1) node[midway, above] {10}--
(10,1) node[midway, above] {11}--
(11,1) node[midway, above] {12}--
(12,0) node[midway, above] {13};
\draw (9,-1)--
(10,0) node[midway, below] {11}--
(11,-1) node[midway, below] {12};
\end{tikzpicture}
\hspace{0.2cm}
\raisebox{0.7cm}{$\rightarrow$}
\hspace{0.2cm}
\begin{tikzpicture}[scale=0.35]
\draw (9,0)  --
(10,1) node[midway, above] {10}--
(11,1) node[midway, above] {11}--
(12,1) node[midway, above] {12}--
(13,1) node[midway, above] {13}--
(14,0) node[midway, above] {14};
\draw (9.5,0)  --
(10.5,0) node[midway, below] {6}--
(11.5,0) node[midway, below] {11}--
(12.5,0) node[midway, below] {12}--
(13.5,-1) node[midway, above] {13};
\end{tikzpicture}
\hspace{0.2cm}
\raisebox{0.7cm}{$\rightarrow$}
\hspace{0.2cm}
\begin{tikzpicture}[scale=0.35]
\draw (9,1)--
(10,1) node[midway, above] {10}--
(11,1) node[midway, above] {11}--
(12,1) node[midway, above] {12}--
(13,1) node[midway, above] {13}--
(14,0) node[midway, above] {14};
\draw (8,0)  --
(9,-1) node[midway, above] {6}--
(10,0) node[midway, above] {10}--
(11,0) node[midway, below] {11}--
(12,0) node[midway, below] {12}--
(13,-1) node[midway, above] {13};
\end{tikzpicture}

Again we consider the resulting vacillating tableau and apply Algorithm~\ref{alg:u2}. Again we obtain the same sequence of labeled words but the other way around and extract elements to our standard Young tableau in pairs. We extract $7,8$ for $a_2,b_2$ and $8$ for $a_1$ and obtain a \emph{height violation} at $10$. We set $b_2=r$ and correct it. Then we obtain a \emph{height violation special case} at $11$. This happens during extracting row $5$, thus during \emph{$i$ odd}. We set $b_3=r$ and $w(11)=2$. We continue extracting $12$ as $b_3$, $13$ as $b_2$ and $14$ as $b_1=b$. Thus we extract $8,14$ as $a,b$.

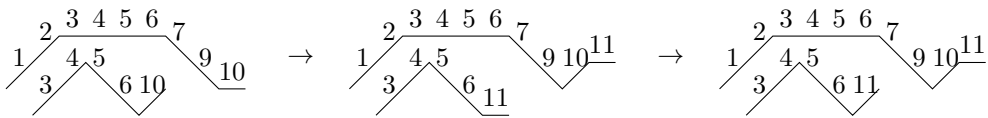
\begin{tikzpicture}[scale=0.35]
\draw (0,0)  --
(1,1) node[midway, above] {1}--
(2,2) node[midway, above] {2}--
(3,2) node[midway, above] {3}--
(4,2) node[midway, above] {4}--
(5,2) node[midway, above] {5}--
(6,2) node[midway, above] {6}--
(7,1) node[midway, above] {7}--
(8,0) node[midway, above] {9}--
(9,0) node[midway, above] {10};
\draw (1,-1)  --
(2,0) node[midway, above] {3}--
(3,1) node[midway, above] {4}--
(4,0) node[midway, above] {5}--
(5,-1) node[midway, above] {6}--
(6,0) node[midway, above] {10};
\end{tikzpicture}
\hspace{0.2cm}
\raisebox{0.7cm}{$\rightarrow$}
\hspace{0.2cm}
\begin{tikzpicture}[scale=0.35]
\draw (0,0)  --
(1,1) node[midway, above] {1}--
(2,2) node[midway, above] {2}--
(3,2) node[midway, above] {3}--
(4,2) node[midway, above] {4}--
(5,2) node[midway, above] {5}--
(6,2) node[midway, above] {6}--
(7,1) node[midway, above] {7}--
(8,0) node[midway, above] {9}--
(9,1) node[midway, above] {10}--
(10,1) node[midway, above] {11};
\draw (1,-1)  --
(2,0) node[midway, above] {3}--
(3,1) node[midway, above] {4}--
(4,0) node[midway, above] {5}--
(5,-1) node[midway, above] {6}--
(6,-1) node[midway, above] {11};
\end{tikzpicture}
\hspace{0.2cm}
\raisebox{0.7cm}{$\rightarrow$}
\hspace{0.2cm}
\begin{tikzpicture}[scale=0.35]
\draw (0,0)  --
(1,1) node[midway, above] {1}--
(2,2) node[midway, above] {2}--
(3,2) node[midway, above] {3}--
(4,2) node[midway, above] {4}--
(5,2) node[midway, above] {5}--
(6,2) node[midway, above] {6}--
(7,1) node[midway, above] {7}--
(8,0) node[midway, above] {9}--
(9,1) node[midway, above] {10}--
(10,1) node[midway, above] {11};
\draw (1,-1)  --
(2,0) node[midway, above] {3}--
(3,1) node[midway, above] {4}--
(4,0) node[midway, above] {5}--
(5,-1) node[midway, above] {6}--
(6,0) node[midway, above] {11};
\end{tikzpicture}
\end{example}

\begin{definition}
\emph{Separation points} are positions that are marked.
\end{definition}

\begin{example}[One tableau, different $n$] In this example we consider a standard Young tableau $Q$ with $7$ rows and $2$ columns in different dimension $n$. The first column of $Q$ is filled with $1,2,\dots,7$, the second column is filled with $8,9,\dots,14.$ As the rows have even length, empty rows are allowed. We see that \emph{separation points} (positions that get marked) make a difference doing so, as those parts of the algorithms are the only ones executed for $a=b=0$. 

We see an illustration of this example in Figure~\ref{fig:Embedding}.

When considering $n=2k+1=9$, we have $k=3$ paths. When we consider $n=9$ or $n=11$ we see how \emph{adjust separation points} alter the paths step by step and creates more paths. Finally we consider $n=2k+1=13$ and $n>13$. We see that when going from $13$ to $14$ we add path $7$, which is an up-step and a down-step. When considering larger $n$, the paths do not change anymore. Path $7$ is dashed.

The reason for this phenomena is that $0$'s in a vacillating tableau are only allowed when the $k$-level is at least $1$. Thus horizontal steps, that are truly horizontal steps, and not some other steps in paths below, are only allowed in the bottommost path.

This are the only differences when considering a tableau in different dimensions $n=2k+1$.
\end{example}

\begin{figure}
\caption{One tableau, different dimensions $n$}
\label{fig:Embedding}
\begin{tikzpicture}[scale=0.35]
  \draw (0,7) -- (2,7);
  \draw (0,6) -- (2,6);
  \draw (0,5) -- (2,5);
  \draw (0,4) -- (2,4);
  \draw (0,3) -- (2,3);
  \draw (0,2) -- (2,2);
  \draw (0,1) -- (2,1);
  \draw (0,0) -- (2,0);
  \draw (0,7) -- (0,0);
  \draw (1,7) -- (1,0);
  \draw (2,7) -- (2,0);
  \draw (0.5,6.5) node {1};
  \draw (0.5,5.5) node {2};
  \draw (0.5,4.5) node {3};
  \draw (0.5,3.5) node {4};
  \draw (0.5,2.5) node {5};
  \draw (0.5,1.5) node {6};
  \draw (0.5,0.5) node {7};
  \draw (1.5,6.5) node {8};
  \draw (1.5,5.5) node {9};
  \draw (1.5,4.5) node {10};
  \draw (1.5,3.5) node {11};
  \draw (1.5,2.5) node {12};
  \draw (1.5,1.5) node {13};
  \draw (1.5,0.5) node {14}; 
\end{tikzpicture}
\hspace{1cm}
\begin{tikzpicture}[scale=0.35]
\draw (0,0)  --
(1,1) node[midway, above] {1}--
(2,1) node[midway, above] {2}--
(3,1) node[midway, above] {3}--
(4,1) node[midway, above] {4}--
(5,1) node[midway, above] {5}--
(6,1) node[midway, above] {6}--
(7,0) node[midway, above] {7}--
(8,1) node[midway, above] {8}--
(9,1) node[midway, above] {9}--
(10,1) node[midway, above] {10}--
(11,1) node[midway, above] {11}--
(12,1) node[midway, above] {12}--
(13,1) node[midway, above] {13}--
(14,0) node[midway, above] {14};
\draw (1,-1.5)  --
(2,-0.5) node[midway, above] {2}--
(3,-0.5) node[midway, above] {3}--
(4,-0.5) node[midway, above] {4}--
(5,-0.5) node[midway, above] {5}--
(6,-1.5) node[midway, above] {6};
\draw(8,-1.5)--
(9,-0.5) node[midway, above] {9}--
(10,-0.5) node[midway, above] {10}--
(11,-0.5) node[midway, above] {11}--
(12,-0.5) node[midway, above] {12}--
(13,-1.5) node[midway, above] {13};
\draw (2,-3)  --
(3,-2) node[midway, above] {3}--
(4,-2) node[midway, above] {4}--
(5,-3) node[midway, above] {5};
\draw(9,-3)--
(10,-2) node[midway, above] {10}--
(11,-2) node[midway, above] {11}--
(12,-3) node[midway, above] {12};
\draw (10,-4.5);
\draw (7,3) node{$n=2k+1=7$};
\end{tikzpicture}
\hspace{1cm}
\begin{tikzpicture}[scale=0.35]
\draw (0,0)  --
(1,1) node[midway, above] {1}--
(2,1) node[midway, above] {2}--
(3,1) node[midway, above] {3}--
(4,1) node[midway, above] {4}--
(5,1) node[midway, above] {5}--
(6,1) node[midway, above] {6}--
(7,1) node[midway, above] {7}--
(8,1) node[midway, above] {8}--
(9,1) node[midway, above] {9}--
(10,1) node[midway, above] {10}--
(11,1) node[midway, above] {11}--
(12,1) node[midway, above] {12}--
(13,1) node[midway, above] {13}--
(14,0) node[midway, above] {14};
\draw (1,-1.5)  --
(2,-0.5) node[midway, above] {2}--
(3,-0.5) node[midway, above] {3}--
(4,-0.5) node[midway, above] {4}--
(5,-0.5) node[midway, above] {5}--
(6,-0.5) node[midway, above] {6}--
(7,-0.5) node[midway, above] {7}--
(8,-0.5) node[midway, above] {8}--
(9,-0.5) node[midway, above] {9}--
(10,-0.5) node[midway, above] {10}--
(11,-0.5) node[midway, above] {11}--
(12,-0.5) node[midway, above] {12}--
(13,-1.5) node[midway, above] {13};
\draw (2,-3)  --
(3,-2) node[midway, above] {3}--
(4,-2) node[midway, above] {4}--
(5,-2) node[midway, above] {5}--
(6,-2) node[midway, above] {6}--
(7,-3) node[midway, above] {7}--
(8,-2) node[midway, above] {8}--
(9,-2) node[midway, above] {9}--
(10,-2) node[midway, above] {10}--
(11,-2) node[midway, above] {11}--
(12,-3) node[midway, above] {12};
\draw (3,-4.5)  --
(4,-3.5) node[midway, above] {4}--
(5,-3.5) node[midway, above] {5}--
(6,-4.5) node[midway, above] {6};
\draw (8,-4.5)  --
(9,-3.5) node[midway, above] {9}--
(10,-3.5) node[midway, above] {10}--
(11,-4.5) node[midway, above] {11};
\draw (7,3) node{$n=2k+1=9$};
\end{tikzpicture}

\hspace{1cm}\begin{tikzpicture}[scale=0.35]
\draw (0,0)  --
(1,1) node[midway, above] {1}--
(2,1) node[midway, above] {2}--
(3,1) node[midway, above] {3}--
(4,1) node[midway, above] {4}--
(5,1) node[midway, above] {5}--
(6,1) node[midway, above] {6}--
(7,1) node[midway, above] {7}--
(8,1) node[midway, above] {8}--
(9,1) node[midway, above] {9}--
(10,1) node[midway, above] {10}--
(11,1) node[midway, above] {11}--
(12,1) node[midway, above] {12}--
(13,1) node[midway, above] {13}--
(14,0) node[midway, above] {14};
\draw (1,-1.5)  --
(2,-0.5) node[midway, above] {2}--
(3,-0.5) node[midway, above] {3}--
(4,-0.5) node[midway, above] {4}--
(5,-0.5) node[midway, above] {5}--
(6,-0.5) node[midway, above] {6}--
(7,-0.5) node[midway, above] {7}--
(8,-0.5) node[midway, above] {8}--
(9,-0.5) node[midway, above] {9}--
(10,-0.5) node[midway, above] {10}--
(11,-0.5) node[midway, above] {11}--
(12,-0.5) node[midway, above] {12}--
(13,-1.5) node[midway, above] {13};
\draw (2,-3)  --
(3,-2) node[midway, above] {3}--
(4,-2) node[midway, above] {4}--
(5,-2) node[midway, above] {5}--
(6,-2) node[midway, above] {6}--
(7,-2) node[midway, above] {7}--
(8,-2) node[midway, above] {8}--
(9,-2) node[midway, above] {9}--
(10,-2) node[midway, above] {10}--
(11,-2) node[midway, above] {11}--
(12,-3) node[midway, above] {12};
\draw (3,-4.5)  --
(4,-3.5) node[midway, above] {4}--
(5,-3.5) node[midway, above] {5}--
(6,-3.5) node[midway, above] {6}--
(7,-3.5) node[midway, above] {7}--
(8,-3.5) node[midway, above] {8}--
(9,-3.5) node[midway, above] {9}--
(10,-3.5) node[midway, above] {10}--
(11,-4.5) node[midway, above] {11};
\draw (4,-6)  --
(5,-5) node[midway, above] {5}--
(6,-5) node[midway, above] {6}--
(7,-6) node[midway, above] {7}--
(8,-5) node[midway, above] {8}--
(9,-5) node[midway, above] {9}--
(10,-6) node[midway, above] {10};
\draw (7,3) node{$n=2k+1=11$};
\draw (10,-9);
\end{tikzpicture}
\hspace{0.5cm}
\begin{tikzpicture}[scale=0.35]
\draw (0,0)  --
(1,1) node[midway, above] {1}--
(2,1) node[midway, above] {2}--
(3,1) node[midway, above] {3}--
(4,1) node[midway, above] {4}--
(5,1) node[midway, above] {5}--
(6,1) node[midway, above] {6}--
(7,1) node[midway, above] {7}--
(8,1) node[midway, above] {8}--
(9,1) node[midway, above] {9}--
(10,1) node[midway, above] {10}--
(11,1) node[midway, above] {11}--
(12,1) node[midway, above] {12}--
(13,1) node[midway, above] {13}--
(14,0) node[midway, above] {14};
\draw (1,-1.5)  --
(2,-0.5) node[midway, above] {2}--
(3,-0.5) node[midway, above] {3}--
(4,-0.5) node[midway, above] {4}--
(5,-0.5) node[midway, above] {5}--
(6,-0.5) node[midway, above] {6}--
(7,-0.5) node[midway, above] {7}--
(8,-0.5) node[midway, above] {8}--
(9,-0.5) node[midway, above] {9}--
(10,-0.5) node[midway, above] {10}--
(11,-0.5) node[midway, above] {11}--
(12,-0.5) node[midway, above] {12}--
(13,-1.5) node[midway, above] {13};
\draw (2,-3)  --
(3,-2) node[midway, above] {3}--
(4,-2) node[midway, above] {4}--
(5,-2) node[midway, above] {5}--
(6,-2) node[midway, above] {6}--
(7,-2) node[midway, above] {7}--
(8,-2) node[midway, above] {8}--
(9,-2) node[midway, above] {9}--
(10,-2) node[midway, above] {10}--
(11,-2) node[midway, above] {11}--
(12,-3) node[midway, above] {12};
\draw (3,-4.5)  --
(4,-3.5) node[midway, above] {4}--
(5,-3.5) node[midway, above] {5}--
(6,-3.5) node[midway, above] {6}--
(7,-3.5) node[midway, above] {7}--
(8,-3.5) node[midway, above] {8}--
(9,-3.5) node[midway, above] {9}--
(10,-3.5) node[midway, above] {10}--
(11,-4.5) node[midway, above] {11};
\draw (4,-6)  --
(5,-5) node[midway, above] {5}--
(6,-5) node[midway, above] {6}--
(7,-5) node[midway, above] {7}--
(8,-5) node[midway, above] {8}--
(9,-5) node[midway, above] {9}--
(10,-6) node[midway, above] {10};
\draw (5,-7.5)  --
(6,-6.5) node[midway, above] {6}--
(7,-6.5) node[midway, above] {7}--
(8,-6.5) node[midway, above] {8}--
(9,-7.5) node[midway, above] {9};
\draw[dashed] (6,-9)--
(7,-8) node[midway,above] {7}--
(8,-9) node[midway,above] {8};
\draw (6,3) node{$n=2k+1=13$, respectively $n>13$};
\end{tikzpicture}
\end{figure}
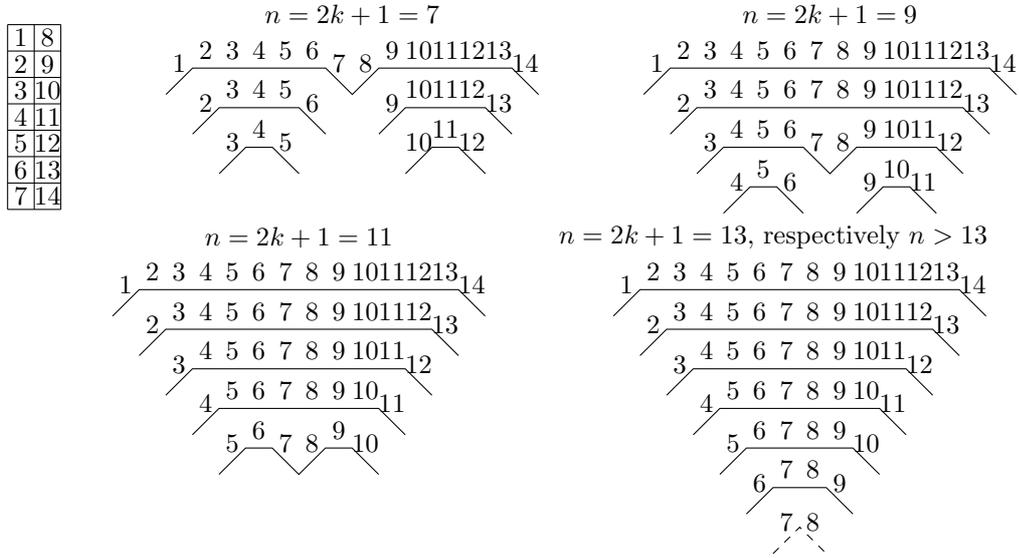

If we ignore everything not concerning $j$ in Algorithm~\ref{alg:2} and point out, that the combination of \emph{separating} left of $a_j$ and \enquote{change $a_{j+1}$ into $0$}, corresponds to \enquote{insert $a$ case 2} while insert the third row in~\cite{3erAlgo}, we get the following:
\begin{theorem}
\label{theo:3erAlgoTheSame}
For tableaux in dimension three Algorithm~\ref{alg:2} and Algorithm~\ref{alg:u2} generate the same output as Algorithm 3 and Algorithm 4 in~\cite{3erAlgo}.
\end{theorem}

\subsection{Properties and Proofs for Bijection $B$}

The first goal of this subsection is to prove the following Theorem:

\begin{theorem}
\label{theo:algo2welldefvactab}
Algorithm~\ref{alg:2} is well-defined and produces a vacillating tableau $V$ of length $r$ and dimension $k$ given a standard Young tableau $Q$ with $2k+1$ rows of even length and $r$ entries.
\end{theorem}

We prepare the proof by stating and proving several lemmas concerning Algorithm~\ref{alg:2}. We use notation form Algorithm~\ref{alg:2}. Variables, etc. also refer to it. Moreover we call marked positions \emph{separation points}.

\begin{corollary}\label{cor:Newjplus1}
For $j$ even, we redefine $a_{j+1}$ to be the next unchanged $j$ so far to the left of $a_j$ and $b_{j+1}$ of be the next changed position to the left during the insertion process of $a$ and $b$ so far. This is just a renaming. However it follows that the $l$-level grows between $c_l$ and $c_{l+1}$, but not somewhere else.
\end{corollary}

\begin{corollary}\label{cor:SeparationInsert}
\emph{Adjust separation point} changes the marked positions as if an $a$ was inserted in between and a $b$ was inserted to the left. 
\end{corollary}

\begin{lemma}
\label{lem:SeparationHeight}
\emph{Separation points} contain \emph{height violations} exactly after initializing a new $j$ and after inserting an even row. This \emph{height violation} is always in $j-1$. This implies that they cause no \emph{height violation} in the end of our insertion process.
\end{lemma}

\begin{proof}
Separation points always start at \emph{separate odd} between $a$ and $b_{j+1}$. As long as they are still between such newly inserted  elements, they expand on $j$. After inserting an odd row there is always another \emph{separate odd} and thus at the last inserted odd row there is no \emph{height violation} anymore.

Once we come into the case \emph{adjust separation point} we do the same as some inserted $a,b$ would do. However one path after the other, beginning with the topmost, is not part of the \emph{separation point} anymore.

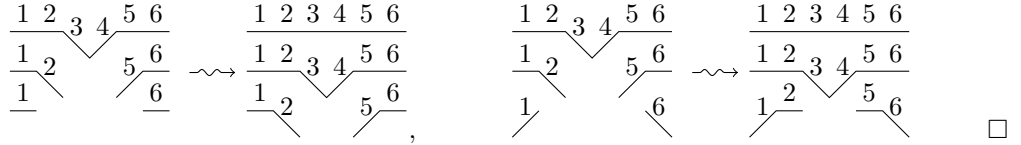
\hspace{1cm}
\begin{tikzpicture}[scale=0.35]
\draw (0,1)  --
(1,1) node[midway, above] {1}--
(2,1) node[midway, above] {2}--
(3,0) node[midway, above] {3}--
(4,1) node[midway, above] {4}--
(5,1) node[midway, above] {5}--
(6,1) node[midway, above] {6};
\draw (0,-0.5)--
(1,-0.5) node[midway, above] {1}--
(2,-1.5) node[midway, above] {2};
\draw (4,-1.5)  --
(5,-0.5) node[midway, above] {5}--
(6,-0.5) node[midway, above] {6};
\draw (0,-2)--(1,-2) node[midway, above] {1};
\draw (5,-2)--(6,-2) node[midway, above] {6};
\draw (0,-3);
\end{tikzpicture}
\begin{tikzpicture}[scale=0.3]
 \draw (0,0.5) node{};
 \draw (0,5) node{};
 \draw [->,decorate,
decoration={snake,amplitude=.4mm,segment length=2mm,post length=1mm, pre length=1mm}] (0,3) -- (2,3);
\end{tikzpicture}
\begin{tikzpicture}[scale=0.35]
\draw (0,1)  --
(1,1) node[midway, above] {1}--
(2,1) node[midway, above] {2}--
(3,1) node[midway, above] {3}--
(4,1) node[midway, above] {4}--
(5,1) node[midway, above] {5}--
(6,1) node[midway, above] {6};
\draw (0,-0.5)--
(1,-0.5) node[midway, above] {1}--
(2,-0.5) node[midway, above] {2}--
(3,-1.5) node[midway, above] {3}--
(4,-0.5) node[midway, above] {4}--
(5,-0.5) node[midway, above] {5}--
(6,-0.5) node[midway, above] {6};
\draw (0,-2)--
(1,-2) node[midway, above] {1}--
(2,-3) node[midway, above] {2};
\draw (4,-3)-- 
(5,-2) node[midway, above] {5}--
(6,-2) node[midway, above] {6};
\end{tikzpicture},
\hspace{1cm}
\begin{tikzpicture}[scale=0.35]
\draw (0,1)  --
(1,1) node[midway, above] {1}--
(2,1) node[midway, above] {2}--
(3,0) node[midway, above] {3}--
(4,1) node[midway, above] {4}--
(5,1) node[midway, above] {5}--
(6,1) node[midway, above] {6};
\draw (0,-0.5)--
(1,-0.5) node[midway, above] {1}--
(2,-1.5) node[midway, above] {2};
\draw (4,-1.5)  --
(5,-0.5) node[midway, above] {5}--
(6,-0.5) node[midway, above] {6};
\draw (0,-3)--(1,-2) node[midway, above] {1};
\draw (5,-2)--(6,-3) node[midway, above] {6};
\end{tikzpicture}
\begin{tikzpicture}[scale=0.3]
 \draw (0,0.5) node{};
 \draw (0,5) node{};
 \draw [->,decorate,
decoration={snake,amplitude=.4mm,segment length=2mm,post length=1mm, pre length=1mm}] (0,3) -- (2,3);
\end{tikzpicture}
\begin{tikzpicture}[scale=0.35]
\draw (0,1)  --
(1,1) node[midway, above] {1}--
(2,1) node[midway, above] {2}--
(3,1) node[midway, above] {3}--
(4,1) node[midway, above] {4}--
(5,1) node[midway, above] {5}--
(6,1) node[midway, above] {6};
\draw (0,-0.5)--
(1,-0.5) node[midway, above] {1}--
(2,-0.5) node[midway, above] {2}--
(3,-1.5) node[midway, above] {3}--
(4,-0.5) node[midway, above] {4}--
(5,-0.5) node[midway, above] {5}--
(6,-0.5) node[midway, above] {6};
\draw (0,-3)--
(1,-2) node[midway, above] {1}--
(2,-2) node[midway, above] {2};
\draw (4,-2)-- 
(5,-2) node[midway, above] {5}--
(6,-3) node[midway, above] {6};
\end{tikzpicture}
\end{proof}

\begin{definition}
The predecessor of some number $a$ or $b$, namely $c$, is the following number $d$. Search for the number $c=c_1$ that is inserted first during the insertion process. This number $c$ has some other number directly below in $Q$, namely $d$. We refer to its insertions with $d_l$ like we do for $a$ or $b$ Algorithm~\ref{alg:2}.
\end{definition}

\begin{lemma}
\label{lem:Predecessor}
If an inserted number $c_{l+1}$ equals its predecessor $d_{l}$, no new \emph{height violation} arises. 
\end{lemma}

\begin{proof}
Inserting $d_l$ made the $l$-level higher between $d_l$ and $d_{l+1}$ and did not change the $(l+1)$-level in this area. When choosing $d_l$ as $c_{l+1}$ this makes the $(l+1)$-level higher in the same area or a smaller one, (and might changes $\pm l$ into $0$'s with a level grow of $1$) which can cause no \emph{height violation}. The following sketch illustrates this for $l+1\neq j$.

\hspace{2cm}
\begin{tikzpicture}[scale=0.35]
\draw (0,1)--(1,0);
\draw [dotted,decorate,
decoration={snake,amplitude=.6mm,segment length=4mm,post length=0mm, pre length=0mm}] (1,0) -- (3,2);
\draw [dotted,decorate,
decoration={snake,amplitude=.6mm,segment length=4mm,post length=0mm, pre length=0mm}] (3,2) -- (5,4);
\draw [dotted,decorate,
decoration={snake,amplitude=.6mm,segment length=4mm,post length=0mm, pre length=0mm}] (1,-2) -- (3,0);
\draw [dotted,decorate,
decoration={snake,amplitude=.6mm,segment length=4mm,post length=0mm, pre length=0mm}] (3,0) -- (5,2);
\end{tikzpicture}
\begin{tikzpicture}[scale=0.3]
 \draw (0,-0.5) node{};
 \draw (0,5) node{};
 \draw [->,decorate,
decoration={snake,amplitude=.4mm,segment length=2mm,post length=1mm, pre length=1mm}] (0,3) -- (2,3);
\end{tikzpicture}
\begin{tikzpicture}[scale=0.35]
\draw (0,1)--(1,1) node[midway, above]{$d_{l+1}$};
\draw [dotted,decorate,
decoration={snake,amplitude=.6mm,segment length=4mm,post length=0mm, pre length=0mm}] (1,1) -- (3,3);
\draw (3,3)--(4,2) node[midway, above]{$d_l$};
\draw [dotted,decorate,
decoration={snake,amplitude=.6mm,segment length=4mm,post length=0mm, pre length=0mm}] (4,2) -- (6,4);
\draw (0,-1)--(1,-2) node[midway,above]{$d_{l+1}$};
\draw [dotted,decorate,
decoration={snake,amplitude=.6mm,segment length=4mm,post length=0mm, pre length=0mm}] (1,-2) -- (3,0);
\draw [dotted,decorate,
decoration={snake,amplitude=.6mm,segment length=4mm,post length=0mm, pre length=0mm}] (4,0) -- (6,2);
\end{tikzpicture}
\begin{tikzpicture}[scale=0.3]
 \draw (0,-0.5) node{};
 \draw (0,5) node{};
 \draw [->,decorate,
decoration={snake,amplitude=.4mm,segment length=2mm,post length=1mm, pre length=1mm}] (0,3) -- (2,3);
\end{tikzpicture}
\begin{tikzpicture}[scale=0.35]
\draw (0,1)--(1,1) node[midway, above]{$d_{l+1}$};
\draw [dotted,decorate,
decoration={snake,amplitude=.6mm,segment length=4mm,post length=0mm, pre length=0mm}] (1,1) -- (3,3);
\draw (3,3)--(4,3) node[midway, above]{$d_l=c_{l+1}$};
\draw [dotted,decorate,
decoration={snake,amplitude=.6mm,segment length=4mm,post length=0mm, pre length=0mm}] (4,3) -- (6,5);
\draw (6,5)--(7,4) node[midway, above]{$c_l$};
\draw (0,-1)--(1,-1) node[midway,below]{$d_{l+1}=c_{l+2}$};
\draw [dotted,decorate,
decoration={snake,amplitude=.6mm,segment length=4mm,post length=0mm, pre length=0mm}] (1,-1) -- (3,1);
\draw (3,1)--(4,0) node[midway,above]{$c_{l+1}$};
\draw [dotted,decorate,
decoration={snake,amplitude=.6mm,segment length=4mm,post length=0mm, pre length=0mm}] (4,0) -- (6,2);
\end{tikzpicture}
\end{proof}

\begin{lemma}
\label{lem:HeightInsert}
\begin{enumerate}
\item If $p$ is a \emph{height violation} in $(l-1)$, $w(p)$ is always $(l-1)$.
\item There are no \emph{height violations} after inserting a pair of $a,b$ if there where no before. The only exceptions are \emph{separation points} where there is a \emph{height violation} in $(j-1)$ before and after inserting an even row.
\item We can always find $a_{j+1}$ and $b_{j+1}$.
\end{enumerate}
\end{lemma}

\begin{proof}
We show these three statements inductively. The base case is clear (empty case). For the inductive step we show one statement after the other.
\begin{enumerate}
\item If there was no \emph{height violation} before, to find a new one, we consider in which situations the level in some paths can grow:
\begin{itemize}
\item between $c_l$ and the first $c_{l+1}$;

\item between a \emph{height violation} in $l$ and the new $c_{l+1}$;

\item between a \emph{height violation} in $l$ and a new \emph{height violation} in $(l-1)$.
\end{itemize}
In this area cannot be a $-l$ except a marked one or $a_l$ if $c$ is a $b$, as this would have been taken for $c_{l+1}$ otherwise.
We illustrated these three cases for $l+1\neq j$.

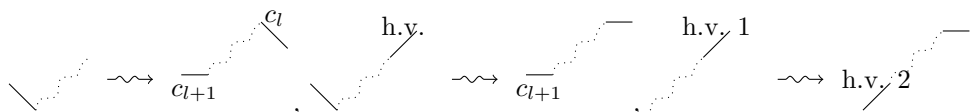
\begin{tikzpicture}[scale=0.35]
\draw (0,1)--(1,0);
\draw [dotted,decorate,
decoration={snake,amplitude=.6mm,segment length=4mm,post length=0mm, pre length=0mm}] (1,0) -- (3,2);
\end{tikzpicture}
\begin{tikzpicture}[scale=0.3]
 \draw (0,2) node{};
 \draw (0,3) node{};
 \draw [->,decorate,
decoration={snake,amplitude=.4mm,segment length=2mm,post length=1mm, pre length=1mm}] (0,3) -- (2,3);
\end{tikzpicture}
\begin{tikzpicture}[scale=0.35]
\draw (0,1)--(1,1) node[midway, below]{$c_{l+1}$};
\draw [dotted,decorate,
decoration={snake,amplitude=.6mm,segment length=4mm,post length=0mm, pre length=0mm}] (1,1) -- (3,3);
\draw (3,3)--(4,2) node[midway, above]{$c_l$};
\end{tikzpicture},
\begin{tikzpicture}[scale=0.35]
\draw (0,1)--(1,0);
\draw [dotted,decorate,
decoration={snake,amplitude=.6mm,segment length=4mm,post length=0mm, pre length=0mm}] (1,0) -- (3,2);
\draw (3,2)--(4,3) node[midway, above]{h.v.};
\end{tikzpicture}
\begin{tikzpicture}[scale=0.3]
 \draw (0,2) node{};
 \draw (0,3) node{};
 \draw [->,decorate,
decoration={snake,amplitude=.4mm,segment length=2mm,post length=1mm, pre length=1mm}] (0,3) -- (2,3);
\end{tikzpicture}
\begin{tikzpicture}[scale=0.35]
\draw (0,1)--(1,1) node[midway, below]{$c_{l+1}$};
\draw [dotted,decorate,
decoration={snake,amplitude=.6mm,segment length=4mm,post length=0mm, pre length=0mm}] (1,1) -- (3,3);
\draw (3,3)--(4,3);
\end{tikzpicture},
\begin{tikzpicture}[scale=0.35]
\draw [dotted,decorate,
decoration={snake,amplitude=.6mm,segment length=4mm,post length=0mm, pre length=0mm}] (1,0) -- (3,2);
\draw (3,2)--(4,3) node[midway, above]{h.v. 1};
\end{tikzpicture}
\begin{tikzpicture}[scale=0.3]
 \draw (0,2) node{};
 \draw (0,3) node{};
 \draw [->,decorate,
decoration={snake,amplitude=.4mm,segment length=2mm,post length=1mm, pre length=1mm}] (0,3) -- (2,3);
\end{tikzpicture}
\begin{tikzpicture}[scale=0.35]
\draw (0,0)--(1,1) node[midway, above]{h.v. 2};
\draw [dotted,decorate,
decoration={snake,amplitude=.6mm,segment length=4mm,post length=0mm, pre length=0mm}] (1,1) -- (3,3);
\draw (3,3)--(4,3);
\end{tikzpicture}
\begin{itemize}
\item \emph{$i$ odd connect} or \emph{$i$ even connect};

This happens between $a$ and $b$ and can only create a \emph{height violation} at a position with $j$-level $0$, thus a $j-1$. ($-j+1$ could also create a \emph{height violation} but needs to be left of such a $j-1$). At \emph{$i$ even connect} such positions are marked.
\end{itemize}

Height violations in $(l-1)$ only happen in those situations. $p$ cannot be a $0$ or else the previous $p$ would be a \emph{height violation} as well and we see inductively that right of $p$ there are no \emph{height violations}. The same argument holds for $l$ and $-(l-1)$. We have seen that in the area where word $l$ gets higher there is either no $-l$ or it is ignored ($a_l$) or it is to the left of a $l-1$ (which is marked). Therefore the only possible value is $l-1$.

\item Therefore at a \emph{height violation}, the $(l-1)$-level increases and the $l$-level decreases. However as this happened somewhere where the $l$-level was increased before (by inserting $c_l$), it sets the $l$-level to its original height. Thus there is no \emph{height violation} until another level growth.

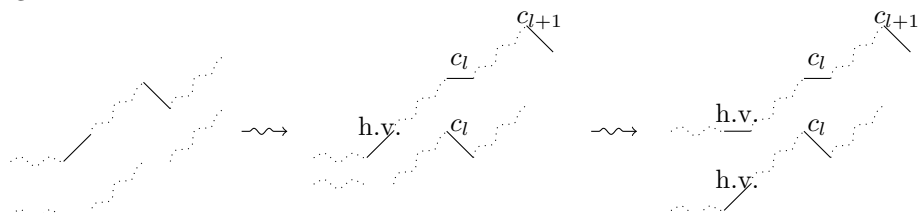
\hspace{0.2cm}
\begin{tikzpicture}[scale=0.35]
\draw [dotted,decorate,
decoration={snake,amplitude=.6mm,segment length=4mm,post length=0mm, pre length=0mm}] (-2,-1) -- (0,-1);
\draw (0,-1)--(1,0);
\draw [dotted,decorate,
decoration={snake,amplitude=.6mm,segment length=4mm,post length=0mm, pre length=0mm}] (1,0) -- (3,2);
\draw (3,2)--(4,1);
\draw [dotted,decorate,
decoration={snake,amplitude=.6mm,segment length=4mm,post length=0mm, pre length=0mm}] (4,1) -- (6,3);
\draw [dotted,decorate,
decoration={snake,amplitude=.6mm,segment length=4mm,post length=0mm, pre length=0mm}] (-2,-3) -- (0,-3);
\draw [dotted,decorate,
decoration={snake,amplitude=.6mm,segment length=4mm,post length=0mm, pre length=0mm}] (1,-3) -- (3,-1);
\draw [dotted,decorate,
decoration={snake,amplitude=.6mm,segment length=4mm,post length=0mm, pre length=0mm}] (4,-1) -- (6,1);
\end{tikzpicture}
\begin{tikzpicture}[scale=0.3]
 \draw (0,-0.5) node{};
 \draw (0,5) node{};
 \draw [->,decorate,
decoration={snake,amplitude=.4mm,segment length=2mm,post length=1mm, pre length=1mm}] (0,3) -- (2,3);
\end{tikzpicture}
\hspace{0.1cm}
\begin{tikzpicture}[scale=0.35]
\draw [dotted,decorate,
decoration={snake,amplitude=.6mm,segment length=4mm,post length=0mm, pre length=0mm}] (-2,-1) -- (0,-1);
\draw (0,-1)--(1,0) node[midway, above]{h.v.};
\draw [dotted,decorate,
decoration={snake,amplitude=.6mm,segment length=4mm,post length=0mm, pre length=0mm}] (1,0) -- (3,2);
\draw (3,2)--(4,2) node[midway, above]{$c_{l}$};
\draw [dotted,decorate,
decoration={snake,amplitude=.6mm,segment length=4mm,post length=0mm, pre length=0mm}] (4,2) -- (6,4);
\draw (6,4)--(7,3) node[midway, above]{$c_{l+1}$};
\draw [dotted,decorate,
decoration={snake,amplitude=.6mm,segment length=4mm,post length=0mm, pre length=0mm}] (-2,-2) -- (0,-2);
\draw [dotted,decorate,
decoration={snake,amplitude=.6mm,segment length=4mm,post length=0mm, pre length=0mm}] (1,-2) -- (3,0);
\draw (3,0)--(4,-1) node[midway, above]{$c_{l}$};
\draw [dotted,decorate,
decoration={snake,amplitude=.6mm,segment length=4mm,post length=0mm, pre length=0mm}] (4,-1) -- (6,1);
\draw (0,-3) node{};
\end{tikzpicture}
\hspace{-0.1cm}
\begin{tikzpicture}[scale=0.3]
 \draw (0,-0.5) node{};
 \draw (0,5) node{};
 \draw [->,decorate,
decoration={snake,amplitude=.4mm,segment length=2mm,post length=1mm, pre length=1mm}] (0,3) -- (2,3);
\end{tikzpicture}
\hspace{0.2cm}
\begin{tikzpicture}[scale=0.35]
\draw [dotted,decorate,
decoration={snake,amplitude=.6mm,segment length=4mm,post length=0mm, pre length=0mm}] (-2,0) -- (0,0);
\draw (0,0)--(1,0) node[midway, above]{h.v.};
\draw [dotted,decorate,
decoration={snake,amplitude=.6mm,segment length=4mm,post length=0mm, pre length=0mm}] (1,0) -- (3,2);
\draw (3,2)--(4,2) node[midway, above]{$c_{l}$};
\draw [dotted,decorate,
decoration={snake,amplitude=.6mm,segment length=4mm,post length=0mm, pre length=0mm}] (4,2) -- (6,4);
\draw (6,4)--(7,3) node[midway, above]{$c_{l+1}$};
\draw [dotted,decorate,
decoration={snake,amplitude=.6mm,segment length=4mm,post length=0mm, pre length=0mm}] (-2,-3) -- (0,-3);
\draw (0,-3)--(1,-2) node[midway, above]{h.v.};
\draw [dotted,decorate,
decoration={snake,amplitude=.6mm,segment length=4mm,post length=0mm, pre length=0mm}] (1,-2) -- (3,0);
\draw (3,0)--(4,-1) node[midway, above]{$c_{l}$};
\draw [dotted,decorate,
decoration={snake,amplitude=.6mm,segment length=4mm,post length=0mm, pre length=0mm}] (4,-1) -- (6,1);
\draw (0,-3) node{};
\end{tikzpicture}

\emph{Height violations} can only add up in pairs of two (\emph{connect} only happens left of $b_{j+1}$). If that happens, also the levels of that \emph{height violations} add up, so it is sufficient to look at each situation separately. The \emph{ignore $a_l$} ensures that everything is considered separately. (Compare with the first tableau in Example~\ref{ex:SpecialCases}.)

\item To conclude we note that once we reach the predecessor no new \emph{height violation} arises (compare with Lemma~\ref{lem:Predecessor}). This also implies inductively that the predecessor cannot be taken from a pair before, as their predecessors are the leftmost positions they can take.

It remains to show that the predecessor $d_l$ can be taken indeed as new $c_{l+1}$. This holds for $l+1\neq j$ as  in this case $d_{l}=-l$. For $l+1=j$ we distinguish between $i$ even and $i$ odd.

If $i$ is odd we can argue the same except for the case that $d_l=0$. In this case, however, \emph{$i$ odd separate} changes this $d_l$ into a $-l$.

If $i$ is even we can argue that in \emph{$i$ odd} $a$'s produce $0$'s that are in $0$-even positions and $b$'s produce $0$'s that are in $0$-odd positions. Now $a_{j+1}$'s are always $j$-even positions and $b_{j+1}$'s are always $j$-odd positions and every such position can be chosen as such.\qedhere
\end{enumerate}
\end{proof}

\begin{remark}
Left sides (down-steps) of \emph{separation points} can never be inserted as $b$ thus are never predecessors of $b$. 
\end{remark}

\begin{lemma}
\label{lem:Sumis0}
The sum over the labeled word $w$ is $0$ after every insertion of a pair $a,b$. In particular the sum over $\pm l$ in $w$ is $0$ for every $l$.
\end{lemma}

\begin{proof}
The sum over all $j$'s is $0$ after initializing $j$ as there is always an even number of $0$'s. Thus we have to show that nothing we do during the insertion process, changes the sum:
\begin{itemize}
\item If $i$ is even and we insert $a_j$, $a_{j+1}$ and $b_j$, $b_{j+1}$ we insert either two $-j$'s and change one $-j$ into a $j$ or we insert one $-j$ and one $0$ and change one $-j$ into a $0$. Otherwise we insert something with $-l$ and change another $-l$ into a $0$ or a $(-l+1)$.
\item When \emph{$i$ even connect}, \emph{$i$ odd connect} or \emph{adjust separating point} we always change an $l$ and a $-l$ into $0$'s or into $l+1$ and $-l-1$. At \emph{$i$ odd separate} we change two $0$'s into a $-l$ and $l$.
\item At \emph{height violation} we do the inverse of finding a $c_{l}$, namely changing a $l-1$ into an $l$ instead of changing an $-l-1$ into an $-l$. As we insert $c_l$ anew later on, this does not change the sum. However, there are two situations where it could be that $c_{l+1}$ is already found but there is still a \emph{height violation}. Therefore we have to adjust the path to ensure sum $0$ in this situations. Those are the two special cases. At \emph{$i$ odd connect} we deal with this by changing a $0$ into a $-j=-l-1$ again. At \emph{$i$ even $a_{j+1}$ 1} we defined and adjusted $a_{j+1}$ before and deal with this by changing it back again. \qedhere
\end{itemize}
\end{proof}

\begin{proof}[Proof of Theorem~\ref{theo:algo2welldefvactab}]
For well-definedness, we have to show that the while loop always terminates, thus that we find an $a_{j+1}$. We have seen this in Lemma~\ref{lem:HeightInsert}.

Moreover we have to show that the vacillating tableau properties hold for our resulting word:
\begin{enumerate}
\item In every initial segment the following holds:
\begin{enumerate}
\item \label{p:WellDefBottom} $\#i - \#(-i) \geq 0$,
\item \label{p:WellDefHeight} $\#i-\#(-i)\geq \#(i+1)-\#(-i-1)$,
\item \label{p:WellDefZero}if the last position is $0$ then $\#k-\#(-k)>0$.
\end{enumerate}
\item \label{p:WellDefSum} The sum over all positions is $0$.
\end{enumerate}

To show that Property~\ref{p:WellDefBottom} is satisfied after any insertion of a pair $a,b$, we have to show that there are no steps with negative $l$-level. There are two steps in the algorithm where we decrease the level of some position. At the first one, \emph{separate odd}, we generate $-j,j$ on $j$-level one. At the second one, \emph{height violation}, we have seen in Lemma~\ref{lem:HeightInsert} that we decrease positions that have been increased before.

To show that Property~\ref{p:WellDefHeight} is satisfied, we have to show that there is no \emph{height violation}. This is shown in Lemma~\ref{lem:HeightInsert} and Lemma~\ref{lem:SeparationHeight}.

To show that Property~\ref{p:WellDefZero} is satisfied we show that $0$'s are always at least on $k$-level one. When initializing a new $j$, $0$'s get changed into $\pm j$. New $0$'s come either from \emph{connect}, where they are on level one or at \emph{$c_{j+1}$, $i$ odd}, where we change a $-j$ on level at least $0$ to a $0$ with level at least one or more. 

Property~\ref{p:WellDefSum} is shown in Lemma~\ref{lem:Sumis0}.

Finally, the number of steps is $r$ as every entry of $Q$ inserts exactly one step.
\end{proof}

Due to what we have seen about the predecessor, the following lemma holds for even length paths and standard Young tableau with all rows of even length.

\begin{lemma}
\label{theo:Concat}
Considering Algorithm~\ref{alg:2}, concatenation of vacillating tableaux of empty shape and even length corresponds to concatenation of standard Young tableaux whose rows have even length.

In particular, the following holds:
\begin{itemize}
\item If a vacillating tableau is composed of two concatenated paths of empty shape and even length, its corresponding standard Young tableau can be written as concatenation of two standard Young tableaux all whose rows have even length.
\item On the other hand if a standard Young tableau can be written as concatenation of two standard Young tableaux whose rows have even length, its corresponding vacillating tableau is also composed of two concatenated paths of empty shape and even length.
\end{itemize}
\end{lemma}

Now we want to show the same for Algorithm~\ref{alg:u2} which we will prove later to be the reversed algorithm of Algorithm~\ref{alg:2}. To see this we will again provide and prove several lemmas first.

\begin{theorem}
\label{theo:ualgo2welldefSYT}
Algorithm~\ref{alg:u2} is well-defined and produces a standard Young tableau, with rows of even length and $r$ entries, given a vacillating tableau of even length $r$ and empty shape.
\end{theorem}

\begin{lemma}
\label{lem:Ualgo2SumIsZero}
The sum over positions in the labeled word $w$ is $0$ after any extraction of a pair $a,b$. In particular the sum over $\pm l$ in $w$ stays $0$.
\end{lemma}

\begin{proof}
For every $c_l$ that is changed from $-l$ into $-(l-1)$ we change a $-(l-1)$ into a $-(l-2)$. If we consider $c_j$, $i$ odd, we loose a $0$ in this process. If we consider $c_j$, $i$ even, we conclude that we either extract two $-j$'s and change a $j$ into a $-j$ or we extract one $0$ and one $-j$ and change a $0$ into a $-j$. If we consider $1$ we simply delete a $-1$ after we inserted a new one to the left.

\emph{Connect}, \emph{separate}, \emph{height violations} and \emph{separation points} also just change $l$ in pairs of two - always an $l$ and a $-l$. For details compare this with the proof of Lemma~\ref{lem:Sumis0}.
\end{proof}

\begin{lemma}
\label{lem:Ualgo2HeightVio}
If $p$ is a \emph{height violation} in $l$, $p$ is always an $l+1$. After extracting a pair of $a,b$ there is no \emph{height violation} if there was no before and the extraction process stops. Again the only exceptions are \emph{separation points}.
\end{lemma}

\begin{proof}
Once again we consider $a$ and $b$ separately as combined \emph{height violations} just add up. \emph{Ignore positions that are corrected height violations of $a$} ensures that everything is considered separately. (Again, compare with the first tableau in Example~\ref{ex:SpecialCases}.)

For a \emph{height violation} we have to consider where the $l$-level for $l>j$ is decreased:
\begin{itemize}
\item between $c_{l+1}$ and $c_l$;
\item between $c_{l+1}$ and a \emph{height violation} in $l+1$; 
\item when \emph{adjusting a height violation} in $l-1$ until finding a new $c_{l}$ or a new \emph{height violation};
\end{itemize}
In this area cannot be a $-l$ except for a marked one, if $c$ is a $b$, as this would have been taken for $c_l$.

In the proof of Lemma~\ref{lem:HeightInsert} those situations are illustrated for Algorithm 1. Here the situation is similar but the other way around.

Moreover we have consider level increasings. Those are only possible for the $j$-level by \emph{separation odd}. However this happens at $j$-even positions, so if there is something on $j$-level zero we change it into $j-1$ and set $b_{j+1}$ and $b_j$ to undefined in \emph{height violations special cases}. Thus there is no \emph{height violation} afterwards.

An $l$ is not a \emph{height violation}, as there would have been be one before. The same holds for a $-(l+1)$. Moreover it cannot be a $-l$ as we have seen in the list above, thus it is an $l+1$. At \emph{adjust height violations} this $l+1$ is changed into an $l$, thus we increase the $l$-level and decrease the $(l+1)$-level. However this happens somewhere, where the $(l+1)$-level has been increased before by choosing $c_{l+1}$.

Again the illustrations in the proof of Lemma~\ref{lem:HeightInsert} show the same situations in Algorithm~\ref{alg:2}. 
\end{proof}

\begin{lemma}
\label{lem:Ualgo2AB}
We always find an $a_l$ and a $b_l$ if we found an $a_{l+1}$ and a $b_{l+1}$. After several steps we find an $a_l$ and a $b_l$ whose extraction does not cause a new \emph{height violation}.
\end{lemma}

\begin{proof}
We start with $a_{j+1}$ and $b_{j+1}$  and distinguish the parity of $i$.

If $i$ is odd, $a_{j+1}$ changes a $0$ in 3-row-position. (If none exists anymore due to \emph{adjust separation point}, nothing happens.) This is the only time when a $0$ is changed except for $b_{j+1}$ and \emph{connect/separate odd}, where $0$'s are changed in pairs of two. Thus after finding $a_{j+1}$ there will remain an odd number of $0$'s, therefore we will find a $b_{j+1}$.

If $i$ is even, we only need to find $a_{j+1}$ as this gives us both a first $a_j$ and a first $b_j$ ($a_j$ is a $0$ or a $-j$ on $j$-level at least one, thus we find $b_j$ as a $-j$ on $j$-level at least zero). We find it left of a 2-row-position. Again if none exists anymore, nothing happens.

For $l<j+1$ we always find a first $c_l$ as $c_{l+1}$ used to be a $0$ not on level zero, so there is a $-l$ to the right. After a \emph{height violation}, we insert an $l$ and get the old $l$-level back so we can find another $-l$. If there would be only one $-l$ for both $a$ and $b$, that would mean, that both $a_{l+1}$ and $b_{l+1}$ were $-(l+1)$ on $l$-level zero with no $-l$ in between. However this would either cause a \emph{height violation} in $a_{l+1}$, which is a contradiction or there is an $(l+1)$ on $(l+1)$-level $0$ next to $a_{l+1}$ and $b_{l+1},\dots,b_j$ are right of that. Moreover $a_{l+1},\dots,a_j$ are on level $0$ each. As there is no $-l$ between $a_{l+1}$ and this $(l+1)$ $a_j$ satisfies the conditions for \emph{adjust separation point} as there is no $l$ in between either ($b_{l+1}$ is also on level $0$), which is also a contradiction.

We find a $c_l$ whose extraction does not cause a new \emph{height violation} as this is the case once we reach a $-l$ not followed by an $l$.  This is the case at some point in a path without \emph{height violations}, that ends on $l$-level zero for all $l$ due to Lemma~\ref{lem:Ualgo2SumIsZero} and Lemma~\ref{lem:Ualgo2HeightVio}.
\end{proof}

\begin{lemma}
\label{lem:UalgoSYT}
For each extracted number in row $i+1$ we extract at least one smaller number in row $i$.
\end{lemma}

\begin{proof}
We show first that that every last $d_{l+1}$ can be a last $c_{l}$ for $l\neq j$: Extracting $d_{l+1}$ decreased the $l$-level by $1$ for $l<j$ between $d_{l+1}$ and $d_l$ without changing the $(l-1)$-level there. Thus we can decrease the $(l-1)$-level in this area when extracting the next row, thus this could be a $c_l$ that causes no further \emph{height violations}. (Compare with the illustrations in the proof of Lemma~\ref{lem:Predecessor}.)

Now we show that each extracted $d_j$ causes a $c_{j+1}$. Again we do so by distinguishing the parity of $i$ when extracting $c_j$.

If $i$ is odd, we extracted $-(j+1)$'s before in an even process. We need to show that those produced $0$'s in 3-row-positions and that we can take those as $a_{j+1}$ and $b_{j+1}$. We distinguish two cases.
If those are separated by $a_j$, they produce automatically two odd sequences of $0$'s, one of which to take. It could be the case that one of them gets even due to some former or later $c_j$ in this round, however this is the same situation as in the next case. If those are not separated by $a_j$, there cannot be a $-j$ between those. Thus at least the right one is on $j$-level two or higher and the other one is either on the same level, or if separated by a $j$, it is in an odd sequence.

If $i$ is even, we extracted $j$'s in an $i$-odd-process before. Thus, due to Corollary~\ref{cor:SeparationInsert} we can use Lemma~34 and its proof in \cite{3erAlgo}.

It remains to show that those new 2- or 3-row-positions will not be changed in \emph{adjust separating points}. This follows as the extracting process of $a$ will leave some negative step in between or will extract a $-1$.
\end{proof}

\begin{proof}[Proof of Theorem~\ref{theo:ualgo2welldefSYT}]
For well-definedness we have to show that the two while loops terminate. That the inner one terminates ensure Lemma~\ref{lem:Ualgo2HeightVio} and Lemma~\ref{lem:Ualgo2AB}. The outer one has to terminate, as with each extraction the word gets smaller by two, so in the end there is nothing left to build a 2-or 3-row position.

Using Lemma~\ref{lem:UalgoSYT}, we see that it produces a standard Young tableau with even row lengths.
\end{proof}

\begin{theorem}
\label{theo:algo2ualgo2Inverse}
Algorithms~\ref{alg:2} and~\ref{alg:u2} are inverse.
\end{theorem}

\begin{proof}

We show this by showing that every step has its inverse in the other algorithm. Those steps are named (commented) the same. We consider them separately. 

For the following steps it follows directly from the definition that they are inverse:
\begin{itemize}
\item Initialize $j$
\item Insert / extract row 1
\item $a$ / $b$ getting inserted or extracted as $a_1$ / $b_1$
\end{itemize}

For the following steps we have to argue a little more:

\begin{itemize}
\item \emph{Height violations:}

We show that if a \emph{height violation} $h$ with $w(h)=l$ in Algorithm~\ref{alg:2} occurs after inserting $c^1_{l+1}$ and we correct it, and insert $c^2_{l+1}$ later, we get the same $h$ with $w(h)=l+1$ as \emph{height violation} when extracting this $c^2_{l+1}$, and the other way around. This is sufficient as due to the fact that \emph{height violations} in $l$ are always $l$ respectively $l+1$, they act inverse.

When we \emph{adjust} a \emph{height violation} $h$ in Algorithm~\ref{alg:2}, we get a $w(h)=l+1$ whose $(l+1)$-level is one less than its $l$-level. When extracting $c^2_{l+1}$ in Algorithm~\ref{alg:u2}, this decreases the $l$-level to the right, and $h$ is the first position with an $l$-level that is too large, as the other positions were no \emph{height violation} before starting \emph{height violation} in Algorithm~\ref{alg:2}. The other way around is similar. If we find a \emph{height violation} $w(h)=l+1$ caused by extracting $c^2_{l+1}$ in Alorithm~\ref{alg:u2} we change this into an $l$ and extract a $c^1_{l+1}$ to the right. When inserting $c^1_{l+1}$ in Algorithm~\ref{alg:2}, we increase the $(l+1)$-level such that exactly at $h$ there is a new \emph{height violation}. Again nothing earlier could have caused it as those positions were no \emph{height violations} in Algorithm~\ref{alg:u2}. We illustrated this for $l+1\neq j$ in the following sketch.

\noindent 
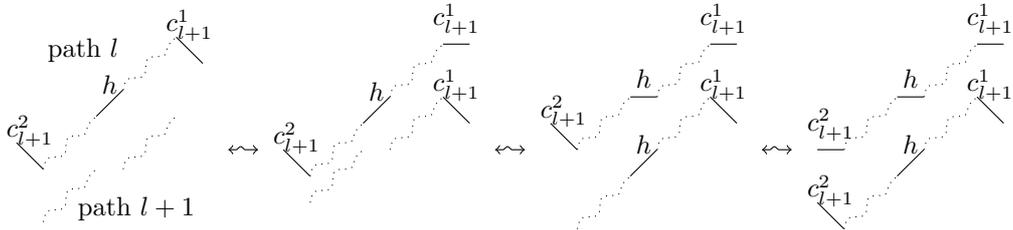
\begin{tikzpicture}[scale=0.35]
\draw (2.5,4.5) node{path $l$};
\draw (0,1)--(1,0) node[midway, above]{$c^2_{l+1}$};
\draw [dotted,decorate, decoration={snake,amplitude=.6mm,segment length=4mm,post length=0mm, pre length=0mm}] (1,0) -- (3,2);
\draw (3,2)--(4,3)node[midway, above]{$h$};
\draw [dotted,decorate, decoration={snake,amplitude=.6mm,segment length=4mm,post length=0mm, pre length=0mm}] (4,3) -- (6,5);
\draw (6,5)--(7,4) node[midway, above]{$c^1_{l+1}$};
\draw [dotted,decorate, decoration={snake,amplitude=.6mm,segment length=4mm,post length=0mm, pre length=0mm}] (1,-2) -- (3,0);
\draw [dotted,decorate, decoration={snake,amplitude=.6mm,segment length=4mm,post length=0mm, pre length=0mm}] (4,0) -- (6,2);
\draw (4.5,-1.5) node{path $l+1$};
\end{tikzpicture}
\begin{tikzpicture}[scale=0.3]
 \draw (0,-0.5);
 \draw (0,5);
 \draw [<->,decorate, decoration={snake,amplitude=.4mm,segment length=2mm,post length=1mm, pre length=1mm}] (0,3) -- (1.3,3);
\end{tikzpicture}
\begin{tikzpicture}[scale=0.35]
\draw (0,1)--(1,0) node[midway, above]{$c^2_{l+1}$};
\draw [dotted,decorate, decoration={snake,amplitude=.6mm,segment length=4mm,post length=0mm, pre length=0mm}] (1,0) -- (3,2);
\draw (3,2)--(4,3)node[midway, above]{$h$};
\draw [dotted,decorate, decoration={snake,amplitude=.6mm,segment length=4mm,post length=0mm, pre length=0mm}] (4,3) -- (6,5);
\draw (6,5)--(7,5) node[midway, above]{$c^1_{l+1}$};
\draw (0,-2);
\draw [dotted,decorate, decoration={snake,amplitude=.6mm,segment length=4mm,post length=0mm, pre length=0mm}] (1,-1) -- (3,1);
\draw [dotted,decorate, decoration={snake,amplitude=.6mm,segment length=4mm,post length=0mm, pre length=0mm}] (4,1) -- (6,3);
\draw (6,3)--(7,2) node[midway, above]{$c^1_{l+1}$};
\end{tikzpicture}
\begin{tikzpicture}[scale=0.3]
 \draw (0,-0.5);
 \draw (0,5);
 \draw [<->,decorate, decoration={snake,amplitude=.4mm,segment length=2mm,post length=1mm, pre length=1mm}] (0,3) -- (1.3,3);
\end{tikzpicture}
\begin{tikzpicture}[scale=0.35]
\draw (0,2)--(1,1) node[midway, above]{$c^2_{l+1}$};
\draw [dotted,decorate, decoration={snake,amplitude=.6mm,segment length=4mm,post length=0mm, pre length=0mm}] (1,1) -- (3,3);
\draw (3,3)--(4,3)node[midway, above]{$h$};
\draw [dotted,decorate, decoration={snake,amplitude=.6mm,segment length=4mm,post length=0mm, pre length=0mm}] (4,3) -- (6,5);
\draw (6,5)--(7,5) node[midway, above]{$c^1_{l+1}$};
\draw [dotted,decorate, decoration={snake,amplitude=.6mm,segment length=4mm,post length=0mm, pre length=0mm}] (1,-2) -- (3,0);
\draw (3,0)--(4,1)node[midway, above]{$h$};
\draw [dotted,decorate, decoration={snake,amplitude=.6mm,segment length=4mm,post length=0mm, pre length=0mm}] (4,1) -- (6,3);
\draw (6,3)--(7,2) node[midway, above]{$c^1_{l+1}$};
\end{tikzpicture}
\begin{tikzpicture}[scale=0.3]
 \draw (0,-0.5);
 \draw (0,5);
 \draw [<->,decorate, decoration={snake,amplitude=.4mm,segment length=2mm,post length=1mm, pre length=1mm}] (0,3) -- (1.3,3);
\end{tikzpicture}
\begin{tikzpicture}[scale=0.35]
\draw (0,1)--(1,1) node[midway, above]{$c^2_{l+1}$};
\draw [dotted,decorate, decoration={snake,amplitude=.6mm,segment length=4mm,post length=0mm, pre length=0mm}] (1,1) -- (3,3);
\draw (3,3)--(4,3)node[midway, above]{$h$};
\draw [dotted,decorate, decoration={snake,amplitude=.6mm,segment length=4mm,post length=0mm, pre length=0mm}] (4,3) -- (6,5);
\draw (6,5)--(7,5) node[midway, above]{$c^1_{l+1}$};
\draw (0,-1)--(1,-2) node[midway, above]{$c^2_{l+1}$};
\draw [dotted,decorate, decoration={snake,amplitude=.6mm,segment length=4mm,post length=0mm, pre length=0mm}] (1,-2) -- (3,0);
\draw (3,0)--(4,1)node[midway, above]{$h$};
\draw [dotted,decorate, decoration={snake,amplitude=.6mm,segment length=4mm,post length=0mm, pre length=0mm}] (4,1) -- (6,3);
\draw (6,3)--(7,2) node[midway, above]{$c^1_{l+1}$};
\end{tikzpicture}

The three special cases are left to consider.

\emph{Ignore $a_l$} at a \emph{height violation} happens when $a_l$ is already inserted but $a_{l+1}$ is not, thus the level of path $l$ is changed but the level of path $l+1$ is not. When we do the inverse this is not relevant as in this case $a$ is always extracted further than $b$. Thus we have to do the same if $b_{l+1}$ is extracted but $b_{l}$ is not. Therefore this special case changes nothing from the argumentation above.

The special case \emph{$i$ even}, happens if there is a \emph{height violation} between $a_{j+1}$ and $a_j$. Thus, together with finding a new $a_j$, it sets $-(j-1),(j-1)$ between those to $0,0$. Therefore when we extract we get the left $0$ as $a_j$ and change it into $-(j-1)$ again. The special case changes the right $0$ into $(j-1)$ and sets $a_j$ and $a_{j+1}$ to undefined. For an illustration of an example see the fourth tableau of Example~\ref{ex:SpecialCases}. The other direction works the same way.

The special case \emph{$i$ odd} happens after \emph{connect}. Thus it changes a $-j$ changed into a $0$ back into a $-j$. Then $b_j$ is inserted anew as this $-j$. When it gets extracted, it detects \emph{height violations}. Correcting them leaves the $0$ we produced on $j$-level $0$, which we change back into a $j$. For an illustration of an example see the fifth tableau of Example~\ref{ex:SpecialCases}. The other direction works the same way.
We point out that Algorithm~\ref{alg:1} \emph{connects} between $a_{j+1}$ and $b_{j+1}$ at $j$-level $0$ whenever $i$ is odd.

\item \emph{Separation points:}

In both algorithms we \emph{mark} and \emph{adjust separation points} while searching for $a$'s and $b$'s. This way we \emph{adjust separation points} before reaching the next $a$ and $b$. Thus in Algorithm~\ref{alg:2} this happens right of $a$ and $b$ and in Algorithm~\ref{alg:u2} this happens to the left. This makes no difference as we consider everything still in the same order, and make an extra iteration for \emph{separating points} at the ends not considered so far.

We mark positions $\pm l$ at certain points, to make certain exceptions for them. We mark $0$'s also (in a slightly different way), but those are not relevant as those are never such exceptions.

The \emph{separation points} we just mark are for an $l$ between $a_l$ and $b_{j+1}$ and mark positions $\pm l$ up to $\pm j$. Due to \emph{$i$ odd separate} and \emph{$i$ even connect}, we mark in each algorithm positions that way that they form the same pattern after the iteration as the other algorithms marks in the beginning of an iteration. The marking $a_l$ ensures that even though $a$ might be inserted on the left part of the \emph{separation point}, all $\pm l$ that should be marked are marked.

The \emph{separation points} we \emph{mark} and \emph{adjust} in Algorithm~\ref{alg:2} are for no $l$ between $a_l$ and $b_{j+1}$. When we \emph{adjust a separation point} in Algorithm~\ref{alg:2}, it is automatically right of the current $b_l$. Thus, what we have to show is that exactly \emph{separation points} we \emph{adjust} form the patterns we demand in Algorithm~\ref{alg:u2} \emph{adjust separation points}.

When an $a$ is in between there are three different ways it can be so. When $a$ is inserted as such, we have a marked $-1$. When $a$ starts to be in between in the first path marked, lets call it $l$, then $a_{l-1}$ is either between the marked positions $\pm l$, therefore a $-(l-1)$ is between some marked positions or it is not, thus it changed a marked $-(l-1)$ into a $-l$ but not the according $(l-1)$. When $a$ causes a \emph{height violation}  even though it is marked and therefore $p<a_l$, we can argue as above. This explains why we look for $\pm l$ between the $0$'s directly to the right. We do so because in Algorithm~\ref{alg:2} we mark all $\pm l$ between a $-j$ until the next $j$, which become the leftmost and rightmost $0$ of their sequence of $0$'s.

When we \emph{adjust a separation point}, we shift the $\pm l$ upwards, as $-(l-1)$ was not between $\pm(l-1)$ it is now not between $\pm l$ any more.
\end{itemize}

With the knowledge of those, the following gets easier:
\begin{itemize}
\item \emph{$i$ even, $i$ odd}:

It remains to show that everything that does not involve marking is inverse. Due to Corollary~\ref{cor:SeparationInsert} we can use Lemma~36 in~\cite{3erAlgo} to show this. We point out that the main arguments include analyzing 2- and 3-row-positions.

\item $a_l$ / $b_l$:

We point out, that we have an index shift of one at $l$ between the formulations. Once we consider this, we see that they operate clearly in the opposite way.
It remains to show, that they act on the same positions. As they always take the next $-l$, and change it, there is no $-l$ that could be taken before from the other algorithm. \qedhere 
\end{itemize}
\end{proof}

\begin{theorem}
\label{theo:algo2DescPres}
Algorithm~\ref{alg:2} is descent preserving.
\end{theorem}

\begin{proof}
This proof is very similar to the proof of Lemma~37 in~\cite{3erAlgo}.

We show that the algorithm preserves descents after every insertion of a pair $a,b$ in the sense that we consider the inserted numbers as a new total order.

In the first step we show that when we insert a pair $(a,b)$, $a$ and $b$ cause a descent except for the case that $a$ and $b$ are neighbors in the order of already inserted numbers. To see why we want this to hold, we consider the partial standard Young tableau consisting only of already inserted numbers. The number smaller than $a$ needs to be in a row below $a$ as numbers in the same row to the right are larger. The same holds for $b$ except if $a$ and $b$ are neighbors in the current order.

In the case that $a$ and $b$ are neighbors, they are both inserted as $-1$'s and we have to show that only $a$ is a descent. Thus everything else is analogously to the general case.

As $a$ or $b$ are inserted with $-1$ the only way that this causes no descent is that the position to the left is a $-1$ or a $1$ on level zero. The latter is not possible as this $1$ would have been on level $-1$ before. A $-1$ directly to the left of $a$ or $b$ would change either into a $0$, a $1$ or a $-2$, depending on $i$.  All cases cause a descent.

In the second step we show that we do not lose descents when inserting a pair $(a,b)$. If an entry was a descent in the partial tableau before inserting $(a,b)$ it is still one in the new partial tableau, either with the same number above, or with $a$ or $b$. In the former case  neither $a$ nor $b$ are inserted between those. In the latter case either $a$ or $b$ is inserted in between. This creates a descent in the vacillating tableau and removes the other descent as $(-1,x)$ can never be a descent.

Inserting a $-l$, always creates a new descent, when ignoring positions with smaller absolute values. The only such value that is not a descent left of a $-l$ is $-l$. However a $-l$ left of an inserted $-l$ is changed into a $-l-1$ while inserting. (Separation points are ignored if $c=b$, however they are adjusted before, if $b$ would be inserted between those, thus changed into a $-l-1$ too.)

It follows, that the position left of our new $-l$ is a descent if and only if it was a descent before changing it into $-l$.

In the third step we consider \emph{connect} and \emph{separate} as well as \emph{height violation} and \emph{adjust separation points}:

We show that \emph{separate} and \emph{connect} neither produce nor cancel a descent. For \emph{$i$ even connect} this is clear as $(j,-j)$ on $l$-level zero is not a descent. For \emph{$i$ odd separate} we consider a $0$ left or right of a position that was changed in \emph{separate}. Those need to be either $\tilde{a}$ or $\tilde{b}$ or they were changed in \emph{connect}, because otherwise they would have been \emph{separated} also. In the former case we want a descent, in the latter too, as $(j,0)$ has changed into $(0,-j)$ or the other way around. The same holds for $j$'s or $-j$'s of \emph{connect}.

At \emph{height violation} we change an $l$ that was a \emph{height violation} to an $(l+1)$. If $l$ was a descent, and $(l+1)$ is none, there needs to be a $(l+1)$ to the right, however this cannot happen, as then the \emph{height violation} would have started earlier to the right. If $l$ was no descent, then $(l+1)$ is none too. In our special cases we undo some change we have just done before, thus we do not change any descents.

For \emph{separation points} we point out that a descent left of a $\pm l$ needs to be a $\pm (l+1)$, thus when \emph{adjusting} them they get either $\pm (l+1)$ and $\pm (l+2)$ or $\pm (j-1)$ and $0$, both preserves the descent. 
\end{proof}

\begin{theorem}
\label{theo:10-1}
Let $Q$ be a standard Young tableau with rows of even length and $V$ be its corresponding vacillating tableau determined by Algorithms~\ref{alg:2} and~\ref{alg:u2}. If and only if for all rows $i=1,2,\dots,2k+1$ of $Q$ the first position in row $i$ is $i$, the first $k$ steps of $V$ are $1,2,\dots,k,0,-k,-k+1,\dots,-1$.
\end{theorem}

\begin{proof}
This holds as Algorithm~\ref{alg:2} is descent preserving and Algorithms~\ref{alg:2} and~\ref{alg:u2} are inverse.
\end{proof}

\subsection{Cut-away-shapes and $\mu$-horizontal strips}

In this subsection we will define a pattern on vacillating tableau, namely \enquote{cut-away-shapes}, and an equivalent pattern on standard Young tableaux, namely \enquote{$\mu$-horizontal strip}. We will see that these are mapped to each other in Algorithms~\ref{alg:2} and~\ref{alg:u2}. The definition of the latter is strongly related to alternative Littlewood-Richardson tableaux.

\begin{definition}
A vacillating tableau of shape $\emptyset$ has \emph{cut-away-shape} $\mu=(\mu_1,\mu_2, \dots, \mu_l)$ if it ends with 

\begin{tabular}{c cc cc c cc c}
 $(\underbrace{-l,-l,\dots, -l,}$ &$\underbrace{-l+1,-l+1,\dots, -l+1,}$ &$\dots $ &  $\underbrace{-2,-2,\dots, -2,}$ &$\underbrace{-1,-1,\dots, -1})$ \\
 $\mu_l$ & $\mu_{l-1}$ & $\dots$ & $\mu_2$ & $\mu_1$ 
\end{tabular}.
\end{definition}

Therefore, if we delete \enquote{cut away} the last $|\mu|$ positions the vacillating tableau has shape $\mu$.

\begin{example} The following vacillating tableau has cut-away-shape $\mu=(\color{blue} 3 \color{black}, \color{violet} 2 \color{black}, \color{red} 1\color{black})$:
 
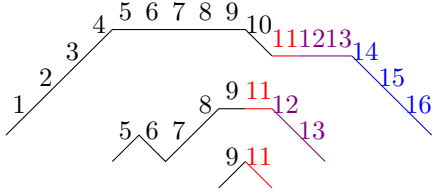
\begin{tikzpicture}[scale=0.35]
\draw (0,0)  --
(1,1) node[midway, above] {1}--
(2,2) node[midway, above] {2}--
(3,3) node[midway, above] {3}--
(4,4) node[midway, above] {4}--
(5,4) node[midway, above] {5}--
(6,4) node[midway, above] {6}--
(7,4) node[midway, above] {7}--
(8,4) node[midway, above] {8}--
(9,4) node[midway, above] {9}--
(10,3) node[midway, above] {10};
\draw(10,3)[red]--
(11,3) node[midway, above] {11};
\draw(11,3)[violet]--
(12,3) node[midway, above] {12}--
(13,3) node[midway, above] {13};
\draw(13,3)[blue]--
(14,2) node[midway, above] {14}--
(15,1) node[midway, above] {15}--
(16,0) node[midway, above] {16};
\draw (4,-1)  --
(5,0) node[midway, above] {5}--
(6,-1) node[midway, above] {6}--
(7,0) node[midway, above] {7}--
(8,1) node[midway, above] {8}--
(9,1) node[midway, above] {9};
\draw(9,1)[red]--
(10,1) node[midway, above] {11};
\draw(10,1)[violet]--
(11,0) node[midway, above] {12}--
(12,-1) node[midway, above] {13};
\draw (8,-2)--
(9,-1) node[midway, above] {9};
\draw(9,-1)[red]--
(10,-2) node[midway, above] {11};
\end{tikzpicture}
\end{example}

\begin{remark}
If a tableau has cut-away-shape $\mu=(\mu_1,\mu_2,\dots,\mu_l)$ it has also cut-away-shape $\tilde{\mu}$ where $\tilde{\mu}\subseteq\mu$ are subpartitions of the form $\tilde{\mu}=(\mu_1,\mu_2,\dots,\mu_{m},\mu_{m+1}-u)$ for every $0\leq m<l$ and $0\leq u <\mu_{m+1}$. 
\end{remark}

\begin{definition}
\label{def:MuHorizontalStrip}
Let $\mu$ be a partition with $\ell(\mu)\leq k$. Let $Q$ be a standard Young tableau with $2k+1$, possibly empty, rows, whose lengths have all the same parity, and $r$ entries. A \emph{$\mu$-horizontal strip} is a pattern of the last $|\mu|$ numbers in the following way:
\begin{enumerate}
\item \label{p:MuHoriHorizontalStrip} For each $j$, the numbers $r-(\mu_1+\mu_2+\dots+\mu_{j-1})-\mu_j+1$ up to $r-(\mu_1+\mu_2+\dots+\mu_{j-1})$ form a horizontal strip filled increasingly from left to right.

By abuse of notation we say that those numbers are in $\mu_j$.

\item \label{p:MuHoriUpperNeighbour} The $i$th number in $\mu_j$ is in a row below the $i$th number of $\mu_{j+1}$ if the latter exists.

\item \label{p:MuHoriMinRow} Go through the elements of $Q$ belonging to the $|\mu|$ last numbers from top to bottom, from right to left. Let $e$ be the current element of the $\mu$-horizontal strip. We define a sequence $v_e$ of elements of the $\mu$-horizontal strip. Let $e$ be the first entry of $v_e$. If $m-1$ entries of $v_e$ are defined, let $f$ be entry number $m-1$. We search now for entry number $m$. For that we consider entries whose that are smaller than $f$ and which are in exactly $m-1$ sequences defined before $v_e$. If this set is nonempty, take the largest entry as entry $m$. If it is empty, $v_e$ has no more entries.

Let $r_e$ be the row $p$ is in.
Now we define the value $o_e$ to be the number of entries in $v_e$ with the following properties. It is the rightmost occurrence in their $\mu_j$ and if number $m$ in $v_e$, all $v_{\tilde{e}}$, where $\tilde{e}\neq e$ is in the same row as $e$, have at most $m-1$ entries.

We require $r_e\geq 2 |v_e| - o_e$.

\end{enumerate}
\end{definition}

\begin{proposition}
\label{prop:MuAndAeLRSame}
If and only if the $|\mu|$ largest elements in a standard Young tableau $Q$ form a $\mu$-horizontal strip, the reverse skew semistandard tableau we obtain by deleting smaller elements and replacing elements in $\mu_j$ by $j$ is an alternative orthogonal Littlewood-Richardson tableau.  
\end{proposition}

\begin{proof}
This follows directly from the definitions (Definition~\ref{def:aoLRT} and Definition~\ref{def:MuHorizontalStrip}). The main difference in the definitions is that in the $\mu$-horizontal strip we only require in the third point of defining $v$ that it is the largest one, and not that it is the rightmost occurrence. Since entries in $\mu_j$ are increasing, this is still equivalent.
\end{proof}

\begin{example}
We consider the following tableaux (the first and the last one are corresponding tableaux to those in Example~\ref{ex:aoLRT}):

\hspace{0.4cm}\begin{tikzpicture}[scale=0.35]
  \draw (0,7) -- (4,7);
  \draw (0,6) -- (4,6);
  \draw (0,5) -- (4,5);
  \draw (0,4) -- (4,4);
  \draw (0,3) -- (2,3);
  \draw (0,2) -- (2,2);
  \draw (0,7) -- (0,2);
  \draw (1,7) -- (1,2);
  \draw (2,7) -- (2,2);
  \draw (3,7) -- (3,4);
  \draw (4,7) -- (4,4);
  \draw (0.5,6.5) node {1};
  \draw (1.5,6.5) node {2};
  \draw (2.5,6.5) node {3};
  \draw (3.5,6.5) node {4};
  \draw (0.5,5.5) node {5};
  \draw (1.5,5.5) node {6};
  \draw (2.5,5.5) node {7};
  \draw (3.5,5.5) node {8};
  \draw (0.5,4.5) node {9};
  \draw (0.5,3.5) node {10};
  \draw[red] (1.5,4.5) node {11};
  \draw[violet] (1.5,3.5) node {12};
  \draw[violet] (2.5,4.5) node {13};
  \draw[blue] (0.5,2.5) node {14};
  \draw[blue] (1.5,2.5) node {15};
  \draw[blue] (3.5,4.5) node {16}; 
      \draw (0,1);
\end{tikzpicture}
\hspace{0.8cm}
\begin{tikzpicture}[scale=0.35]
  \draw (0,7) -- (2,7);
  \draw (0,6) -- (2,6);
  \draw (0,5) -- (2,5);
  \draw (0,4) -- (2,4);
  \draw (0,7) -- (0,4);
  \draw (1,7) -- (1,4);
  \draw (2,7) -- (2,4);
  \draw (0.5,6.5) node {1};
  \draw (1.5,6.5) node {2};
  \draw (0.5,5.5) node {3};
  \draw[violet] (1.5,5.5) node {4};
  \draw[blue] (0.5,4.5) node {5};
  \draw[blue] (1.5,4.5) node {6};
    \draw (0,1);
\end{tikzpicture}
\hspace{0.8cm}
\begin{tikzpicture}[scale=0.35]
  \draw (0,7) -- (4,7);
  \draw (0,6) -- (4,6);
  \draw (0,5) -- (4,5);
  \draw (0,4) -- (2,4);
  \draw (0,7) -- (0,4);
  \draw (1,7) -- (1,4);
  \draw (2,7) -- (2,4);
  \draw (3,7) -- (3,5);
  \draw (4,7) -- (4,5);
  \draw (0.5,6.5) node {1};
  \draw (1.5,6.5) node {2};
  \draw (0.5,5.5) node {3};
  \draw (1.5,5.5) node {4};
  \draw (2.5,6.5) node {5};
  \draw (3.5,6.5) node {6};
  \draw[violet] (2.5,5.5) node {7};
  \draw[blue] (0.5,4.5) node {8};
  \draw[blue] (1.5,4.5) node {9};
  \draw[blue] (3.5,5.5) node {10};
    \draw (0,1);
\end{tikzpicture}
\hspace{0.8cm}
\begin{tikzpicture}[scale=0.35]
  \draw (0,7) -- (4,7);
  \draw (0,6) -- (4,6);
  \draw (0,5) -- (4,5);
  \draw (0,4) -- (4,4);
  \draw (0,3) -- (2,3);
  \draw (0,2) -- (2,2);
  \draw (0,1) -- (2,1);
  \draw (0,7) -- (0,1);
  \draw (1,7) -- (1,1);
  \draw (2,7) -- (2,1);
  \draw (3,7) -- (3,4);
  \draw (4,7) -- (4,4);
  \draw (0.5,6.5) node {1};
  \draw (1.5,6.5) node {2};
  \draw (2.5,6.5) node {3};
  \draw (3.5,6.5) node {4};
  \draw (0.5,5.5) node {5};
  \draw (1.5,5.5) node {6};
  \draw (0.5,4.5) node {7};
  \draw (1.5,4.5) node {8};
  \draw (0.5,3.5) node {9};
  \draw (1.5,3.5) node {10};
  \draw (2.5,5.5) node {11};
  \draw[red] (3.5,5.5) node {12};
  \draw[violet] (0.5,2.5) node {13};
  \draw[violet] (1.5,2.5) node {14};
  \draw[blue] (0.5,1.5) node {15};
  \draw[blue] (1.5,1.5) node {16};
  \draw[blue] (2.5,4.5) node {17};
  \draw[blue] (3.5,4.5) node {18};
\end{tikzpicture}
\hspace{0.8cm}
\begin{tikzpicture}[scale=0.35]
  \draw (0,7) -- (8,7);
  \draw (0,6) -- (8,6);
  \draw (0,5) -- (8,5);
  \draw (0,4) -- (8,4);
  \draw (0,3) -- (8,3);
  \draw (0,2) -- (6,2);
  \draw (0,1) -- (4,1);
  \draw (0,7) -- (0,1);
  \draw (1,7) -- (1,1);
  \draw (2,7) -- (2,1);
  \draw (3,7) -- (3,1);
  \draw (4,7) -- (4,1);
  \draw (5,7) -- (5,2);
  \draw (6,7) -- (6,2);
  \draw (7,7) -- (7,3);
  \draw (8,7) -- (8,3);
  \draw (0.5,6.5) node {1};
  \draw (1.5,6.5) node {2};
  \draw (2.5,6.5) node {3};
  \draw (3.5,6.5) node {4};
  \draw (0.5,5.5) node {5};
  \draw (1.5,5.5) node {6};
  \draw (2.5,5.5) node {7};
  \draw (3.5,5.5) node {8};
  \draw (0.5,4.5) node {9};
  \draw (1.5,4.5) node {10};
  \draw (2.5,4.5) node {11};
  \draw (3.5,4.5) node {12};
  \draw (0.5,3.5) node {13};
  \draw (1.5,3.5) node {14};
  \draw (2.5,3.5) node {15};
  \draw (3.5,3.5) node {16};
  \draw (0.5,2.5) node {17};
  \draw (1.5,2.5) node {18};
  \draw (2.5,2.5) node {19};
  \draw (3.5,2.5) node {20};
  \draw (4.5,6.5) node {21};
  \draw (5.5,6.5) node {22};
  \draw (6.5,6.5) node {23};
  \draw (7.5,6.5) node {24};
  \draw (4.5,5.5) node {25};
  \draw (5.5,5.5) node {26};
  \draw (6.5,5.5) node {27};
  \draw (7.5,5.5) node {28};
  \draw (4.5,4.5) node {29};
  \draw (5.5,4.5) node {30};
  \draw[red] (4.5,3.5) node {31};
  \draw[red] (6.5,4.5) node {32};
  \draw[red] (7.5,4.5) node {33};
  \draw[violet] (4.5,2.5) node {34};
  \draw[violet] (5.5,3.5) node {35};
  \draw[violet] (6.5,3.5) node {36};
  \draw[violet] (7.5,3.5) node {37};
  \draw[blue] (0.5,1.5) node {38};
  \draw[blue] (1.5,1.5) node {39};
  \draw[blue] (2.5,1.5) node {40};
  \draw[blue] (3.5,1.5) node {41};
  \draw[blue] (5.5,2.5) node {42};
      \draw (0,1);
\end{tikzpicture}

The first tableau contains a $\mu=(3,2,1)$-horizontal strip (as well as a $(3,2)$-, $(3,1)$-, $(3)$-, $(2)$-, $(1)$-, and $\emptyset$-horizontal strip). It is the corresponding standard Young tableau to the vacillating tableau in the previous example. The $v$'s are: $(16)$, $(13)$, $(11)$, $(12,11)$, $(15,13,11)$, $(14,12)$. (Compare with Example~\ref{ex:aoLRT}.)

The second tableau contains a $(2,1)$-horizontal strip but not a $(2,2)$, $(2,2,1)$ or $(2,2,2)$ one due to the third condition. The $v$'s are: $(4)$, $(6,4)$, $(5)$.

The third tableau contains a $(3,1)$-horizontal strip. The $v$'s are: $(10)$, $(7)$, $(9,7)$, $(8)$.

The fourth tableau contains a $(4,2,1)$-horizontal strip but not a $(4,2,2)$-horizontal strip due to the third condition. The $v$'s are: $(12)$, $(18,12)$, $(17)$, $(14)$, $(13)$, $(16,14,12)$, $(15,13)$.

The last (fifth) tableau contains a $\mu=(5,4,3)$-horizontal strip. The $v$'s are: $(33)$, $(32)$; $(37,33)$, $(36,32)$, $(35)$, $(31)$; $(42,37,33)$, $(34,31)$; $(41,36,32)$, $(40,35,31)$, $(39,34)$, $(38)$.\end{example}

Before we prove that $\mu$-horizontal strips are equivalent to cut-away-shapes, we state some facts about Algorithm~\ref{alg:2} we will need later on. These follow directly of the formulation of the algorithm. 
We see that everything happens right of the rightmost up-step that is not part of the right part of a \emph{separation point}. Therefore \emph{height violations} do not play a role here.

\begin{proposition}
For the $|\mu|$ largest positions in $Q$ it holds that:
\begin{itemize}
\item A $-l$ gets a $-(l+1)$ if and only if it is chosen as some $c_{l+1}$ for $l<j$.
\item $-j$'s chosen get $(j+1)$ or $-(j+1)$ when chosen as $c_{j+1}$ in insert row $2j+1$. (They get $0$'s first, and are initialized later.)
\item $-j$'s chosen get either $0$ or $j$ when chosen as $c_{j+1}$ in insert row $2j$. Only if there is only one position right of them they become a $0$ still in our considered part of the path.
\item An $l$ can only get a negative entry if it is part of a \emph{separation point}.
\end{itemize}
\end{proposition}

\begin{lemma}
\label{lem:MuhoriCutaway}
We consider an element $e$ in $\mu_i$ in row $r_e$ which gets inserted.
\begin{enumerate}
\item For $l\leq\lfloor r_e/2 \rfloor$ element number $l$ in $v_e$ is $e_l$.
\item If $|v_e|<\lfloor r_e/2 \rfloor$ $e_l$ with $l>|v_e|$ are left of the part of the labeled word we consider.
\item If $|v_e|>\lfloor r_e/2 \rfloor$ element number $l$ with $l>\lfloor r_e/2 \rfloor$ is part of a \emph{separation point} directly left of our down-steps. Each time nothing is changed to the right of it, the rightmost one of those gets a $-j$.
\end{enumerate}
\end{lemma}

\begin{proof}
We prove this inductively on the row $r_e$ an element is in.

For the base case we consider an element of the first row. This is the only one belonging to the $\mu$-horizontal strip and the last one of the first row. Thus it gets inserted as a $-1$. One could say that it was inserted as a $0$, thus part of a \emph{separation point}, but changes into a $-j=-1$ when \emph{initializing row $j=1$}. 

We show the induction step by another induction on the elements in $v_e$.
The base case is clear as $e$ gets inserted as $-1$.

Now we consider element $l$ in $v_e$. This is a $-(l-1)$ and was in $l-1$ $v$'s before. Moreover it is left of $e_{l-1}$. Every $-(l-1)$ that is between those, was in some other $v$ in the same row, or else it would have been taken instead. Thus this $-(l-1)$ is $e_l$ and gets changed into an $-l$.
Therefore the first property in question holds.

The second property holds, as once there is no element number $l$ in $v_e$ left, we know that there is no untouched $-(l-1)$ in our part of the path in question left, thus $e_l$ is more to the left.

The third property is more complicated. We point out, that elements, that are number $l$ in $v_e$ with $l>\lfloor r_e/2 \rfloor$, are counted by $o$. Thus they are the rightmost ones of their $\mu_m$. Due to the Yamanouchi property and Propositions~\ref{prop:2ndPropYamanouchi} and~\ref{prop:MuAndAeLRSame} we can argue that in those paths in which they are, there is no other position so far.

Another crucial point for the third property is, that once elements counted by $o$ occur, they also occur in the next row, if there is an element that is larger. The only way how they get less,  is when we correct our separation point, thus if a smaller element is considered or there is an empty row.

We now consider elements number $j+1$ up to $|v_e|$ during the insertion process of $e$.
\begin{itemize}
\item An element that is number $j+1$ in $v_e$ is a $-j$ and gets a $0$ that is the rightmost $0$. This is clear if $i$ is odd. If $i$ is even this follows as then element $j$ needs to be counted as $o$ as well and therefore it is the only element inserted to path $j$ in our area of question. In this case it becomes a $0$ on $j$-level $1$.

\item  The rightmost $0$ gets a $-(j+1)$ on level $0$ if $i$ is odd just before inserting the next row.

\item An element, that is number $j+2$ in $v_e$, is a $0$ before and a $j$ afterwards, if $i$ is odd due to \emph{separate odd}. If $i$ is even, it was and is a $(j-1)$.

\item An element that is number $j+m$ in $v_e$ is a $j-m+1$ (respectively $j-m+2$) if $i$ is even (respectively odd). 
\end{itemize}

Now if we insert a $c_l$ into the first path that contains such a $j-m+1$ (respectively $j-m+2$), there are two possible cases. In the first case, $c_l$ is inserted right of the corresponding $-j+m-1$ (respectively $j-m+2$). In this case $c_l$ is larger, and $v_c$ contains all elements our $-j+m-1$ (respectively $j-m+2$) had in its $v$ as well. Thus those are all counted by $o$. We do not adjust the separation point and the procedure goes on. In the second case, $c_l$ is inserted to the left. Therefore we \emph{adjust the separation point} and the $j-m+1$ (respectively $j-m+2$) becomes a $j+1-m+1$ (respectively $j+1-m+2$). In the same step either a $j-1$ becomes a $j$ or a $j-1$  becomes a $0$ (and thus later on a $-(j+1)$) depending on the parity of $i$. The same happens if for a row there is nothing inserted in the area of question.
\end{proof}

\begin{lemma}
\label{lem:muhorizontal1}
A standard Young tableau $Q$ containing a $\mu$-horizontal step is mapped to a vacillating tableau of cut-away-shape $\mu$ by Algorithm~\ref{alg:2}.
\end{lemma}

\begin{proof}
Lemma~\ref{lem:MuhoriCutaway} tells us that if an element is in $j$ different $v_e$'s, it ends up as a $-j$. As elements in $\mu_j$ are in exactly $j$ different $v_e$'s (compare with Propositions~\ref{prop:2ndPropYamanouchi} and~\ref{prop:MuAndAeLRSame}), we get cut away-shape $\mu$.
\end{proof}

\begin{lemma}
\label{lem:muhorizontal2}
If a vacillating tableau has cut-away-shape $\mu$, it is mapped by Algorithm~\ref{alg:u2} to a standard Young tableau containing a $\mu$-horizontal strip.
\end{lemma}

\begin{proof}
Let $V$ be a vacillating tableau with cut-away-shape $\mu$. Let $Q$ be its corresponding standard Young tableau and let $\tilde{\mu}$ be the largest partition such that $Q$ contains a $\tilde{\mu}$-horizontal strip. Now by Lemma~\ref{lem:muhorizontal1}, $V$ also contains a $\tilde{\mu}$-horizontal strip. If $\tilde{\mu} \supseteq \mu$ we are done. If $\tilde{\mu}\subsetneq \mu$ we show that we get a contradiction.

In this case let $p$ be the largest position in $Q$ that is not in the $\tilde{\mu}$-horizontal strip. We add it to the $\tilde{\mu}$-horizontal strip such that $\tilde{\mu}\subseteq\mu$. Now we know that this does not satisfy one of the three conditions. Therefore we distinguish cases.

\begin{enumerate}
\item If the last $\tilde{\mu}_j$ is not a horizontal strip, then $p$ is a descent, which gives a contradiction as Algorithms~\ref{alg:2} and~\ref{alg:u2} are descent preserving and $p$ is not a descent in $V$.
\item If the word does not satisfy the second condition, the reversed reading word of the according alternative orthogonal Littlewood-Richarson tableau is not Yamanouchi. This gives a contradiction to Propositions~\ref{prop:2ndPropYamanouchi} and~\ref{prop:MuAndAeLRSame} and Lemma~\ref{lem:MuhoriCutaway}.
\item If the inequality of the third property is not satisfied there are two possible cases. 
\begin{itemize}
\item It could be that a $v$ got longer (this happens exactly if $p$ is in it). For it to be to long, $p$ needs to be at least number $j+1$. However we know, that $p+1$ was inserted at least as often. Therefore $p$ is inserted on level $2$. This is a contradiction to being part of a \emph{separation point} due to being number $(j+1)$, compare with Lemma~\ref{lem:MuhoriCutaway}.
\item Or it could be that a $\tilde{v}$ in the same row got longer (then $p$ is in this $\tilde{v}$). In this case again there needs to be at least one number $j+1$. The first path with a \emph{separation point} belonging to a position counted by $o$ gets also level $2$ positions, which is also a contradiction.
\qedhere 
\end{itemize}
\end{enumerate}
\end{proof}

Thus we have proven (by Lemma~\ref{lem:muhorizontal1} and~\ref{lem:muhorizontal2}) the following theorem:

\begin{theorem}
\label{theo:muhorizontal}
If and only if a standard Young tableau $Q$ contains a $\mu$-horizontal strip, the corresponding vacillating tableau has cut-away-shape $\mu$.
\end{theorem}

\subsection{Conjectures for Bijection B}

\begin{conjecture}
\label{con:ConCat}
Concatenation of standard Young tableaux, whose row lengths have all the same parity corresponds to concatenation of vacillating tableaux of shape $\emptyset$ in general.
\end{conjecture}

This is proven for $k=1$ in \cite{3erAlgo}, and for standard Young tableaux with even row length in Theorem~\ref{theo:Concat}.

\begin{conjecture}
Evacuation (Sch\"utzenberger involution) in a standard Young tableau corresponds to the reversal of the corresponding vacillating tableau.
\end{conjecture}

\section*{Acknowledgements}
The author would like to thank Martin Rubey and Stephan Pfannerer for valuable discussions and helpful comments.




\appendix
\section*{Appendix}

\begin{table}[!h]
\centering
\caption{List of all tableaux with $n=5$, $r=3$.
Note that there is not necessarily an orthogonal Littlewood-Richardson tableau for every combination of $\mu$ and $\lambda$ with $\mu\subseteq\lambda$.
}
\label{tab:ListOfExamples}
\begin{tabular}{|p{1.4cm}|p{1.4cm}|p{2.2cm}|p{1.8cm}|p{1.4cm}|p{1.8cm}|}
\hline
\\[-1em]
$\lambda$ & $\mu$ & $L$ & $\tilde{L}$ & $Q$ & $V$\\
\hline
\hline
\\[-1em]
$(1,1,1)$ & $(1,1)$ & 
\begin{tikzpicture}[scale=0.35]
  \draw (0,0)--(0,1);
  \draw (1,0)--(1,1);
  \draw (0,0)--(1,0);
  \draw (0,1)--(1,1);
  \draw (0.5,0.5) node {1};
  \draw (2.5,0)--(2.5,1);
  \draw (3.5,0)--(3.5,1);
  \draw (2.5,0)--(3.5,0);
  \draw (2.5,1)--(3.5,1);
  \draw (3,0.5) node {1};
  \draw (5,0)--(5,1);
  \draw (6,0)--(6,1);
  \draw (5,0)--(6,0);
  \draw (5,1)--(6,1);
  \draw (5.5,0.5) node {1};
\end{tikzpicture}
&
\begin{tikzpicture}[scale=0.35]
  \draw (0,0)--(0,5);
  \draw (1,0)--(1,5);
  \draw (0,0)--(1,0);
  \draw (0,1)--(1,1);
  \draw (0,2)--(1,2);
  \draw (0,3)--(1,3);
  \draw (0,4)--(1,4);
  \draw (0,5)--(1,5);
  \draw (0.5,1.5) node {2};
  \draw (0.5,0.5) node {1};
\end{tikzpicture}
&
\begin{tikzpicture}[scale=0.35]
  \draw (0,0)--(0,3);
  \draw (1,0)--(1,3);
  \draw (0,0)--(1,0);
  \draw (0,1)--(1,1);
  \draw (0,2)--(1,2);
  \draw (0,3)--(1,3);
  \draw (0.5,2.5) node {1};
  \draw (0.5,1.5) node {2};
  \draw (0.5,0.5) node {3};
\end{tikzpicture}
&
\begin{tikzpicture}[scale=0.35]
  \draw (0,0)--(1,1)--(2,1)--(3,1);
  \draw (1,-1)--(2,0)--(3,0);
\end{tikzpicture}
\\
\hline
\multirow{17}{*}{$(2,1)$} & 
\multirow{5}{*}{
$(1)$}&
\multirow{2}{*}{\begin{tikzpicture}[scale=0.35]
  \draw (0,0)--(0,1);
  \draw (1,0)--(1,1);
  \draw (0,0)--(1,0);
  \draw (0,1)--(1,1);
  \draw (0.5,0.5) node {1};
  \draw (5,1)--(5,3);
  \draw (6,1)--(6,3);
  \draw (5,1)--(6,1);
  \draw (5,2)--(6,2);
  \draw (5,3)--(6,3);
  \draw (5.5,1.5) node {2};
  \draw (5.5,2.5) node {1};
\end{tikzpicture}}
 & 
\multirow{4}{*}{\begin{tikzpicture}[scale=0.35]
  \draw (0,0)--(0,2);
  \draw (1,0)--(1,2);
  \draw (2,0)--(2,2);
  \draw (0,0)--(2,0);
  \draw (0,1)--(2,1);
  \draw (0,2)--(2,2);
  \draw (1.5,0.5) node {1};
\end{tikzpicture}}
 &
\begin{tikzpicture}[scale=0.35]
  \draw (0,0)--(0,2);
  \draw (1,0)--(1,2);
  \draw (2,1)--(2,2);
  \draw (0,0)--(1,0);
  \draw (0,1)--(2,1);
  \draw (0,2)--(2,2);
  \draw (0.5,1.5) node {1};
  \draw (1.5,1.5) node {2};
  \draw (0.5,0.5) node {3};
\end{tikzpicture}
&
\begin{tikzpicture}[scale=0.35]
  \draw (0,0)--(1,1)--(2,2)--(3,1);
\end{tikzpicture}
\\
\cline{5-6}
& & & & 
\begin{tikzpicture}[scale=0.35]
  \draw (0,0)--(0,2);
  \draw (1,0)--(1,2);
  \draw (2,1)--(2,2);
  \draw (0,0)--(1,0);
  \draw (0,1)--(2,1);
  \draw (0,2)--(2,2);
  \draw (0.5,1.5) node {1};
  \draw (1.5,1.5) node {3};
  \draw (0.5,0.5) node {2};
\end{tikzpicture}
&
\begin{tikzpicture}[scale=0.35]
  \draw (0,0)--(1,1)--(2,1)--(3,1);
    \draw (1,-1)--(2,0)--(3,-1);
\end{tikzpicture}
\\
\cline{2-6}
&
\multirow{7}{*}{$(2,1)$}
&
\multirow{6}{*}{\begin{tikzpicture}[scale=0.35]
  \draw (0,0)--(0,2);
  \draw (1,0)--(1,2);
  \draw (0,0)--(1,0);
  \draw (0,1)--(1,1);
  \draw (0,2)--(1,2);
  \draw (0.5,0.5) node {2};
  \draw (0.5,1.5) node {1};
  \draw (2.5,1)--(2.5,2);
  \draw (3.5,1)--(3.5,2);
  \draw (2.5,1)--(3.5,1);
  \draw (2.5,2)--(3.5,2);
  \draw (3,1.5) node {1};
\end{tikzpicture}
}
&  
\multirow{5}{*}{ \begin{tikzpicture}[scale=0.35]
  \draw (0,0)--(0,3);
  \draw (1,0)--(1,3);
  \draw (2,0)--(2,3);
  \draw (0,0)--(2,0);
  \draw (0,1)--(2,1);
  \draw (0,2)--(2,2);
  \draw (0,3)--(2,3);
  \draw (0.5,0.5) node {1};
  \draw (1.5,0.5) node {1};
  \draw (1.5,1.5) node {2};
\end{tikzpicture}}
&
\begin{tikzpicture}[scale=0.35]
  \draw (0,0)--(0,2);
  \draw (1,0)--(1,2);
  \draw (2,1)--(2,2);
  \draw (0,0)--(1,0);
  \draw (0,1)--(2,1);
  \draw (0,2)--(2,2);
  \draw (0.5,1.5) node {1};
  \draw (1.5,1.5) node {2};
  \draw (0.5,0.5) node {3};
\end{tikzpicture}
&
\begin{tikzpicture}[scale=0.35]
  \draw (0,0)--(1,1)--(2,2)--(3,2);
  \draw (2,-1)--(3,0);
  \draw (3,2) node{};
\end{tikzpicture}
\\
\cline{5-6}
& & & & 
\begin{tikzpicture}[scale=0.35]
  \draw (0,0)--(0,2);
  \draw (1,0)--(1,2);
  \draw (2,1)--(2,2);
  \draw (0,0)--(1,0);
  \draw (0,1)--(2,1);
  \draw (0,2)--(2,2);
  \draw (0.5,1.5) node {1};
  \draw (1.5,1.5) node {3};
  \draw (0.5,0.5) node {2};
\end{tikzpicture}
&
\begin{tikzpicture}[scale=0.35]
  \draw (0,0)--(1,1)--(2,1)--(3,2);
  \draw (1,-1)--(2,0);
  \draw (3,2) node{};
\end{tikzpicture}\\
\hline
\multirow{7}{*}{$(3)$} & $(1)$ &
\begin{tikzpicture}[scale=0.35]
  \draw (0,0)--(0,1);
  \draw (1,0)--(1,1);
  \draw (0,0)--(1,0);
  \draw (0,1)--(1,1);
  \draw (0.5,0.5) node {3};
  \draw (5,1)--(5,3);
  \draw (6,1)--(6,3);
  \draw (5,1)--(6,1);
  \draw (5,2)--(6,2);
  \draw (5,3)--(6,3);
  \draw (5.5,1.5) node {2};
  \draw (5.5,2.5) node {1};
\end{tikzpicture}
& 
\begin{tikzpicture}[scale=0.35]
  \draw (0,0)--(0,1);
  \draw (1,0)--(1,1);
  \draw (2,0)--(2,1);
  \draw (3,0)--(3,1);
  \draw (4,0)--(4,1);
  \draw (0,0)--(4,0);
  \draw (0,1)--(4,1);
  \draw (3.5,0.5) node {1};
\end{tikzpicture}
&
\begin{tikzpicture}[scale=0.35]
  \draw (0,0)--(0,1);
  \draw (1,0)--(1,1);
  \draw (2,0)--(2,1);
  \draw (3,0)--(3,1);
  \draw (0,0)--(3,0);
  \draw (0,1)--(3,1);
  \draw (0.5,0.5) node {1};
  \draw (1.5,0.5) node {2};
  \draw (2.5,0.5) node {3};
\end{tikzpicture}
&
\begin{tikzpicture}[scale=0.35]
  \draw (0,0)--(1,1)--(2,0)--(3,1);
\end{tikzpicture}\\
\cline{2-6}
& $(3)$ &
 \begin{tikzpicture}[scale=0.35]
  \draw (0,0)--(0,3);
  \draw (1,0)--(1,3);
  \draw (0,0)--(1,0);
  \draw (0,1)--(1,1);
  \draw (0,2)--(1,2);
  \draw (0,3)--(1,3);
  \draw (0.5,2.5) node {1};
  \draw (0.5,1.5) node {2};
  \draw (0.5,0.5) node {3};
\end{tikzpicture}
&
\begin{tikzpicture}[scale=0.35]
  \draw (0,0)--(0,2);
  \draw (1,0)--(1,2);
  \draw (2,0)--(2,2);
  \draw (3,1)--(3,2);
  \draw (4,1)--(4,2);
  \draw (0,0)--(2,0);
  \draw (0,1)--(4,1);
  \draw (0,2)--(4,2);
  \draw (0.5,0.5) node {1};
  \draw (1.5,0.5) node {1};
  \draw (3.5,1.5) node {1};
\end{tikzpicture}
& 
\begin{tikzpicture}[scale=0.35]
  \draw (0,0)--(0,1);
  \draw (1,0)--(1,1);
  \draw (2,0)--(2,1);
  \draw (3,0)--(3,1);
  \draw (0,0)--(3,0);
  \draw (0,1)--(3,1);
  \draw (0.5,0.5) node {1};
  \draw (1.5,0.5) node {2};
  \draw (2.5,0.5) node {3};
\end{tikzpicture}
&
\begin{tikzpicture}[scale=0.35]
  \draw (0,0)--(1,1)--(2,2)--(3,3);
\end{tikzpicture}\\
\hline
\end{tabular}
\end{table}

\end{document}